\documentclass[twoside,11pt]{article}

\usepackage[utf8]{inputenc}
\usepackage[english]{babel}
\usepackage{amsmath}
\usepackage{amsthm}
\usepackage{amsfonts}
\usepackage{amssymb}
\usepackage{graphicx}
\usepackage{colortbl,dcolumn}
\usepackage{paralist}  

\def\H{{\cal H}}
\def\L{\mathcal{L}}
\def\F{\textbf{F}}
\def\R{\mathbb{R}}

\def\H2{H^2(\R^N)}
\def\L2{L^2(\R^N)}
\def\to{\rightarrow}

\def\ET{\textbf{E}}
\def\n{\hat{n}_0}

       \newtheorem{lemma}{\bf Lemma}[section]
       \newtheorem{theorem}{\bf Theorem}[section]
       \newtheorem{proposition}{\bf Proposition}[section]
       
       \newtheorem{definition}{\bf Definition}[section]
       \newtheorem{remark}{\bf Remark}[section]

       \numberwithin{equation}{section}

\begin{document}

\title{{\LARGE Steady hydrodynamic model of
semiconductors \\
with sonic boundary} \footnotetext{\small
*Corresponding author.}
 \footnotetext{\small E-mail addresses: lijy645@nenu.edu.cn (J. Li), \ \  ming.mei@mcgill.ca (M. Mei), \\
 zhanggj100@nenu.edu.cn (G. Zhang), \ \
 zhangkj201@nenu.edu.cn (K. Zhang)} }

\author{{Jingyu Li$^1$, Ming Mei$^{2,3}$, Guojing Zhang$^1$$^\ast$  and  Kaijun Zhang$^1$}\\[2mm]
\small\it $^1$School of Mathematics and Statistics, Northeast Normal University,\\
\small\it   Changchun 130024, P.R.China \\
\small\it $^2$Department of Mathematics, Champlain College Saint-Lambert,\\
\small\it     Saint-Lambert, Quebec, J4P 3P2, Canada\\
\small\it $^3$Department of Mathematics and Statistics, McGill University,\\
\small\it     Montreal, Quebec, H3A 2K6, Canada }

\date{}

\maketitle

\begin{abstract}

In this paper, we study the well-posedness/ill-posedness and
regularity of stationary solutions to the hydrodynamic model of
semiconductors represented by Euler-Poisson equations with sonic
boundary, and make a classification on these solutions. When the
doping profile is subsonic, we prove that, the  corresponding
steady-state equations with sonic boundary possess a unique interior
subsonic solution, and at least one interior supersonic solution;
and if  the relaxation time is large and the
doping profile is a small perturbation of constant, then the equations admit infinitely  many interior transonic shock solutions;
 while,  if the relaxation time is small enough and the doping
 profile is a subsonic constant, then the equations admits infinitely  many interior
 $C^1$ smooth transonic solutions, and no transonic shock solution exists. When the
doping profile is supersonic, we show that the system does not hold
any subsonic solution; furthermore, the system
 doesn't admit
any supersonic solution or any transonic solution if such a
supersonic doping profile is small enough or the relaxation time
 is small, but it has at least one supersonic solution and infinitely
many transonic solutions if the supersonic doping profile is close
to the sonic line and the  relaxation time is large. The interior
subsonic/supersonic solutions all are globally $C^{\frac{1}{2}}$
H\"older-continuous, and the H\"older exponent $\frac{1}{2}$ is
optimal. The non-existence of any type solutions in the case of
small doping profile or small relaxation time indicates that the
semiconductor effect for the system is remarkable and cannot be
ignored.  The proof
 for the existence of subsonic/supersonic solutions  is the technical compactness analysis
combining   the energy method and the phase-plane analysis, while
the approach for the existence of multiple transonic solutions is
artfully constructed.  The results obtained significantly improve and develop
the existing studies.

\

\indent \textbf{Keywords}: Euler-Poisson equations, hydrodynamic
model of semiconductors, sonic boundary, subsonic solutions,
supersonic solutions, transonic solutions with shock, $C^1$ smooth
transonic solution.

\indent \textbf{AMS (2010) Subject Classification}: 35R35, 35Q35,
76N10, 35J70


\end{abstract}

\newpage

\tableofcontents


\baselineskip=18pt

\section{Introduction}

The hydrodynamic model of semiconductors, first introduced by
Bl{\o}tekj{\ae}r in \cite{B}, is usually described  for  the
charged fluid particles such as electrons and holes in semiconductor
devices \cite{B,Jungel,Markowich-Ringhofer-Schmeiser} and positively
and negatively charged ions in plasma \cite{S-M}. The governing
equations are  Euler-Poisson equations as follows
\cite{Guo,HMWY,HMWY2,LMM}:
\begin{equation} \label{hydrodynamic}
\left\{ \begin{aligned}
         &\rho_t+(\rho u)_x =0,\\
         &(\rho u)_t+(\rho u^2+P(\rho))_x=\rho E-\frac{\rho
         u}{\tau},\\
         & E_x=\rho-b(x).
\end{aligned} \right.
\end{equation}
Here $\rho$, $u$ and $E$ represent the electron density, the
velocity and the electric field, respectively. $P(\rho)$ is the
pressure function of the electron density. When the system is
isothermal, the pressure function is physically represented by
\begin{equation}\label{pressure}
P(\rho)=T\rho, \ \mbox{ with the constant temperature } T>0.
\end{equation}
The function $b(x)>0$ is the doping profile standing for the density
of impurities in semiconductor device. The constant $\tau>0$ denotes
the momentum relaxation time.

In this series of study, we are mainly interested in investigating
the existence of the solutions to \eqref{hydrodynamic} with sonic
boundary, and the large-time behavior of the solutions. At the first
but important stage, we focus on the existence and classification of
all stationary solutions to the steady-state system of equations
with sonic boundary. This will be the main purpose of the present paper.

 In this paper, we consider the following steady-state equations to \eqref{hydrodynamic} in the  bounded domain $[0,1]$.
Denote $J=\rho u$, the current density, then we have the stationary
equations of \eqref{hydrodynamic} as follows
\begin{equation} \label{stationary}
\left \{\begin{array}{ll}
        J = \text{constant},\\
        \left(\dfrac{J^2}{\rho}+P(\rho)\right)_x=\rho E-\dfrac{J}{\tau}, \qquad x\in (0,1).\\
         E_x=\rho-b(x).
        \end{array} \right.
\end{equation}
Using the terminology from gas dynamics, we call
$c:=\sqrt{P'(\rho)}=\sqrt{T}>0$ the sound speed for $P(\rho)=T\rho$
(see \eqref{pressure}). Thus, the stationary flow of
\eqref{stationary} is called to be subsonic/sonic/supersonic, if the
fluid velocity satisfies
\begin{equation}\label{pressure-2}
\mbox{fluid velocity: } u=\frac{J}{\rho} \lesseqqgtr  c=\sqrt{P'(\rho)}=\sqrt{T}: \mbox{ sound speed}. \\
\end{equation}
We consider the current driven flow, thus the current density $J$ is
a prescribed constant. Note that if $(\rho(x),E(x))$ is a solution
to \eqref{stationary} with a given constant current density $J$,
then $(\rho(1-x),-E(1-x))$ is a solution to \eqref{stationary} with
respect to $-J$ and $b(1-x)$. So, we may consider only the case of
$J>0$. Without loss of generality, let us assume throughout the
paper
\[
T=J=1.
\]
Thus, \eqref{stationary} is transformed to
\begin{equation} \label{1.5}
\left \{\begin{array}{ll}
        \left(1-\dfrac{1}{\rho^2}\right)\rho_x=\rho E-\dfrac{1}{\tau},\\
         E_x=\rho-b(x).
        \end{array} \right.
\end{equation}
From \eqref{pressure-2}, it can be identified that,  $\rho>1$ is for
the subsonic flow, $\rho=1$ stands for the sonic flow, and
$0<\rho<1$ represents for the supersonic flow. Therefore, our sonic
boundary conditions to \eqref{stationary} are proposed as follows
\begin{equation}\label{boundary}
\mbox{sonic boundary: } \ \ \rho(0)=\rho(1)=1.
\end{equation}
Dividing the first equation of \eqref{1.5} by $\rho$ and differentiating the resultant
equation with respect to $x$, and substituting the second equation of \eqref{1.5} to
this modified equation, then we have
\begin{equation} \label{elliptic}
\begin{cases}
\left[\left(\dfrac{1}{\rho}-\dfrac{1}{\rho^3}\right)\rho_x\right]_x
+\dfrac{1}{\tau}\left(\dfrac{1}{\rho}\right)_x- [\rho-b(x)]=0,\
\ x\in(0,1), \\
\rho(0)=\rho(1)=1 \ (\mbox{sonic boundary}).
\end{cases}
\end{equation}
When $\rho(x)>1$ or $0<\rho(x)<1$ for $x\in (0,1)$, the equation
\eqref{elliptic} is elliptic but degenerate at the sonic boundary.
When $\rho(x)>0$ varies around the sonic line $\rho=1$ for $x\in
(0,1)$, the system then changes its property and occurs phase
transitions. The degeneracy of \eqref{elliptic} on the boundary will
cause us some essential difficulty in the study of well-posedness
and regularity of the solutions, and the phenomena of structure of
solutions are really rich and interesting.

Throughout the paper we assume that the
doping profile $b(x)\in L^\infty(0,1)$  and denote
\[
\underline{b}:=\underset{x\in(0,1)}{\text{essinf }}b(x) \ \  \mbox{
and }  \ \ \overline{b}:=\underset{x\in(0,1)}{\text{esssup }}b(x).
 \]
Now we introduce the concepts of interior
subsonic/supersonic/transonic solutions.

\begin{definition}\label{definition-1}
$\rho(x)$  is called an interior subsonic (correspondingly, interior
supersonic)  solution of equation \eqref{elliptic}, if
$\rho(0)=\rho(1)=1$ but $\rho(x)\geq1$ (correspondingly,
$0<\rho(x)\leq1$) for $x\in (0,1)$, and $(\rho(x)-1)^2\in
H_0^1(0,1)$, and it holds that for any $\varphi\in H_0^1(0,1)$
\[
\int_0^1\Big(\frac{1}{\rho}-\frac{1}{\rho^3}\Big)\rho_x\varphi_x dx
+\frac{1}{\tau}\int_0^1\frac{\varphi_x}{\rho}dx+ \int_0^1(\rho-b)\varphi
dx=0,
\]
which is equivalent to
\begin{equation} \label{weak-solution}
\frac{1}{2}\int_0^1\frac{\rho+1}{\rho^3}\left((\rho-1)^2\right)_x\varphi_xdx
+\frac{1}{\tau}\int_0^1\frac{\varphi_x}{\rho}dx+ \int_0^1(\rho-b)\varphi
dx=0.
\end{equation}

\end{definition}
Once $\rho=\rho(x)$ is determined by equation \eqref{elliptic}, in
view of the first equation of \eqref{1.5}, the electric field $E(x)$ can be solved by
\begin{equation*}
E(x)=\left(\frac{1}{\rho}-\frac{1}{\rho^3}\right)\rho_x+\frac{1}{\tau\rho}
=\frac{(\rho+1)[(\rho-1)^2]_x}{2\rho^3}+\frac{1}{\tau\rho}.
\end{equation*}
In this way, we could obtain the interior subsonic/supersonic solutions to system \eqref{1.5}-\eqref{boundary}.

\begin{definition}\label{definition-2}
$\rho(x)>0$ is called  a $C^1$ transonic  solution of system
\eqref{1.5}-\eqref{boundary}, if $\rho(x)\in C^1(0,1)$ with $\rho(0)=\rho(1)=1$ and there exists a number $x_0\in (0,1)$ such that
\begin{equation*}
\rho(x)=\left \{\begin{array}{ll}
        \rho_{sup}(x), \ x\in(0,x_0),\\
        \rho_{sub}(x), \ x\in(x_0,1),
        \end{array} \right.
\end{equation*}
where $0<\rho_{sup}(x)\le 1$, $\rho_{sub}(x)\ge 1$ and
\begin{equation}\label{c1-transonic}
\rho_{sup}(x_0)=\rho_{sub}(x_0) \ \mbox{ and }
\ \rho_{sup}'(x_0)=\rho_{sub}'(x_0).
\end{equation}

$\rho(x)>0$ is called  a transonic shock solution of system
\eqref{1.5}-\eqref{boundary}, if $\rho(0)=\rho(1)=1$ and it is separated at a point
$x_0\in(0,1)$ in the form
\begin{equation*}
\rho(x)=\left \{\begin{array}{ll}
        \rho_{sup}(x), \ x\in(0,x_0),\\
        \rho_{sub}(x), \ x\in(x_0,1),
        \end{array} \right.
\end{equation*}
where $0<\rho_{sup}(x)\le 1$ and $\rho_{sub}(x)\ge 1$ satisfies the
entropy condition  at $x_0$
\begin{equation}\label{entropy}
0<\rho_{sup}(x_0^-)<1<\rho_{sub}(x_0^+),
\end{equation}
 and the Rankine-Hugoniot condition
\begin{equation}\label{RH}
\begin{split}
\rho_{sup}(x_0^-)+\frac{1}{\rho_{sup}(x_0^-)}&=\rho_{sub}(x_0^+)+\frac{1}{\rho_{sub}(x_0^+)},\\
E_{sup}(x_0^-)&=E_{sub}(x_0^+).
\end{split}
\end{equation}
\end{definition}
Set $\rho_l=\rho_{sup}(x_0^-)$ and $\rho_r=\rho_{sub}(x_0^+)$, a
simple computation from \eqref{RH} shows that
\begin{equation}\label{2.24}
\rho_l\rho_r=1.
\end{equation}

The existence of subsonic/supersonic/transonic solutions to the
steady-state Euler-Poisson equations for the hydrodynamic model of
semiconductors has been intensively studied. In 1990, Degond and
Markowich \cite{Degond-Markowich} first showed the existence of
subsonic solution when  the flow and its boundary are completely
subsonic. The uniqueness was obtained with a very strong subsonic
background, namely, $|J|\ll 1$. Then, the steady subsonic flows were
deeply studied with different boundaries as well as the higher
dimensions case in
\cite{Bae,Bae-Duan-Xie,Degond-Markowich2,Fang-Ito,Guo,Jerome,Nishibata-Suzuki},
see the references and therein. For the case of steady supersonic
flows, Peng and Violet \cite{Peng-Violet} obtained the existence and
uniqueness  of supersonic solution when the boundary is with a
strongly supersonic background (i.e. $J\gg 1$). On the other hand,
the case of steady transonic flows has been also paid a lot of
attention.  By a phase-plane analysis, Ascher {\it et al}
\cite{AMPS} first tested the existence of transonic solution when
the boundary is subsonic but the   constant background charge $b(x)$ is
supersonic, which was then extended by Rosini \cite{Rosini} for a
bit general case.  On the other hand,  by using the vanishing
viscosity limit method, Gamba  constructed  1-D transonic solutions
with transonic shocks in \cite{Gamba}, and  2-D transonic solutions
in \cite{Gamba2}, but the solutions as the limits of vanishing
viscosity yield boundary layers. Recently, Luo-Xin \cite{Luo-Xin}
and Luo-Rauch-Xie-Xin \cite{Luo-Rauch-Xie-Xin} studied  the
hydrodynamic model \eqref{hydrodynamic} of Euler-Poisson equations
without the effect of semiconductor, namely, the momentum equation
\eqref{hydrodynamic}$_2$ is missing the term of $-\frac{J}{\tau}$.
This means  the current density $J=0$ (the absence of semiconductor
effect for the device), or the relaxation time $\tau=\infty$ (the
huge relaxation time). Some interesting results on the
structure of steady solutions with non-sonic boundary are obtained.
Precisely, based on
 phase-plane analysis, Luo-Xin
\cite{Luo-Xin} thoroughly studied the
existence/non-existence, uniqueness/non-uniqueness of the transonic
solutions with one side supersonic boundary and the other side
subsonic boundary when the doping profile $b(x)$ is a constant
either in the supersonic regime or the subsonic regime. Some
restrictions on the boundary and the domain are also needed. Then,
Luo-Rauch-Xie-Xin \cite{Luo-Rauch-Xie-Xin}  showed the existence of
the transonic solution to the variable doping profile $b(x)$ which
is regarded as a small perturbation of the constant doping profile
$b(x)\equiv b$, and further proved  the time-asymptotic stability of
the transonic shock profiles.

In this paper, the model considered is with the semiconductor
effect, and the boundary is, in particular,  sonic. These features
make the study more difficult and different from the existing
studies. In fact, the elliptic equation \eqref{elliptic} is
degenerate at the boundary, but the equations considered in the
previous studies  are uniformly elliptic whatever in the supersonic
regime \cite{Peng-Violet} or the subsonic regime
\cite{Degond-Markowich}. On the other hand, when the doping profile
is sonic or supersonic, we realize that there is
 no any physical solution if the doping profile is
small or the relaxation time is small, which is totally different
from the studies \cite{Luo-Xin,Luo-Rauch-Xie-Xin} in the case
without the semiconductor effect. In fact, this demonstrates that
the semiconductor effect is remarkable and cannot be ignored.

The main purpose in this paper is to prove the
well-posedness/ill-posedness of the steady Euler-Poisson equations
\eqref{1.5} with the sonic boundary \eqref{boundary},
 and the regularity of subsonic and supersonic
solutions when these solutions exist, and the property of the
infinitely many transonic solutions. Precisely speaking, when the
doping profile is subsonic, we prove that, the  corresponding
steady-state equations with sonic boundary possess a unique interior
subsonic solution, and at least one interior supersonic solution,
and infinitely  many interior transonic shock solutions for
$\tau\gg1$ (the small effect of semiconductor), where the case
$\tau=\infty$ studied in \cite{Luo-Xin,Luo-Rauch-Xie-Xin} is our
special case; while, when $\tau\ll 1$ (the large effect of
semiconductor), the system possesses infinitely many $C^1$ transonic
solutions, and no transonic shocks exist. Note that, the transonic
shocks have been intensively studied in
\cite{AMPS,Gamba,Gamba2,Luo-Xin,Luo-Rauch-Xie-Xin}, but, to our best
knowledge, the $C^1$  transonic solutions in
semiconductor models are first obtained in the present paper. Essentially,
the strong damping effect makes the transonic solutions to be $C^1$
smooth. Recall that $C^2$  transonic flow also arises in
finite de Laval nozzles, where the geometry structure of the nozzle
causes the transonic flow to be $C^2$ smooth (see the interesting
work of C. Wang and Z. Xin \cite{Wang-Xin-smooth,Wang-Xin-smooth2}).  On the other
hand, when the doping profile is supersonic, we show that the system
does not hold any subsonic solution; and the system also has
 no supersonic solution and no transonic solution if such a
supersonic doping profile is small enough or the relaxation time is
large; but it possesses at least one supersonic solution and
infinitely many transonic shock solutions if the supersonic doping
profile is close to the sonic line and the  semiconductor effect is
small. When the doping profile is sonic, then the system exists the
sonic solution. In all cases mentioned above, all interior
subsonoic/supersonic solutions obtained are proved to be globally
$C^{\frac{1}{2}}$ H\"older continuous, and the
$C^{\frac{1}{2}}$-regularity  is optimal. We notice that the same
regularity $C^{\frac{1}{2}}$ was also obtained for the
subsonic-sonic flow for the steady nozzle in \cite{Wang,Wang-Xin}.
Regarding the other interesting studies on the subsonic-sonic flow
for the steady nozzle, we refer to
\cite{Bers,Chen-Dafermos-Slemrod-Wang,Chen-Huang-Wang,Xie-Xin}. To
prove the existence of the subsonic/supersonic/transonic solutions
and their regularity are non-trivial, because the degeneracy of
ellipticity for equation \eqref{elliptic} at the sonic boundary
causes us essential difficulty. Here, for the existence of
subsonic/supersonic solutions to equations \eqref{1.5} and \eqref{boundary}, the
proof adopted is the technical compactness analysis combining  the
energy method with the help of the phase-plane analysis, while for
the existence of multiple transonic shock solutions and $C^1$-smooth
transonic solutions, the approach is the artful construction method.
These results are presented in the following Theorems
\ref{main-thm-1}-\ref{main-thm-3}, which  essentially improve and
develop the existing studies.

\begin{theorem}[Case of subsonic doping profile]\label{main-thm-1}
Let the doping profile be subsonic such that $b(x)\in L^\infty(0,1)$
and $\underline{b}>1$. Then the steady-state Euler-Poisson equations \eqref{1.5} and
\eqref{boundary} admit:
\begin{enumerate}
\item   A unique pair of interior subsonic solution
$(\rho_{sub},E_{sub})(x)\in C^{\frac{1}{2}}[0,1]\times H^1(0,1)$ satisfying
\begin{equation}\label{2.2-new}
1+m\sin(\pi x)\leq\rho_{sub}(x)\leq\overline{b},\ \ x\in[0,1],
\end{equation}
and particularly,
\begin{equation}
\begin{cases}
C_1(1-x)^{\frac{1}{2}}\le \rho_{sub}(x)-1 \le C_2(1-x)^{\frac{1}{2}}, \\
-C_3(1-x)^{-\frac{1}{2}}\le \rho'_{sub}(x) \le
-C_4(1-x)^{-\frac{1}{2}},
\end{cases}
\ \mbox{ for } x \mbox{ near } 1, \label{2.2-3new}
\end{equation}
where $m=m(\tau,\underline{b})<\overline{b}-1$ is a
small positive constant, and $C_2>C_1>0$ and $C_3>C_4>0$ are some
positive constants;

\item At least one pair of supersonic solution $(\rho_{sup},E_{sup})(x)\in C^{\frac{1}{2}}[0,1]\times H^1(0,1)$ satisfying  $0<\rho_{sup}(x)\le 1$
and
\begin{equation}
\begin{cases}
C_5x^{\frac{1}{2}}\le 1-\rho_{sup}(x)  \le C_6x^{\frac{1}{2}}, \\
-C_7x^{-\frac{1}{2}}\le \rho'_{sup}(x)  \le -C_8x^{-\frac{1}{2}},
\end{cases}
\ \mbox{ for } x \mbox{ near } 0, \label{2.2-3}
\end{equation}
where $C_6>C_5>0$ and $C_7>C_8>0$ are some positive constants.
$\rho_{sup}$ has only one critical point $z_0$ over $(0,1)$, such
that $(\rho_{sup})_x<0$ on $(0,z_0)$ and $(\rho_{sup})_x>0$ on
$(z_0,1)$.

\item Assume further that $\tau$ is large and that $\bar{b}-\underline{b}\ll1$, then equations  \eqref{1.5}-
\eqref{boundary} have
infinitely many transonic solutions $(\rho_{trans},E_{trans})(x)$ combining
stationary shocks which satisfy the entropy condition
\eqref{entropy} and the Rankine-Hugoniot jump condition \eqref{RH}
at different jump locations $x_0$, where $x_0$ can be uniquely
determined when $\rho_l$ satisfying $\rho_r-\rho_l\ll 1$ is fixed,
but the choice of $\rho_l$ can be infinitely many;

\item Assume further that $b(x)=b>1$ is a constant, then when $\tau$ is
small enough,  equations  \eqref{1.5}-
\eqref{boundary} have infinitely many $C^1$
 transonic solution; moreover, in this case there is no
transonic shock solution.

\end{enumerate}
\end{theorem}

\begin{theorem}[Case of supersonic doping profile]\label{main-thm-3}
Let the doping profile be supersonic such that $b(x)\in
L^\infty(0,1)$ and $0<b(x)\leq\overline{b}\leq1$. Then:
\begin{enumerate}
\item there is no interior subsonic solution to  equations  \eqref{1.5}-
\eqref{boundary};

\item  there is no interior supersonic solution nor transonic solution to  \eqref{1.5}-
\eqref{boundary},
if the doping profile is sufficiently small such that
$\overline{b}(1+\sqrt{2\overline{b}})<1$;

\item there is no interior supersonic solution nor transonic solution to  \eqref{1.5}-
\eqref{boundary},
if the relaxation time  is  small with $\tau<\frac{1}{3}$;

\item there exists at least one interior supersonic solution $(\rho_{sup},E_{sup})(x)$ to \eqref{1.5}-
\eqref{boundary},
satisfying  $\rho_{sup}\in C^{\frac{1}{2}}[0,1]$  and the optimal
estimate \eqref{2.2-3}, if the doping profile $b(x)$ is close to the
sonic boundary $\rho=1$ and the relaxation time is sufficiently
large $\tau\gg 1$;

\item there exist infinitely many transonic shock solutions $(\rho_{trans},E_{trans})(x)$ to  \eqref{1.5}-
\eqref{boundary}
joint with some stationary shocks satisfying the  entropy condition
\eqref{entropy} and the Rankine-Hugoniot jump condition \eqref{RH}
at different jump locations $x_0$, if the doping profile $b(x)$ is
close to the sonic boundary $\rho=1$ and the relaxation time is
sufficiently large $\tau\gg 1$, where $x_0$ can be uniquely
determined when $\rho_l$ satisfying $\rho_r-\rho_l\ll 1$ is fixed,
but the choice of $\rho_l$ can be infinitely many.

\end{enumerate}
\end{theorem}

\begin{remark}  In Parts 1 and 2 of Theorem \ref{main-thm-1}, see also Part 4 of Theorem \ref{main-thm-3}, the estimates
\eqref{2.2-3new} and \eqref{2.2-3} imply that $C^{\frac{1}{2}}[0,1]$
is the optimal H\"{o}lder space for the global regularity of the
subsonic solution $\rho_{sub}(x)$ and the supersonic solution
$\rho_{sup}(x)$. Such a regularity $C^{\frac{1}{2}}$ matches also the analysis for the steady
nozzle in \cite{Wang,Wang-Xin}.

In Part 3 of Theorem \ref{main-thm-1} and Part 5 of Theorem \ref{main-thm-3}, when $\tau\gg 1$, namely, the semiconductor effect is small, then the steady hydrodynamic system
possesses infinitely many transonic shock solutions. The similar results in \cite{Luo-Xin,Luo-Rauch-Xie-Xin} can be regarded
in some sense of our special example as  $\tau=\infty$ .

In Part 4 of Theorem \ref{main-thm-1}, if $b(x)$ is a constant and
$\tau$ is small, Part 4 implies that the regularity of the subsonic
solution on the left boundary, as well as the regularity for the
supersonic solution on the right boundary, can be lifted up to
$C^1$. It seems that such a $C^1$ regularity of transonic solutions
is the first result obtained for semiconductor models so far.
Essentially, the strong damping effect (the semiconductor effect) of
$-\frac{J}{\tau}$ makes the transonic solutions to be $C^1$ smooth.
Notice that the $C^2$  transonic flow also arises in the
finite de Laval nozzles, where the geometry structure causes the
transonic flow to be smooth. For details, we  refer to the
interesting works of C. Wang and Z. Xin \cite{Wang-Xin-smooth,Wang-Xin-smooth2}.
\end{remark}

\begin{remark}
If $b(x)\equiv1$, then  the steady-state Euler-Poisson equations \eqref{1.5} and
\eqref{boundary} admit the sonic solution $(\rho_{sonic},E_{sonic})(x)\equiv (1,\frac{1}{\tau})$.
\end{remark}

\begin{remark}
Theorem \ref{main-thm-3} indicates that when the doping profile is
small enough, or the relaxation time is small enough, then the
system has no solution. This also explains the physical phenomenon
that the semiconductor device doesn't work efficiently when the
background of the device is too pure.
\end{remark}

The paper is organized as follows. In Section 2, we prove Theorem \ref{main-thm-1}. The adopted approach is the method of viscosity vanishing and the
technical energy method with the help of phase-plane analysis. The existence of infinity many $C^1$-smooth  transonic solutions and transonic shocks  are proved by the artful construction method. In Section 3,
the main duty is to prove Theorem \ref{main-thm-3}.
Finally, in Section 4, when the pressure function is $P(\rho)=T\rho^\gamma$ for $\gamma>1$,
the hydrodynamic system of Euler-Poisson equations becomes  isentropic. We conclude that  the results presented in Theorems \ref{main-thm-1} and
\ref{main-thm-3}  all hold for the isentropic system with $\gamma>1$.

\section{The case of subsonic doping profile}\label{subsonic doping profile}
In this section, we assume that $\underline{b}>1$. In other words,
the doping profile is subsonic. First of all, let us test a special
case when $b(x)\equiv b>1$ (constant), we may observe the structure
of stationary solutions to system \eqref{1.5}-\eqref{boundary} from the
phase-plane analysis.  Notice that, when $b>1$, the critical point
of system \eqref{1.5}  is $A=\Big(b,\dfrac{1}{\tau b}\Big)$, and the
Jacobian matrix of system \eqref{1.5}  at $A$ is:
\begin{equation*}
J(A) =\begin{bmatrix}
\dfrac{b}{\tau(b^2-1)} & \dfrac{b^3}{b^2-1} \\
 1 & 0
\end{bmatrix}.
\end{equation*}
It is easy to see that the eigenvalues $\lambda$ of matrix
$J(A)$ satisfy the following characteristic equation
\begin{equation}\label{eigenvalue}
\lambda^2-\dfrac{b\lambda}{\tau(b^2-1)}-\dfrac{b^3}{b^2-1}=0.
\end{equation}
Notice that,  $\lambda_1\lambda_2=-\dfrac{b^3}{b^2-1}<0$, where
$\lambda_1$ and $\lambda_2$ are the roots of \eqref{eigenvalue}.
Thus, $A$ is a saddle point. On the other hand, it follows from system
\eqref{1.5} that
\begin{equation}\label{direction}
\dfrac{d\ET}{d\rho}=\dfrac{(\rho-b)(1-\frac{1}{\rho^2})}{\rho
\ET-\frac{1}{\tau}},
\end{equation}
which helps to determine the directions of all trajectories. Here
and in the sequel, to avoid confusion, we denote by $\ET=\ET(\rho)$
the function of the trajectory.

Figure \ref{fig1-1} is the phase-plane of $(\rho,\ET)$ with
$\tau=15$ and $b=1.5$, from which we observe that there exist at
least one interior subsonic solution and one interior supersonic
solution. In Figure \ref{fig1-2}, we draw the profiles of the
interior subsonic solution and interior supersonic solution. Figure
\ref{fig2-1} demonstrates how to construct an interior transonic
shock solution when $\tau$ is large: the discontinuous trajectory in
blue stands for a transonic shock solution with smaller length (e.g.
$\frac{1}{2}$) and is structured by a stationary shock at $x_0$ with
the Rankine-Hugoniot jump condition \eqref{RH} linking the other two
solutions: one is a supersonic solution $\rho_{sup}(x)$ with
$\rho_{sup}(0)=1$ and $\rho_{sup}(x_0^-)=\rho_l<1$, and the other is
a subsonic solution $\rho_{sub}(x)$ with
$\rho_{sub}(x_0^+)=\rho_r>1$ and $\rho_{sub}(\frac{1}{2})=1$; the
discontinuous trajectory in red represents a similar transonic shock
solution with larger length (e.g. $\frac{3}{2}$) satisfying the
entropy condition and the Rankine-Hugoniot condition at some jump
location. By continuity, there is an interior transonic shock
solution to \eqref{1.5} on $[0,1]$. Since the choice of
$\rho_l=\rho_{sup}(x_0^+)$ can be infinitely many when
$\rho_r-\rho_l\ll1$, there are infinitely many transonic shock
solutions. In Figure \ref{fig2-2}, we draw two transonic shock
solutions to system \eqref{1.5} with different $\rho_l$.

While, when $\tau$ is small, we see in Figure \ref{fig1-5} that the
phase-plane changes dramatically: many subsonic trajectories start
from the same point $(1,\frac{1}{\tau})$, and many supersonic
trajectories end at the same point $(1,\frac{1}{\tau})$. As a
result, one can see that there are possibly smooth transonic
solutions, which is constructed by two solutions at some location
$x_0$: one is an interior supersonic solution with
$\rho_{sup}(0)=1=\rho_{sup}(x_0)$, and the other is an interior
subsonic solution with $\rho_{sub}(x_0)=1=\rho_{sub}(1)$. Since the
transition location $x_0$ can be chosen arbitrarily in $(0,1)$,
these smooth transonic solutions are infinitely many.


\begin{figure}
\centering
\includegraphics[height=2.5in]{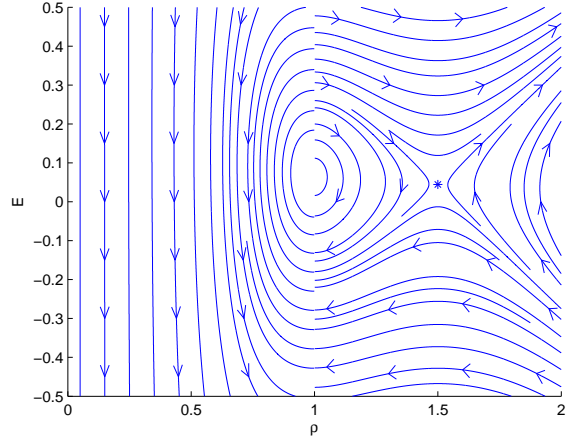}
\caption{Phase plane of $(\rho,E)$ with $\tau=15$ and $b=1.5$; $*$
is the saddle point $A=(1.5,2/45)$.} \label{fig1-1}
\end{figure}



\begin{figure}
\centering
\includegraphics[height=3in]{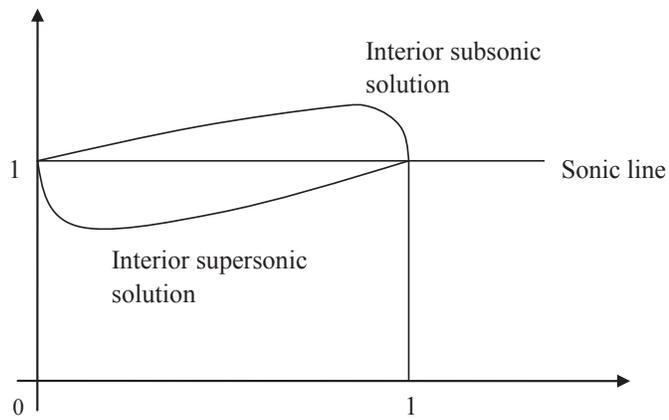}
\caption{Interior subsonic solution and interior supersonic solution for the case of subsonic doping
profile.}
\label{fig1-2}
\end{figure}



\begin{figure}
\centering
\includegraphics[height=3in]{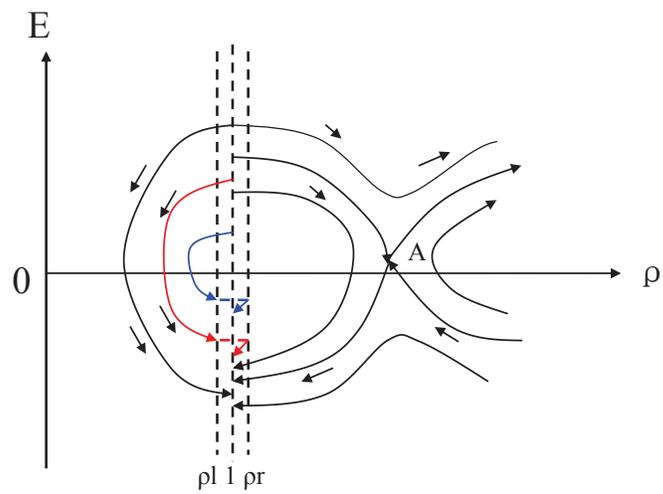}
\caption{Transonic shock trajectories  in the phase plane of
$(\rho,E)$ for the case of subsonic doping profile when $\tau$ is
large.} \label{fig2-1}
\end{figure}



\begin{figure}
\centering
\includegraphics[height=2.5in]{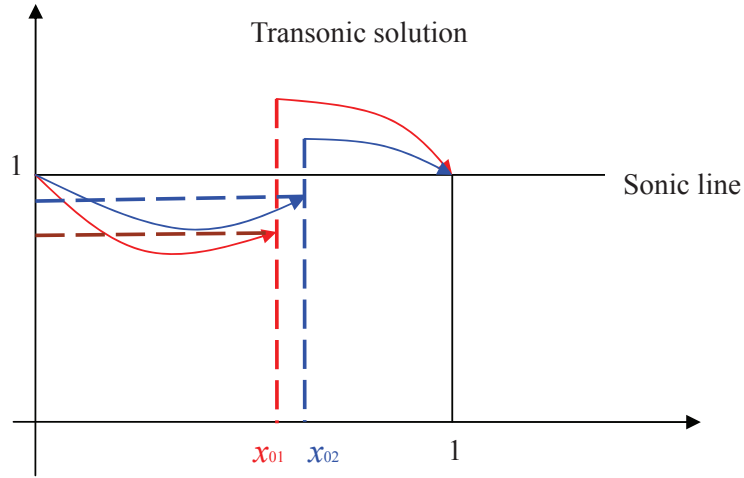}
\caption{Transonic shock solutions in the case of subsonic doping
profile when $\tau$ is large.} \label{fig2-2}
\end{figure}



\begin{figure}
\centering
\includegraphics[height=2.5in]{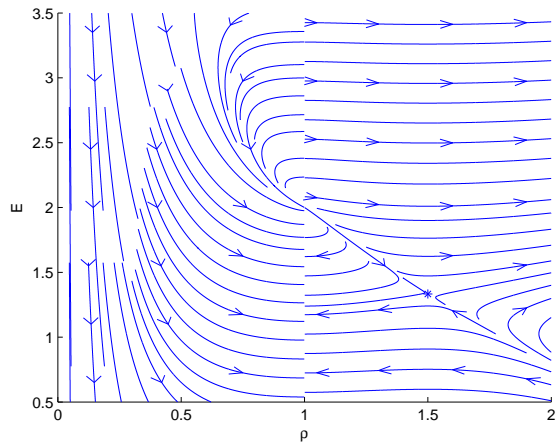}
\caption{Phase plane of $(\rho,E)$ with $\tau=0.5$ and $b=1.5$; $*$
is the saddle point $A=(1.5,4/3)$.} \label{fig1-5}
\end{figure}


Next we are going to prove Theorem \ref{main-thm-1} for a subsonic
doping profile $b(x)>1$ in general form.

\subsection{Unique interior subsonic solution}
Firstly, we  prove that there exists a unique interior subsonic
solution to equation \eqref{elliptic}. The adopted approach is the technical
compactness method, which is inspired by the vanishing viscosity
method.

\begin{theorem}\label{thm1}
Assume that $b\in L^\infty(0,1)$ and $\underline{b}>1$, then
equation \eqref{elliptic} has a unique interior subsonic solution
$\rho_{sub}$ satisfying
\begin{equation}\label{2.2}
1+m\sin(\pi x)\leq\rho_{sub}\leq\overline{b},\ \ x\in[0,1],
\end{equation}
where $m=m(\tau,\underline{b})$ is a positive
constant.
\end{theorem}

Since the equation \eqref{elliptic} is partially elliptic but degenerates at the boundary, so the corresponding solution to  \eqref{elliptic} will lack the necessary regularity, and we cannot directly work on \eqref{elliptic}.   In order to prove Theorem \ref{thm1}, now we  consider
the following approximate equation:
\begin{equation}\label{approximate equation}
\left \{\begin{array}{ll}
\left[\left(\dfrac{1}{\rho_j}-\dfrac{j^2}{(\rho_j)^3}\right)(\rho_j)_x\right]_x
+\left(\dfrac{j}{\tau\rho_j}\right)_x- [\rho_j-b(x)]=0,& x\in(0,1)\\
\rho_j(0)=\rho_j(1)=1,&
\end{array} \right.
\end{equation}
where the parameter $j$ is chosen as a constant such that $0<j<1$.
Thus, the equation \eqref{approximate equation} is expected to be
uniformly elliptic in $[0,1]$, because
$\frac{1}{\rho_j}-\frac{j^2}{\rho_j^3}=\frac{1}{\rho_j^3}(\rho_j+j)(\rho_j-j)>0$
for the expected solution $\rho_j\ge 1$. To show the wellposedness
of the approximate equation \eqref{approximate equation} and to
establish the lower bound estimate in \eqref{2.2}, we need the
following comparison principle.

\begin{lemma}[Comparison principle]\label{comparison-principle}
Let $U\in C^1[0,1]$ be a weak solution of \eqref{approximate
equation} satisfying $U\geq1$ on $[0,1]$, and that
\begin{equation}\label{equality}
\int_0^1\left[\left(\frac{1}{U}-\frac{j^2}{U^3}\right)U_x+\frac{j}{\tau
U}\right]\varphi_xdx + \int_0^1(U-b)\varphi dx=0 \ \
\text{for any } \varphi\in H_0^1(0,1),
\end{equation}
where $0<j<1$ is a constant, and let $V\in C^1[0,1]$ be such that
$V(x)>0$ for $x\in[0,1]$, $V(0)\leq1$, $V(1)\leq1$ and
\begin{equation*}
\int_0^1\left[\left(\frac{1}{V}-\frac{j^2}{V^3}\right)V_x+\frac{j}{\tau
V}\right]\varphi_xdx + \int_0^1(V-b)\varphi dx\leq0
\ \ \text{for any } \varphi\geq0,\ \varphi\in H_0^1(0,1).
\end{equation*}
Then $U(x)\geq V(x)$ over $[0,1]$.
\end{lemma}

\begin{proof} Inspired by the textbook \cite{Gilbarg-Trudinger} (see Theorem 2.7 in Section 10.4), we can prove this comparison principle.
Let us  denote
$$A(z,p):=\left(\frac{1}{z}-\frac{j^2}{z^3}\right)p+\frac{j}{\tau z}$$
for simplicity. Then, for
any $\varphi\in H_0^1(0,1)$, $\varphi\geq0$, we have
\begin{equation} \label{2.4}
\int_0^1[A(V,V_x)-A(U,U_x)]\varphi_xdx+ \int_0^1(V-U)\varphi
dx\leq0.
\end{equation}
Set $e(x):=V(x)-U(x)$. A simple calculation gives
\begin{equation*}
\begin{split}
A(V,V_x)-A(U,U_x)&=A(V,V_x)-A(U,V_x)+A(U,V_x)-A(U,U_x)\\&=\int_0^1\frac{\partial
A}{\partial z}(V_t,V_x)dt\cdot e(x)+\int_0^1\frac{\partial
A}{\partial p}(U,(V_{t})_x)dt\cdot e_x(x),\end{split}\end{equation*}
where $V_t(x):=tV(x)+(1-t)U(x)$. Taking
$\varphi(x)=\frac{e^+(x)}{e^+(x)+h}$ with $e^+(x):=\max\{0,e(x)\}$
and $h>0$ being a constant, a straightforward computation yields
\begin{equation*}
\left[\ln(1+e^+(x)/h)\right]_x=\frac{e^+_x(x)}{e^+(x)+h} \ \text{
and }\ \varphi_x=\frac{h}{e^+(x)+h}\left[\ln(1+e^+(x)/h)\right]_x.
\end{equation*}
Since $0<j<1$, $v\in C^1[0,1]$ and $\underset{x\in[0,1]}{\min} v>0$,
it is easy to see that
\begin{equation*}
\begin{split}
&\int_0^1\frac{\partial A}{\partial
p}(U,(V_{t})_x)dt=\frac{1}{U}-\frac{1}{U^3}+\frac{1-j^2}{U^3}\geq
\frac{1-j^2}{\|U\|_{L^\infty}^3},\\
&\int_0^1\frac{\partial A}{\partial z}(V_t,V_x)dt\leq
C\|V_x\|_{C[0,1]}+\frac{Cj}{\tau}\leq C.\end{split}
\end{equation*}
It then follows from \eqref{2.4} that
\begin{equation*}
\begin{split}
&\frac{h(1-j^2)}{\|U\|_{L^\infty}^3}\int_0^1\left|\left[\ln(1+e^+(x)/h)\right]_x\right|^2dx
+ \int_0^1\frac{(e^+(x))^2}{e^+(x)+h}dx\\ &\leq
Ch\int_0^1\frac{e^+(x)}{e^+(x)+h}\left|\left[\ln(1+e^+(x)/h)\right]_x\right|dx\\
&\leq\frac{h(1-j^2)}{2\|U\|_{L^\infty}^3}\int_0^1\left|\left[\ln(1+e^+(x)/h)\right]_x\right|^2dx
+\frac{C^2h\|U\|_{L^\infty}^3}{2(1-j^2)},\end{split}
\end{equation*}
where we have used Young's inequality in the second inequality.
Thus,
\[
\int_0^1\left|\left[\ln(1+e^+(x)/h)\right]_x\right|^2dx
\leq\frac{C^2\|U\|_{L^\infty}^6}{(1-j^2)^2} \ \text{ for any }h>0.
\]
This inequality together with Poincar\'e's inequality leads to
\begin{equation}
\int_0^1\left[\ln(1+e^+(x)/h)\right]^2dx\le 2
\int_0^1\left|\left[\ln(1+e^+(x)/h)\right]_x\right|^2dx
\leq\frac{C^2\|U\|_{L^\infty}^6}{(1-j^2)^2} \ \text{ for any }h>0.
\label{new-1}
\end{equation}
Now letting $h\rightarrow0^+$, one can see that if $e^+(x)\neq0$ for
some $x\in(0,1)$, then
\[
\lim_{h\to 0^+}\int_0^1\left|\left[\ln(1+e^+(x)/h)\right]\right|^2dx=\infty,
\]
which is a contradiction to \eqref{new-1}. Therefore, $U(x)\geq
V(x)$ over $[0,1]$.
\end{proof}

Let us now prove the wellposedness of equation \eqref{approximate
equation}.
\begin{lemma}\label{subsonic solution}
Assume that $b(x)\in L^\infty(0,1)$ and $\underline{b}>1$, then
\eqref{approximate equation} admits a unique weak solution $\rho_j$
satisfying $\rho_j\in H_0^1(0,1)$ and
\begin{equation}\label{2.3}
1+m\sin(\pi x)\leq\rho_j(x)\leq\overline{b},\ \ x\in[0,1],
\end{equation}
where $m=m(\tau,\underline{b})<\overline{b}-1$ is a positive constant
independent of $j$.
\end{lemma}

\begin{remark} In \cite{Degond-Markowich}, Degond and Markowich  also obtained the uniqueness of the subsonic solution, but they needed to restrict the current density sufficiently small $j\ll 1$ (the completely subsonic case). Here, we still have the uniqueness of the subsonic solution for   any $j$ with $0<j<1$ in the
case of subsonic doping profile.
\end{remark}

\begin{proof}
Because $0<j<1$, the fluid velocity of equation \eqref{approximate
equation} is $j/\rho_j$, which is subsonic if $\rho_j\geq1$. In
other words, equation \eqref{approximate equation} is uniformly
elliptic for $\rho_j\geq1$. Recall Theorem 1 of
\cite{Degond-Markowich}, equation \eqref{approximate equation} has a
subsonic weak solution $\rho_j\in H^2(0,1)$ satisfying
$1\leq\rho_j(x)\leq\overline{b}$. Thus, we only need to show that
such $\rho_j$ is unique for any $0<j<1$, and to establish the lower
bound estimate in \eqref{2.3}.

Suppose that there are two solutions $u$ and $v$ satisfying
$u,v\geq1$, $u,v\in H^2(0,1)$. By the Sobolev imbedding theorem,
$u$, $v\in C^1[0,1]$. Hence, the comparison principle (Lemma
\ref{comparison-principle}) gives  $u(x)=v(x)$ over  $[0,1]$.

We now derive the lower bound estimate for $\rho_j(x)$. Denote
\[
q(x):=1+m\sin(\pi x),
\]
 where $m>0$ is a constant to be determined
later. Since $0<j<1$, it is easy to calculate that
\begin{equation*}
-\left[\left(\dfrac{1}{q}-\dfrac{j^2}{q^3}\right)q_x\right]_x
-\left(\dfrac{j}{\tau q}\right)_x+ (q-b)\leq
Cm+ (1-b)\leq
Cm+ (1-\underline{b})<0,
\end{equation*}
if $m$ is small enough such that $Cm< (\underline{b}-1)$. Here
$C=C(\tau)$ is a positive constant. Thus, by Lemma
\ref{comparison-principle} again, we have $\rho_j(x)\geq
q(x)=1+m\sin(\pi x)$ on $[0,1]$.
\end{proof}

\begin{proof}[Proof of Theorem \ref{thm1}]
Multiplying \eqref{approximate equation} by $(\rho_j-1)$, we have
\begin{equation} \label{2.5}
\begin{split}
&(1-j^2)\int_0^1\frac{|(\rho_j)_x|^2}{(\rho_j)^3}dx
+\frac{4}{9}\int_0^1\frac{(\rho_j+1)}{(\rho_j)^3}\cdot|((\rho_j-1)^{\frac{3}{2}})_x|^2dx
\\&\quad+\frac{j}{\tau}\int_0^1\frac{(\rho_j)_x}{\rho_j}dx+ \int_0^1(\rho_j-b)(\rho_j-1)dx=0.
\end{split}\end{equation}
Noting that
\begin{equation*}
\begin{split}
\frac{j}{\tau}\int_0^1\frac{(\rho_j)_x}{\rho_j}dx&=\frac{j}{\tau}\int_0^1(\ln\rho_j)_xdx=0,\\
\int_0^1(\rho_j-b)(\rho_j-1)dx&=\int_0^1(\rho_j-1)^2dx+\int_0^1(1-b)(\rho_j-1)dx\\
&\geq\frac{1}{2}\int_0^1(\rho_j-1)^2dx-\frac{1}{2}\int_0^1(b-1)^2dx,
\end{split}\end{equation*}
$0<j<1$, and $1\leq\rho_j\leq\overline{b}$, it follows from
\eqref{2.5} that
\begin{eqnarray*}
&&\frac{(1-j^2)}{\overline{b}^3}\int_0^1|(\rho_j)_x|^2dx
+\frac{8}{9\overline{b}^3}\int_0^1\left|((\rho_j-1)^{\frac{3}{2}})_x\right|^2dx
+\frac{1}{2}\int_0^1(\rho_j-1)^2dx\\
&&\leq\frac{1}{2}\int_0^1[b(x)-1]^2dx,
\end{eqnarray*}
which gives
\begin{equation}\label{eqn-2.8}
\left\|(\rho_j-1)^{\frac{3}{2}}\right\|_{H^1}\leq C\ \text{and }\
\left\|(1-j^2)(\rho_j)_x\right\|_{L^2}\leq C(1-j^2)^{\frac{1}{2}}.
\end{equation}
Thus, by the compact imbedding $H^1(0,1)\hookrightarrow
C^{1/2}[0,1]$, there exists a function $\rho$ such that, as
$j\rightarrow1^-$, up to a subsequence,
\begin{align}
&(\rho_j-1)^{\frac{3}{2}}\rightharpoonup(\rho-1)^{\frac{3}{2}} \ \
\text{weakly in } H^1(0,1),\label{eqn-2.9}\\&
(\rho_j-1)^{\frac{3}{2}}\rightarrow(\rho-1)^{\frac{3}{2}} \ \
\text{strongly in } C^{\frac{1}{2}}[0,1],\label{eqn-2.10}
\\&
(1-j^2)(\rho_j)_x\rightarrow0 \ \ \text{strongly in }
L^2(0,1).\label{eqn-2.11}
\end{align}
Observing that
$((\rho_j-1)^2)_x=\frac{4}{3}(\rho_j-1)^{\frac{1}{2}}((\rho_j-1)^{\frac{3}{2}})_x$,
we get from \eqref{eqn-2.8} that
\begin{equation*}
\left\|(\rho_j-1)^2\right\|_{H^1}=\left\|(\rho_j-1)^2\right\|_{L^2}+\left\|((\rho_j-1)^2)_x\right\|_{L^2}
\leq C\left\|(\rho_j-1)^{\frac{3}{2}}\right\|_{H^1}\leq C,
\end{equation*}
which leads to
\begin{equation}\label{eqn-2.12}
(\rho_j-1)^2\rightharpoonup(\rho-1)^2 \ \ \text{weakly in } H^1(0,1)
\ \text{as }j\rightarrow1^-.
\end{equation} Now we multiply
\eqref{approximate equation} by $\varphi\in H_0^1(0,1)$ to derive
\begin{equation*}
\begin{split}
&\frac{1}{2}\int_0^1\frac{\rho_j+1}{\rho_j^3}[(\rho_j-1)^2]_x\varphi_xdx
+\int_0^1\frac{1}{\rho_j^3}(1-j^2)(\rho_j)_x\varphi_xdx\\&\quad+\frac{j}{\tau}\int_0^1\frac{\varphi_x}{\rho_j}
dx+ \int_0^1[\rho_j(x)-b(x)]\varphi dx=0.
\end{split}\end{equation*}
Letting $j\rightarrow1^-$, and applying
\eqref{eqn-2.10}-\eqref{eqn-2.12}, we prove the existence of weak
solution $\rho(x)=\rho_{sub}(x)$ satisfying
\eqref{weak-solution}. Since  $m$ presented in \eqref{2.3} is
independent of $j$, then the lower bound estimate in \eqref{2.2}
immediately follows from \eqref{2.3} and \eqref{eqn-2.10}.

To prove the uniqueness of interior subsonic solution, we first need
to investigate the regularity of $w(x)$ defined by
$w(x):=(\rho(x)-1)^2$. Clearly, $w\in H_0^1(0,1)$. From
\eqref{elliptic}, it can be verified that $w$ satisfies
\begin{equation}\label{w}
\left(\dfrac{(2+\sqrt{w})w_x}{2(1+\sqrt{w})^3}+\dfrac{1}{\tau(1+\sqrt{w})}\right)_x
- (\sqrt{w}+1-b)=0,\ \ x\in(0,1).
\end{equation}
For simplicity, we set
$$
f_1(x):=\dfrac{2+\sqrt{w}}{(1+\sqrt{w})^3},\quad
f_2(x):=\dfrac{1}{1+\sqrt{w}},\quad
f_3(x):=\frac{f_1(x)w_x(x)}{2}+\frac{f_2(x)}{\tau}.$$ Because
\eqref{w} holds in the sense of distribution, we have $f_3\in
H^1(0,1)$. By Sobolev imbedding theorem, we have $w,f_3\in
C^{1/2}[0,1]$. Since $w\geq0$ on $[0,1]$, then
\begin{equation*}
\begin{split}
|\sqrt{w(y)}-\sqrt{w(x)}|=\frac{|w(y)-w(x)|}{\sqrt{w(y)}+\sqrt{w(x)}}
\leq\frac{|w(y)-w(x)|}{\sqrt{|w(y)-w(x)|}}\leq
C|y-x|^{1/4}.\end{split}\end{equation*}
On the other hand, for any $x,y\in[0,1]$, it holds
\begin{equation*}
f_2(x)-f_2(y)=\dfrac{1}{1+\sqrt{w(x)}}-\dfrac{1}{1+\sqrt{w(y)}}
=\dfrac{\sqrt{w(y)}-\sqrt{w(x)}}{(1+\sqrt{w(x)})(1+\sqrt{w(y)})}.
\end{equation*}
Thus,
\[|f_2(x)-f_2(y)|\leq|\sqrt{w(y)}-\sqrt{w(x)}|\leq
C|y-x|^{1/4}.\] This means $f_2\in C^{1/4}[0,1]$. Similarly, we have
$f_1\in C^{1/4}[0,1]$. Notice that
$w_x=\frac{2f_3-2f_2/\tau}{f_1}\in C^{1/4}[0,1]$, then
\begin{equation}\label{2.16}
w\in C^{1+1/4}[0,1]. \end{equation} Now, integrating \eqref{w} over
$[0,x]$ and setting
$G_w(x):=\dfrac{(2+\sqrt{w(x)})w_x(x)}{2(1+\sqrt{w(x)})^3}+\dfrac{1}{\tau(1+\sqrt{w(x)})}$,
then
\begin{equation} \label{ode-w}
\left \{\begin{array}{ll}
        \dfrac{(2+\sqrt{w})w_x}{2(1+\sqrt{w})^3}=G_w-\dfrac{1}{\tau(1+\sqrt{w})},\\
       \displaystyle{  G_w(x)= G_w(0)+\int_0^x[\sqrt{w(s)}+1-b(s)]ds.}
        \end{array} \right.
\end{equation}

We are now ready to prove the uniqueness of interior subsonic solution.
Suppose $\rho_1(x)$ and $\rho_2(x)$ are two different interior subsonic
solutions to equation \eqref{elliptic}. So, there exists at least a
number $z\in(0,1)$ such that $\rho_1(z)\neq\rho_2(z)$. Without loss
of generality, we may assume $\rho_1(z)>\rho_2(z)$, then
$w_1(z)>w_2(z)$. Since $w_1,w_2\in C^{1+1/4}[0,1]$, there exists a
maximal interval $[a,c]\subset[0,1]$  such that $z\in(a,c)$,
$$w_1(a)=w_2(a),\ w_1(c)=w_2(c)\ \ \text{and}\ \ w_1(x)>w_2(x),\ x\in (a,c).$$
Obviously, it holds
\begin{align}
(w_1)_x(a)=\underset{x\rightarrow
a^+}{\lim}\frac{w_1(x)-w_1(a)}{x-a} \geq\underset{x\rightarrow
a^+}{\lim}\frac{w_2(x)-w_2(a)}{x-a}=(w_2)_x(a),\label{w11}\\
(w_1)_x(c)=\underset{x\rightarrow
c^-}{\lim}\frac{w_1(x)-w_1(c)}{x-c} \leq\underset{x\rightarrow
c^-}{\lim}\frac{w_2(x)-w_2(c)}{x-c}=(w_2)_x(c).\label{w22}
\end{align}
Owing to \eqref{w22} and the first equation of \eqref{ode-w},
$$G_{w_1}(c)\leq G_{w_2}(c).$$
Substituting this inequality into the second equation of
\eqref{ode-w}, we have
$$ G_{w_1}(a)+\int_a^c[\sqrt{w_1(x)}+1-b(x)]dx\leq G_{w_2}(a)+\int_a^c[\sqrt{w_2(x)}+1-b(x)]dx.$$
Since  $w_1(x)>w_2(x)$ over $(a,c)$, then
$$G_{w_1}(a)<G_{w_2}(a).$$
Using the first equation of \eqref{ode-w} again, we obtain
$$(w_1)_x(a)<(w_2)_x(a),$$
which contradicts to \eqref{w11}. Therefore, $\rho_1(x)=\rho_2(x)$ over
$[0,1]$, namely, the interior subsonic solution $\rho_{sub}(x)$ is unique.
\end{proof}

We proceed to study the regularity of this interior subsonic
solution.

\begin{proposition}\label{local-analysis-sub}
$\rho_{sub}\in C^{1/2}[0,1]$, and there exist $0<s_1<1$, $C_i$
$(i=1,2,3,4)$ such that
\begin{equation} \label{eqn-2.20}
\begin{split}
C_1(1-x)^{1/2}&<\rho_{sub}(x)-1<C_2(1-x)^{1/2}, \\
-C_3(1-x)^{-1/2}&<(\rho_{sub})_x(x)<-C_4(1-x)^{-1/2},
\end{split}
\ \ \mbox{ for } x\in [1-s_1,1].
\end{equation}

\end{proposition}

\begin{remark}
This proposition indicates that $\frac{1}{2}$ is the optimal
exponent in H\"{o}lder space for the global regularity of the unique interior
subsonic solution $\rho_{sub}(x)$. And the derivative of the
approximate subsonic solution sequence $\{\rho_j\}_{0<j<1}$
constructed in Lemma \ref{subsonic solution} blows up as
$j\rightarrow1^-$  for $x\approx 1$, namely, $\underset{j\to
1^-}{\lim} \rho'_j(x)=-\infty$ for $x\approx 1$.
\end{remark}

\begin{proof} For convenience, we denote by $\rho$ the interior
subsonic solution of \eqref{1.5}. By \eqref{2.16}, we
have $(\rho-1)^2=w\in C^1[0,1]$. Since $\rho\geq1$ on $[0,1]$, then
\[
|\rho(x)-1+\rho(y)-1|=|\rho(x)-1|+|\rho(y)-1|\ge |(\rho(x)-1)-(\rho(y)-1)| =|\rho(x)-\rho(y)|.
\]
Thus, we have
\begin{equation*}
\frac{|\rho(x)-\rho(y)|^2}{|x-y|}
=\frac{|\rho(x)-\rho(y)||(\rho(x)-1)^2-(\rho(y)-1)^2|}{|x-y||\rho(x)-1+\rho(y)-1|}
\leq\frac{|w(x)-w(y)|}{|x-y|}\leq C,
\end{equation*}
for any $x,y\in[0,1]$, which indicates that $\rho\in C^{1/2}[0,1]$.

Now we are going to prove the estimates in  \eqref{eqn-2.20}. We
first claim $E(1)<\dfrac{1}{\tau}.$ Otherwise, if
$E(1)\geq\dfrac{1}{\tau}$, then it will imply a contradiction. In
fact, since $\rho\in C[0,1]$ and $\rho(1)=1<\underline{b}\le b(x)$
for $x\in [0,1]$,  there exists $\hat{\epsilon}>0$ such that
$\rho(x)-b(x)<0$ for a.e. $x\in[1-\hat{\epsilon},1]$. By integrating
the second equation of \eqref{1.5} over $[x,1]$ for
$x\in[1-\hat{\epsilon},1]$, we have
\[
 E(x)= E(1)-\int_x^1[\rho(s)-b(s)]ds> E(1)\geq\dfrac{1}{\tau}, \ \ \mbox{ for }  x\in[1-\hat{\epsilon},1].
\]
Noting $\rho(x)>1$ over $(0,1)$, we have
$E(x)-\dfrac{1}{\tau\rho(x)}\geq\dfrac{1}{\tau}\left(1-\dfrac{1}{\rho(x)}\right)>0$
for $x\in[1-\hat{\epsilon},1]$. It then follows from the first
equation of \eqref{1.5} that $\rho_x(x)>0$ on
$[1-\hat{\epsilon},1]$, which contradicts to the fact that
$\rho(1)=1$ and $\rho(x)>1$ over $(0,1)$.

Now let $q:=E(1)-\dfrac{1}{\tau}$, then $q<0$. Based on the continuity of the function
$\left(E(x)-\dfrac{1}{\tau\rho(x)}\right)$, there exists a number
$0<s_1<\hat{\epsilon}$ such that
\begin{equation}\label{new-2}
\frac{3q}{2}\le E(x)-\dfrac{1}{\tau\rho(x)}\le \frac{q}{2}<0 \text{
for } x\in[1-s_1,1].
\end{equation}
From the first equation of \eqref{1.5}, we have
\[
E(x)-\frac{1}{\tau\rho(x)}=\Big(1-\frac{1}{\rho^2}\Big)\frac{\rho_x}{\rho}=
\frac{\rho+1}{\rho^3}(\rho-1)\rho_x = \frac{\rho+1}{2\rho^3}\Big((\rho-1)^2\Big)_x.
\]
Applying \eqref{new-2} to the above equation, we then
have
\[
\frac{3q\rho^3(x)}{\rho(x)+1}\le
\Big((\rho-1)^2\Big)_x=\Big[E(x)-\frac{1}{\tau\rho(x)}\Big]\frac{2\rho^3(x)}{\rho(x)+1}
\le  \frac{q\rho^3(x)}{\rho(x)+1}<0 \text{ for } x\in[1-s_1,1].
\]
Applying  \eqref{2.2}  to the above inequalities, we can estimate
\begin{equation}\label{new-4}
\frac{3q\bar{b}^3}{2}<\Big((\rho(x)-1)^2\Big)_x<\frac{q}{\bar{b}+1}<0
\ \text{for}\ x\in[1-s_1,1].
\end{equation}
Integrating \eqref{new-4} over $[x,1]$ for $x\in [1-s_1,1]$, we get
\begin{equation}\label{new-5}
C_1(1-x)^{\frac{1}{2}}<\rho(x)-1<C_2(1-x)^{\frac{1}{2}}, \
\text{for}\ x\in[1-s_1,1],
\end{equation}
with
\[
C_1:=\sqrt{\frac{|q|}{\bar{b}+1}} \mbox{ and }
C_2:=\sqrt{\frac{3|q|{\bar{b}}^3}{2}}.
\]
Furthermore, from \eqref{new-4}, we have
\[
\frac{3q\bar{b}^3}{4(\rho(x)-1)}<\rho_x(x)<\frac{q}{2(\bar{b}+1)(\rho(x)-1)}<0
\ \text{for}\ x\in[1-s_1,1].
\]
This with \eqref{new-5} together implies
\[
-C_3(1-x)^{-\frac{1}{2}}<\rho_x(x)<-C_4(1-x)^{-\frac{1}{2}}, \ \ \
x\in[1-s_1,1],
\]
for some positive constants $C_3$ and $C_4$.
The proof is complete.
\end{proof}

\subsection{Interior supersonic solutions}

We next prove the existence of interior supersonic solutions of
\eqref{elliptic}.

\begin{theorem}\label{thm2}
Assume that $b\in L^\infty(0,1)$ and $\underline{b}>1$, then
equation \eqref{elliptic} admits an interior supersonic solution
$\rho_{sup}(x)$ satisfying $\ell\leq\rho_{sup}(x)\leq1$ over $[0,1]$ for
some positive constant $\ell$. Moreover,
$\rho_{sup}$ satisfies the following properties.

\begin{compactenum}[(i)]

\item For any $\frac{1}{2}>\epsilon>0$, there exists a number
$\delta>0$ such that $\rho_{sup}(x)\leq1-\delta$ for any
$x\in[\epsilon,1-\epsilon]$.

\item $\rho_{sup}$ has only one critical point $z_0$ over $(0,1)$ such
that $(\rho_{sup})_x<0$ on $(0,z_0)$ and $(\rho_{sup})_x>0$ on
$(z_0,1)$, i.e. $z_0$ is the minimal point.

\end{compactenum}

\end{theorem}

As shown in the proof of Theorem \ref{thm1}, we consider the approximate
equation
\begin{equation}\label{supersonic equation}
\left \{\begin{array}{ll}
\left[\left(\dfrac{1}{\rho_k}-\dfrac{k^2}{(\rho_k)^3}\right)(\rho_k)_x\right]_x
+\left(\dfrac{k}{\tau\rho_k}\right)_x- [\rho_k(x)-b(x)]=0,& x\in(0,1)\\
\rho_k(0)=\rho_k(1)=1,&
\end{array} \right.
\end{equation}
but with the parameter   $1<k<\infty$.

\begin{lemma}\label{supersonic-solution}
Let  the doping profile be subsonic with  $b(x)\in L^\infty(0,1)$ and $\underline{b}>1$. Then
\eqref{supersonic equation} admits a weak solution $\rho_k(x)$
satisfying
\begin{equation}\label{2.7}
\rho_k\in H^1(0,1)\ and \ 0<\rho_k(x)\leq1 \ over \ [0,1].
\end{equation}
\end{lemma}

\begin{remark}
Peng and Violet \cite{Peng-Violet} showed that if $k$ is large
enough, then equation \eqref{supersonic equation} has a supersonic
solution. Our Lemma \ref{supersonic-solution} further show that, in
the case of subsonic doping profile, for all $1<k<\infty$, equation
\eqref{supersonic equation} has a supersonic solution. So, our result essentially improves the previous study in
\cite{Peng-Violet}.
\end{remark}

\begin{proof}
The velocity $u_k(x)=\dfrac{k}{\rho_k(x)}$ satisfies
\begin{equation}\label{2.8}
\left \{\begin{array}{ll}
\left[\left(u_k-\dfrac{1}{u_k}\right)(u_k)_x\right]_x
+\dfrac{(u_k)_x}{\tau }- \left(\dfrac{k}{u_k}-b\right)=0,& x\in(0,1)\\
u_k(0)=u_k(1)=k.&
\end{array} \right.
\end{equation}
So we only need to show that \eqref{2.8} has a weak solution $u_k\in
H^1(0,1)$ satisfying $k\leq u_k<\infty$. To this end, we define an
operator $\mathcal{T}:\psi\rightarrow u$ by solving the following
linear elliptic equation
\begin{equation}\label{2.9}
\left \{\begin{array}{ll}
\left[\left(\psi-\dfrac{1}{\psi}\right)u_x\right]_x
+\dfrac{u_x}{\tau }- \left(\dfrac{k}{\psi}-b\right)=0,& x\in(0,1)\\
u(0)=u(1)=k.&
\end{array} \right.
\end{equation}
Set
\begin{equation*}
\mathcal{X}:=\{\psi(x):\psi\in C^1[0,1],k\leq \psi(x)\leq
M,\psi(0)=\psi(1)=k,\|\psi\|_{C^\alpha[0,1]}\leq
\Lambda,\|\psi\|_{C^1[0,1]}\leq \Upsilon(\Lambda)\},
\end{equation*}
where $0<\alpha<1/2$, $M$, $\Lambda$ and $\Upsilon(\Lambda)$ are
some positive constants to be determined later. Suppose that
$\psi\in \mathcal{X}$, by $L^2$ theory of elliptic equation and the
Sobolev imbedding theorem, we see that equation \eqref{2.9} has a
unique solution $u\in C^{1+\alpha}[0,1]$ for $0<\alpha<1$.
Multiplying \eqref{2.9} by $(u-k)^-(x):=\min\{0,(u-k)(x)\}$, we have
\begin{equation}\label{2.17}
\int_0^1\left(\psi-\dfrac{1}{\psi}\right)|[(u-k)^-]_x|^2dx
-\dfrac{1}{\tau}\int_0^1u_x(u-k)^-dx+ \int_0^1\left(\dfrac{k}{\psi}-b\right)(u-k)^-dx=0.
\end{equation}
Because $k>1$ and $\psi\geq k$, we have $\psi-\dfrac{1}{\psi}\geq
k-1>0$, and noting that
\begin{equation*}
\dfrac{1}{\tau}\int_0^1u_x(u-k)^-dx=\dfrac{1}{2\tau}\int_0^1([(u-k)^-]^2)_xdx=0,
\end{equation*}
it follows from \eqref{2.17} that
\begin{equation}\label{2.10}
\left(k-1\right)\int_0^1|[(u-k)^-]_x|^2dx +
\int_0^1\left(\dfrac{k}{\psi}-b\right)(u-k)^-dx\leq0.\end{equation}
This inequality in combination with the fact that
$\frac{k}{\psi(x)}-b(x)<0$ gives $(u-k)^-(x)=0$ for all $x\in[0,1]$.
Thus, $u(x)\geq k$ over $[0,1]$. Now multiplying \eqref{2.9} by
$(u-k)$, just as shown in \eqref{2.10}, using Young's inequality and
Poincar\'{e}'s inequality, we get
\begin{equation*}
\begin{split}
(k-1)\int_0^1|(u-k)_x|^2dx
&\leq \int_0^1\left(b-\dfrac{1}{\psi}\right)(u-k)dx\\
&\leq\frac{k-1}{2}\int_0^1(u-k)^2dx+\frac{1}{2(k-1)}\int_0^1\left(b(x)-\dfrac{k}{\psi}\right)^2dx\\
&\leq\frac{k-1}{2}\int_0^1|(u-k)_x|^2dx+\frac{1}{2(k-1)}\int_0^1b^2(x)dx.\end{split}\end{equation*}
It then follows that
\begin{equation*}
\|u_x\|_{L^2(0,1)}\leq\frac{\|b\|_{L^2}}{k-1}.
\end{equation*}
Furthermore, a straightforward computation yields
\begin{equation*}
0<u(x)\leq k+\frac{\|b\|_{L^2}}{k-1}.\end{equation*} Thus,
the compact imbedding of $H^1(0,1)$ into $C^{\alpha_0}[0,1]$ with
$0<\alpha_0<1/2$ gives
\begin{equation*}
\|u\|_{C^{\alpha_0}[0,1]}\leq C_0(k,\|b\|_{L^2}) \text{ for
a constant }C_0>0.
\end{equation*}
Hence we determine $M=1+\frac{\|b\|_{L^2}}{k-1}$,
$\alpha=\alpha_0$ and $\Lambda=C_0(k,\|b\|_{L^2})$. By the
H\"{o}lder estimate for the first order derivative of divergence
form elliptic equation \cite{Gilbarg-Trudinger}, we derive
\begin{equation*}
\|u\|_{C^{1+\alpha}[0,1]}\leq C_1(k,\|b\|_{L^2},\Lambda).
\end{equation*}
Now we take $\Upsilon(\Lambda)=C_1(k,\|b\|_{L^2},\Lambda)$
with $\Lambda=C_0(k,\|b\|_{L^2})$, then it is easy to see
that $u\in \mathcal{X}$ and $\mathcal{X}$ is a nonempty bounded and
closed convex set in $C^1[0,1]$. On the other hand, by the
Arzel\`{a}-Ascoli theorem, the imbedding
$C^{1+\alpha}[0,1]\hookrightarrow C^1[0,1]$ is compact. Thus, the
operator $\mathcal{T}$ is a compact map of $\mathcal{X}$ into
itself. By Schauder fixed point theorem (see Corollary 2.3.10 in
\cite{Chang}), there exists a fixed point $u\in \mathcal{X}$ such
that
$$\mathcal{T}(u)=u.$$
Therefore, equation \eqref{2.8} has a weak solution $u_k\in
C^1[0,1]$, and $\rho_k(x)=k /u_k(x)$ is a desired weak supersonic
solution of \eqref{supersonic equation}.
\end{proof}

\begin{proof}[Proof of Theorem \ref{thm2}]
Multiplying \eqref{2.8} by $(u_k-k)$ and using Young's inequality,
we have
\begin{equation*}
\begin{split}
&(k-1)\int_0^1\frac{u_k+1}{u_k}|(u_k)_x|^2dx+\frac{4}{9}\int_0^1\frac{u_k+1}{u_k}|[(u_k-k)^{3/2}]_x|^2dx\\
&= \int_0^1\left(b-\dfrac{k}{u_k}\right)(u_k-k)dx\\
&\leq\frac{1}{3}\int_0^1(u_k-k)^3dx
+\frac{2}{3}\int_0^1\left(b-\dfrac{k}{u_k}\right)^{3/2}dx\\
&\leq\frac{1}{3}\int_0^1|[(u_k-k)^{3/2}]_x|^2dx
+\frac{2}{3}\int_0^1b^{3/2}(x)dx.\end{split}\end{equation*}
Thus, we have
\begin{equation}\label{2.11}
\|(k-1)^{\frac{1}{2}}(u_k)_x\|_{L^2}+\|(u_k-k)^{\frac{3}{2}}\|_{H^1}\leq
C
\end{equation}
for a constant $C$ independent of $k$, where we have used $k>1$ and
$u_k\geq k$. This inequality together with the Sobolev imbedding
theorem yields
\begin{equation}\label{2.20}
\|u_k\|_{L^\infty}\leq k+C^{\frac{2}{3}}.
\end{equation}
Hence
\begin{equation}\label{2.21}
\rho_k(x)=\frac{k}{u_k(x)}\geq\frac{k}{\|u_k\|_{L^\infty}}
\geq\frac{k}{k+C^{\frac{2}{3}}}\geq\frac{1}{1+C^{\frac{2}{3}}}\triangleq\ell,\
\forall \ x\in[0,1].
\end{equation}
A direct calculation yields
\begin{equation*}
(\rho_k)_x=-\frac{k(u_k)_x}{u_k^2}\ \text{and }
((1-\rho_k)^2)_x=\frac{4k(u_k-1)^{\frac{1}{2}}((u_k-1)^{\frac{3}{2}})_x}{3u_k^3}.
\end{equation*}
It then follows from \eqref{2.11} and \eqref{2.20} that
\begin{equation*}
\begin{split}
&\|(1-\rho_k)^2\|_{H^1}+\|(1-\rho_k)^{3/2}\|_{H^1}\leq
C_1,\\
&\|(k-1)(\rho_k)_x\|_{L^2}\leq
C_1(k-1)^{\frac{1}{2}}.
\end{split}
\end{equation*}
Thus, there exists a function $\rho_{sup}(x)$ such that, as
$k\rightarrow1^+$, up to a subsequence,
\begin{equation}\label{2.22}
\begin{split}
&(1-\rho_k)^2\rightharpoonup(1-\rho_{sup})^2 \ \ \text{weakly in }
H^1(0,1),\\&(1-\rho_k)^{3/2}\rightharpoonup(1-\rho_{sup})^{3/2} \ \
\text{weakly in } H^1(0,1),\\&
(1-\rho_k)^{3/2}\rightarrow(1-\rho_{sup})^{3/2} \ \ \text{strongly
in } C^{\frac{1}{2}}[0,1]\\
&(k-1)(\rho_k)_x\rightarrow0 \ \ \text{strongly in } L^2(0,1).
\end{split}\end{equation} Applying the same procedure as the proof
of Theorem \ref{thm1}, one can show that $\rho_{sup}$ satisfies
\eqref{weak-solution}. The lower bound of $\rho_{sup}$ follows from
\eqref{2.21} and the third convergence of \eqref{2.22}.

Let us now prove that $\rho_{sup}(x)<1$  for any interior point
$x\in(0,1)$. Observing that if a function $\rho$ satisfies
$\rho(x)\equiv1$ on an interval $[\hat{a},\hat{c}]\subset[0,1]$,
then $\rho$ is not a solution of equation \eqref{elliptic} because
$\underline{b}>1$. Thus, for any $1\gg\epsilon>0$, there exists a
$\delta>0$ and two points $\hat{a}_\epsilon\in(0,\epsilon]$ and
$\hat{c}_\epsilon\in[1-\epsilon,1)$ such that
$\rho_{sup}(\hat{a}_\epsilon),
\rho_{sup}(\hat{c}_\epsilon)\leq1-\delta<1$. We only need to show
that $\rho_{sup}(x)\leq1-\delta$ over
$[\hat{a}_\epsilon,\hat{c}_\epsilon]$. Actually, set
$w:=(1-\rho_{sup})^2$, then $w\in H_0^1(0,1)$, $w(\hat{a}_\epsilon),
w(\hat{c}_\epsilon)\geq\delta^2$ and it follows from
\eqref{weak-solution} that  for any $\varphi\in
H_0^1(\hat{a}_\epsilon,\hat{c}_\epsilon)$
\begin{equation*}
\frac{1}{2}\int_{\hat{a}_\epsilon}^{\hat{c}_\epsilon}\frac{2-\sqrt{w}}{(1-\sqrt{w})^3}w_x\varphi_xdx
+\frac{1}{\tau}\int_{\hat{a}_\epsilon}^{\hat{c}_\epsilon}\frac{\varphi_x}{1-\sqrt{w}}dx
+ \int_{\hat{a}_\epsilon}^{\hat{c}_\epsilon}(1-\sqrt{w}-b)\varphi
dx=0.
\end{equation*}
Taking $\varphi(x)=(w-\delta^2)^-(x)$, then
\begin{equation*}
\frac{1}{2}\int_{\hat{a}_\epsilon}^{\hat{c}_\epsilon}\frac{2-\sqrt{w}}{(1-\sqrt{w})^3}|[(w-\delta^2)^-]_x|^2dx
+\frac{1}{\tau}\int_{\hat{a}_\epsilon}^{\hat{c}_\epsilon}\frac{[(w-\delta^2)^-]_x}{1-\sqrt{w}}dx
+ \int_{\hat{a}_\epsilon}^{\hat{c}_\epsilon}(1-\sqrt{w}-b)(w-\delta^2)^-
dx=0.
\end{equation*}
Observing that $\rho_{sup}\geq\ell$, hence
$2-\sqrt{w}>1-\sqrt{w}\geq\ell>0$. This implies that the first term of
the equality is non-negative. Because $b>\underline{b}>1$, the third
term is also non-negative. On the other hand a simple computation
gives $-2(\sqrt{w}+\ln(1-\sqrt{w}))_x=\frac{w_x}{1-\sqrt{w}}$, which
implies the second term is zero. Thus, $(w-\delta^2)^-(x)=0$ over
$[\hat{a}_\epsilon,\hat{c}_\epsilon]$. And as a result,
$\rho_{sup}(x)\leq1-\delta$ over
$[\hat{a}_\epsilon,\hat{c}_\epsilon]$.

It is left to show (ii). We only need to show that if
$z_0\in(0,1)$ is a critical point of $\rho_{sup}$, then it must be a
local minimal point. Because $\rho_{sup}\in C[0,1]$ and
$\rho_{sup}<1$ over $(0,1)$, by the interior regularity theory of
elliptic equation and the Sobolev imbedding, for any $z_0\in(0,1)$,
there exists an interval $z_0\in I\subset(0,1)$ such that $z_0\in I$,
$\rho_{sup}\in W^{2,p}(I)$ for any $1<p<\infty$ and $\rho_{sup}\in
C^1(\overline{I})$. Now if $z_0$ is a critical point, then
$(\rho_{sup})_x(z_0)=0$. Since $\rho_{sup}\in C^1(\overline{I})$,
there exists a $\delta>0$ such that
\[|(\rho_{sup})_x(x)|<\frac{\tau(\underline{b}-1)}{2} \text{ for any
}x\in(z_0-\delta,z_0+\delta).
\]
If $x\in(z_0,z_0+\delta)$, we integrate \eqref{elliptic} over
$(z_0,x)$ to derive
\begin{equation*}
\begin{split}
\left(\frac{1}{\rho_{sup}}-\frac{1}{\rho_{sup}^3}\right)(\rho_{sup})_x
=&\int_{z_0}^x\left[\rho_{sup}-b+\frac{(\rho_{sup})_x}{\tau\rho^2}\right]ds\\
<&\int_{z_0}^x\left(1-\underline{b}+\frac{|\rho_x|}{\tau\rho^2}\right)ds\\
<&\int_{z_0}^x\left(1-\underline{b}+\frac{\underline{b}-1}{2}\right)ds\\
=&\frac{(1-\underline{b})(x-z_0)}{2}\\<&0,
\end{split}\end{equation*}
where we have used $(\rho_{sup})_x(z_0)=0$ and $\rho_{sup}<1$. Thus,
\[(\rho_{sup})_x(x)>0 \text{ on } (z_0,z_0+\delta).\]
Similarly, integrating \eqref{elliptic} over $(z_0-\delta,x)$, one
can get that
\[(\rho_{sup})_x(x)<0 \text{ on } (z_0-\delta,z_0).\]
Therefore, $z_0$ is a local minimal point of $\rho_{sup}$. The proof
is complete.
\end{proof}

As in Proposition \ref{local-analysis-sub}, we also study the
optimal global regularity of the interior supersonic solution.

\begin{proposition}\label{local-analysis-super}
$\rho_{sup}\in C^{1/2}[0,1]$, and there exist $s_2\ll 1$, $C_i$
$(i=5,6,7,8)$ such that
\begin{equation} \label{eqn-2.30}
\begin{split}
-C_5x^{1/2}&<\rho_{sup}-1<-C_6x^{1/2}, \\
-C_7x^{-1/2}&<(\rho_{sup})_x<-C_8x^{-1/2},
\end{split}
\mbox{ for } x \in [0,s_2].
\end{equation}

\end{proposition}
\begin{proof}The proof is similar to that of Proposition \ref{local-analysis-sub}.
Here for supersonic solutions, we need the local analysis for the
solution near $x=0$. We omit the details.\end{proof}

\subsection{Infinitely many transonic solutions with shocks}

 We turn to study the existence of
transonic solutions of \eqref{1.5}-\eqref{boundary}. We first
consider Euler-Poisson equations \eqref{1.5} with constant doping
profile $\underline{b}$ but without the semiconductor effect (namely
$\frac{1}{\tau}=0$, or say $\tau=\infty$), and the boundary
condition  subjected  is completely supersonic. That is
\begin{equation} \label{2.35}
\left \{\begin{array}{ll}
        \left(1-\dfrac{1}{\rho^2}\right)\rho_x=\rho E,\\
         E_x=\rho-\underline{b},\\
        \rho(0)=\rho(L)=1-\delta, \ \mbox{ (supersonic boundary),}
        \end{array} \right.
\end{equation}
where $L\geq\frac{1}{4}$ is the parameter of length and $\delta>0$
is a small constant. As shown in the proof of Theorem \ref{thm2},
for any $\delta>0$, \eqref{2.35} has a supersonic solution. We have
the following uniform estimates with respect to $\delta$ for the
supersonic solutions of \eqref{2.35}.

\begin{lemma}\label{lem4}
Assume that $\underline{b}>1$, and that $(\rho_L,E_L)(x)$ are
supersonic solutions of \eqref{2.35}. Then
\[\beta(L,\underline{b})\leq\underset{x\in[0,L]}{\min}\rho_L(x)\leq\gamma(L,\underline{b}),
\text{ and } E_L(0)\geq C(L,\underline{b}),\] where
$\beta(L,\underline{b})$, $\gamma(L,\underline{b})$
and $C(L,\underline{b})$ are positive constants independent
of $\delta$.
\end{lemma}

\begin{proof}
For convenience, we denote $(\rho_L,E_L)$ by $(\rho,E)$. In the
phase-plane $(\rho,\ET)$, we have
\[\frac{d\ET}{d\rho}=\frac{(\rho+1)(\rho-\underline{b})(\rho-1)}{
\ET\rho^3}.\]  Integrating the above equation with respect to
$\rho$, we obtain the part of trajectory through $(1-\delta,E(0))$
as follows
\begin{equation} \label{2.36}
\frac{E^2(x)}{2}=\frac{E^2(0)}{2}-\frac{2\rho(0)-\underline{b}}{2\rho^2(0)}-\rho(0)
+\underline{b}\ln\rho(0)+\frac{2\rho(x)-\underline{b}}{2\rho^2(x)}+\rho(x)-\underline{b}\ln\rho(x),
\end{equation}
and
\begin{equation*}
E(x)=\pm2\sqrt{\frac{E^2(0)}{2}-\frac{2\rho(0)-\underline{b}}{2\rho^2(0)}-\rho(0)
+\underline{b}\ln\rho(0)+\frac{2\rho(x)-\underline{b}}{2\rho^2(x)}+\rho(x)-\underline{b}\ln\rho(x)}.
\end{equation*}
Thus, all trajectories are symmetric with respect to $\textbf{E}\equiv0$ and
the supersonic solution obtained satisfies $0<\rho(x)<1-\delta$ and
is symmetric in $x\in(0,L)$. Set
$\underline{\rho}:=\underset{x\in[0,L]}{\min}\rho(x)$, by the
symmetry of $\rho(x)$ in $(0,L)$, we know  that $\rho(x)$ reaches
its minimum at $x=\frac{L}{2}$. Thus,
\begin{equation}\label{2.37}
\underline{\rho}=\rho(L/2) \text{ and }\rho'(L/2)=0.
\end{equation}
We next estimate $\underline{\rho}$. The velocity $u(x)=1/\rho(x)$
satisfies $u(x)\geq\frac{1}{1-\delta}$ and
\begin{equation}\label{2.38}
\left(\left(u-\dfrac{1}{u}\right)u_x\right)_x=\frac{1-\underline{b}u}{\lambda
u}, \ \ \ u(0)=u(L)=\frac{1}{1-\delta}.
\end{equation}
Multiplying \eqref{2.38} by $(u-\frac{1}{1-\delta})^2$, we get
\begin{equation}
2\int^L_0\Big(u-\frac{1}{u}\Big)\Big(u-\frac{1}{1-\delta}\Big)(u_x)^2dx
= \int^L_0\frac{\underline{b}u-1}{u}\Big(u-\frac{1}{1-\delta}\Big)^2dx.
\label{2.39}
\end{equation}
Artfully, we can reduce the left-hand-side of \eqref{2.39} to
\begin{eqnarray}
&&
2\int^L_0\Big(u-\frac{1}{u}\Big)\Big(u-\frac{1}{1-\delta}\Big)(u_x)^2dx \notag \\
&&=2\int^L_0\frac{u+1}{u}(u-1)\Big(u-\frac{1}{1-\delta}\Big)(u_x)^2dx \notag \\
&&=2\int^L_0\frac{u+1}{u}\Big(\frac{\delta}{1-\delta} +u-\frac{1}{1-\delta}\Big)\Big(u-\frac{1}{1-\delta}\Big)(u_x)^2dx \notag \\
&&=\frac{2\delta}{1-\delta}\int^L_0\frac{u+1}{u}\Big(u-\frac{1}{1-\delta}\Big)(u_x)^2dx \notag \\
&& \ \ + 2\int^L_0\frac{u+1}{u}\Big(u-\frac{1}{1-\delta}\Big)^2(u_x)^2dx \notag \\
&&=\frac{2\delta}{1-\delta}\int^L_0\frac{u+1}{u}\Big(u-\frac{1}{1-\delta}\Big)(u_x)^2dx \notag \\
&& \ \ +
\frac{1}{2}\int^L_0\frac{u+1}{u}\Big|\Big(\Big(u-\frac{1}{1-\delta}\Big)^2\Big)_x\Big|^2dx,
\label{2.40}
\end{eqnarray}
and by using Cauchy-Schwarz's inequality and  Poincar\'e's inequality,
we can estimate the right-hand-side of \eqref{2.39} as follows
\begin{eqnarray}
&& \int^L_0\frac{\underline{b}u-1}{u}\Big(u-\frac{1}{1-\delta}\Big)^2dx \notag \\
&&\le \frac{1}{2L^2} \int^L_0\Big(u-\frac{1}{1-\delta}\Big)^4dx +\frac{\underline{b}^2L^3}{2} \notag \\
&&\le \frac{1}{4}
\int^L_0\Big|\Big(\Big(u-\frac{1}{1-\delta}\Big)^2\Big)_x\Big|^2 dx
+\frac{\underline{b}^2L^3}{2}. \label{2.41}
\end{eqnarray}
Substituting \eqref{2.40} and \eqref{2.41} to \eqref{2.39}, we then
have
\begin{equation*}
\frac{2\delta}{1-\delta}\int_0^L\frac{u+1}{u}\Big(u-\frac{1}{1-\delta}\Big)u_x^2dx
+\int_0^L\frac{u+2}{4u}\left|\Big(\Big(u-\frac{1}{1-\delta}\Big)^2\Big)_x\right|^2
\le \frac{\underline{b}^2L^3}{2},
\end{equation*}
which gives
\begin{equation}\label{2.42}
\big\|\big(\big(u-\frac{1}{1-\delta}\big)^2\big)_x\big\|_{L^2(0,L)}\leq
\underline{b}L\sqrt{2L}.
\end{equation}
Notice that, for $\phi\in H^1_0(0,L)$,  Sobolev's inequality
\[
\|\phi\|_{L^\infty}\le
\sqrt{2}\|\phi\|_{L^2}^{1/2}\|\phi_x\|_{L^2}^{1/2}
\]
and Poincar\'e's inequality
\[
\|\phi\|_{L^2}\le 2L\|\phi_x\|_{L^2}
\]
imply
\[
\|\phi\|_{L^\infty}\le 2\sqrt{L}\|\phi_x\|_{L^2}.
\]
Thus, from \eqref{2.42} we have
\[
\Big(u(x)-\frac{1}{1-\delta}\Big)^2\le
2\sqrt{L}\big\|\big(\big(u-\frac{1}{1-\delta}\big)^2\big)_x\big\|_{L^2(0,L)}\le2\sqrt{2}\underline{b}L^2,
\]
which gives
\[
u(x)\leq\frac{1}{1-\delta}+\sqrt{2\sqrt{2}\underline{b}}\cdot
L.
\]
Thus, we can estimate the minimum of $\rho(x)$ by
\begin{equation*}
\underline{\rho}\geq\left(\frac{1}{1-\delta}+\sqrt{2\sqrt{2}\underline{b} }\cdot
L\right)^{-1}
\geq\left(2+\sqrt{2\sqrt{2}\underline{b} }\cdot
L\right)^{-1} \triangleq\beta(L), \text{ when
}\delta\leq\frac{1}{2}.
\end{equation*}
On the other hand, by \eqref{2.35}, since $\underline{b}\geq1>\rho$,
we have
\begin{equation}\label{2.43}
\rho_{xx}=\frac{\rho^3}{\rho+1}\left[\frac{1}{\rho^2(1-\rho)}\left(\frac{3}{\rho^2}-1\right)\rho_x^2
+\frac{\underline{b}-\rho}{ (1-\rho)}\right]\geq\frac{\underline{\rho}^3}{2}\geq\frac{\beta^3(L)}{2}
\text{ on } [0,L].
\end{equation}
By Taylor expansion
\[\rho(0)=\rho(L/2)-\rho'(L/2)L/2+\rho^{''}(\xi)(L/2)^2/2 \
\text{ with }\xi\in[0,L/2],
\]
it then follows from \eqref{2.37} and \eqref{2.43} that
\begin{equation}
\underline{\rho}\leq1-\delta-\frac{L^2\beta^3(L)}{2^4}
\leq1-\frac{L^2}{2^4}\cdot\frac{1}{(2+\sqrt{2\sqrt{2}\underline{b} }\cdot
L)^3} \triangleq\gamma(L). \label{2.44}
\end{equation}
Since $\underline{\rho}=\rho(L/2)$ is the minimum value, from
\eqref{2.35} and the fact $\rho_x(L/2)=0$, we have $E(L/2)=0$. Thus,
in view of \eqref{2.36}, we further obtain
\begin{equation*}
\begin{split}
\frac{E^2(0)}{2}&=\frac{2-\underline{b}-2\delta}{2(1-\delta)^2}+1-\delta
-\underline{b}\ln(1-\delta)-\Big[\frac{2\underline{\rho}-\underline{b}}{2\underline{\rho}^2}
+\underline{\rho}-\underline{b}\ln\underline{\rho}\Big]\\
&=\frac{\delta[2-2\underline{b}-(2-\underline{b})\delta]}{2(1-\delta)^2}-\delta
-\underline{b}\ln(1-\delta)+1\\&\quad+\frac{(2-\underline{b})(\underline{\rho}-1)^2
+(2-2\underline{b})(\underline{\rho}-1)}{2\underline{\rho}^2}
-\underline{\rho}+\underline{b}\ln\underline{\rho}\\
&\geq\frac{\delta[2-2\underline{b}-(2-\underline{b})\delta]}{2(1-\delta)^2}-\delta+f(\underline{\rho}),
\end{split}\end{equation*}
where
\[
f(s): =1+\frac{(2-\underline{b})(s-1)^2
+(2-2\underline{b})(s-1)}{2s^2}-s+\underline{b}\ln s, \ \ s\in(0,1).
\]
Notice that $f(1)=0$ and
$f'(s):=-\frac{(\underline{b}-s)(1-s^2)}{s^3}<0$ for $s\in(0,1)$,
namely, $f(s)$ is decreasing and positive for $s\in(0,1)$, using the
boundness estimates carried out in \eqref{2.44}:
$\underline{\rho}\le \gamma(L)$, we have, when $\delta$ is small
such that
$\frac{\delta^2}{2(1-\delta)^2}+\delta\leq\frac{f(\gamma(L))}{2}$,
then
\begin{equation}\label{2.45}
E^2(0)\geq
2\left[\frac{-\delta^2}{2(1-\delta)^2}-\delta+f(\underline{\rho})\right]
\geq
2\left[\frac{-\delta^2}{2(1-\delta)^2}-\delta+f(\gamma(L))\right]
\geq f(\gamma(L)).
\end{equation}
Integrating the second equation of \eqref{2.35} over $[0,L/2]$, we
get
\begin{equation*}
E(0)=E(L/2)+ \int_0^{L/2}(1-\rho)dt=\int_0^{L/2}(1-\rho)dt>0.\end{equation*}
Hence, it follows from \eqref{2.45} that, $E(0)$ has a positive
lower bound
\begin{equation}\label{2.46}
E(0)\geq \sqrt{f(\gamma(L)) },
\end{equation}
which is independent of $\delta$.
\end{proof}

%

\begin{theorem}\label{thm3}
If $\underline{b}>1$, $\tau\gg1$ and $0\leq\bar{b}-\underline{b}\ll1$, then system \eqref{1.5}-\eqref{boundary} has infinitely many
transonic shock solutions over $[0,1]$.
\end{theorem}

\begin{proof}
The proof is technical and longer, we divide it into seven steps.

\emph{Step 1.} Let $\eta$ be a small number to be determined later
such that $\delta<\eta\ll1$. Denote by $(\rho_1,E_1)(x)$ the
solution of \eqref{2.35} with $L=\frac{1}{2}$. Then by \eqref{2.46},
\begin{equation}\label{E10}
E_1(0)\geq \sqrt{f(\gamma(1/2)) }\triangleq\Lambda_1.\end{equation}

Let us consider system \eqref{1.5} with the supersonic initial
value:
\begin{equation} \label{2.47}
\left \{\begin{array}{ll}
        \left(1-\dfrac{1}{\rho^2}\right)\rho_x=\rho E-\dfrac{1}{\tau},\\
         E_x=\rho-b(x),\\
        (\rho(0),E(0))=(1-\delta,E_1(0)).
        \end{array} \right.
\end{equation}
In this step, we will show that when $\tau\gg1$, there exists a
number $x_1\leq C\eta$ such that $\rho(x_1)=1-\eta$, and $E(x_1)\geq
E_1(0)-C\eta^2$, where $C>0$ is a constant independent of $\tau$,
$\delta$ and $\eta$.

It is easy to see that if $\tau\geq\frac{4}{\Lambda_1}\geq\frac{4}{E_1(0)}$ and
$\delta\leq\frac{1}{4}$, then the initial data of \eqref{2.47}
satisfies
\[
\rho(0)E(0)-\frac{1}{\tau}=(1-\delta)E_1(0)-\frac{1}{\tau}\geq\frac{E_1(0)}{2}>0.
\]
Observing that \eqref{2.47} is a standard initial value problem for
ODE system without degeneracy, it follows that \eqref{2.47} has a
unique supersonic solution on some interval. Because $b\geq
\underline{b}>1>\rho$, the solution component $E$ keeps decreasing.
Using the result (ii) of Theorem \ref{thm2},  $\rho$ is decreasing
until it attains the unique critical point, after that $\rho$ keeps
increasing. Denote by $x_1$ the first number that $\rho(x)$ attains
$1-\eta$, namely $\rho(x_1)=1-\eta$. By the second equation of
\eqref{2.47},
\[ E(x)= E_1(0)+\int_0^x(\rho-b)ds\geq E_1(0)-\bar{b}x \text{ for } x\in(0,x_1).
\]
Since $\rho\in(1-\eta,1-\delta)$ on $(0,x_1)$, if
$\eta\leq\frac{1}{2}$ and $\tau\geq\frac{4}{\Lambda_1}\geq\frac{4}{E_1(0)}$, then
\[\rho
E-\frac{1}{\tau}\geq(1-\eta)(E_1(0)-\bar{b}x)-\frac{E_1(0)}{4}\geq(1-\eta)\left(\frac{E_1(0)}{4}-\bar{b}x\right)
\ \text{ for }x\in(0,x_1),
\]
which in combination with the first equation of \eqref{2.47} leads
to
\begin{equation}\label{new-2.49}
x_1=\frac{\rho(x_1)-\rho(0)}{\rho_x(\xi)}=\frac{(\eta-\delta)(1-\rho^2(\xi))}{(\rho(\xi) E(\xi)-\frac{1}{\tau})\rho^2(\xi)}
\leq\frac{2\eta^2}{(1-\eta)^3(\frac{E_1(0)}{4}-\bar{b}x_1)} \ \ \text{
with }\xi\in(0,x_1).
\end{equation}
To solve this inequality, we notice that when $\eta$ is small such that $\eta\leq\min\big\{\frac{ E_1(0)}{2^4\sqrt{b}},\frac{1}{2}\big\}$, then
\[\frac{E_1^2(0)}{4}-\frac{8\bar{b}\eta^2}{(1-\eta)^3}\geq\frac{E_1^2(0)}{4}-2^6\bar{b}\eta^2\geq\frac{1}{4}(E_1^2(0)-2^8\bar{b}\eta^2)\geq0.\]
Thus, we get from inequality \eqref{new-2.49} that
\begin{equation*}
\begin{split}
x_1\leq&\frac{1}{2\bar{b}}\left(\frac{E_1(0)}{4}
-\left(\frac{E_1^2(0)}{4^2}-\frac{8\bar{b}\eta^2}{(1-\eta)^3}\right)^{1/2}\right)
\\=&\frac{4\eta^2}{(1-\eta)^3\left(\frac{E_1(0)}{4}+(\frac{E_1^2(0)}{4^2}
-\frac{8\bar{b}\eta^2}{(1-\eta)^3})^{1/2}\right)}\\
\leq&\frac{16\eta^2}{(1-\eta)^3E_1(0)} \leq\frac{2^7\eta^2}{\Lambda_1},
\end{split}\end{equation*}
where we have used \eqref{E10} in the last inequality. In view of the second equation of \eqref{2.47}, we further get
\begin{equation}\label{2.49}
E(x_1)=E_1(0)+ \int_0^{x_1}(\rho-b)ds\geq
E_1(0)-\bar{b}x_1\geq E_1(0)-\frac{2^7\bar{b}\eta^2}{\Lambda_1}.
\end{equation}

\emph{Step 2.} Now let us consider the initial value problem for the ODE system without
semiconductor effect
\begin{equation} \label{2.50}
\left \{\begin{array}{ll}
        \left(1-\dfrac{1}{\hat{\rho}^2}\right)\hat{\rho}_x=\hat{\rho} \hat{E},\\
         \hat{E}_x=\hat{\rho}-\underline{b},\\
        (\hat{\rho}(0),\hat{E}(0))=(1-\delta,\hat{E}_0).
        \end{array} \right.
\end{equation}
In this step, we prove that there exist numbers $x_2>0$
and $\hat{E}_0>0$ such that  $x_2\leq C\eta^2$ and the solution of \eqref{2.50} satisfies
$\hat{\rho}(x_2)=1-\eta$ and $\hat{E}(x_2)=E(x_1)$. Here $E$ and
$x_1$ are given by step 1, and $C>0$ is a constant independent of $\tau$, $\delta$ and $\eta$.

We argue by shooting method. Using phase-plane analysis, it is easy
to see that, for any $\hat{E}_0>0$, there exists $\hat{L}(\hat{E}_0)>0$, such that
\eqref{2.50} has a symmetric supersonic solution on $[0,\hat{L}(\hat{E}_0)]$
satisfying
\[\hat{\rho}(0)=\hat{\rho}(\hat{L}(\hat{E}_0))=1-\delta, \ \hat{E}(0)=\hat{E}(\hat{L}(\hat{E}_0))=\hat{E}_0.\]

Now taking $\hat{E}_0=2E_1(0)$, suppose $\bar{x}_2$ is the first
number that $\hat{\rho}$ attains $1-\eta$, since
$\hat{\rho}\in(1-\eta,1-\delta)$ on $(0,\bar{x}_2)$, by the second
equation of \eqref{2.50},
\begin{equation} \label{2.51}
\hat{E}(x)=\hat{E}(0)+ \int_0^{x}(\hat{\rho}-b)ds\geq
2E_1(0)-\bar{b}x \text{ for
}x\in(0,\bar{x}_2).\end{equation} Hence
\begin{equation*}
\hat{\rho}\hat{E}(x)\geq(1-\eta)(2E_1(0)-\bar{b}x),\end{equation*} which in combination with the first
equation of \eqref{2.50} leads to
\begin{equation}\label{new-2.53}
\bar{x}_2=\frac{\hat{\rho}(\bar{x}_2)-\hat{\rho}(0)}{ \hat{\rho}_x(\hat{\xi})}=\frac{(\eta-\delta)(1-\hat{\rho}^2(\hat{\xi}))}{\hat{\rho}^3(\hat{\xi})\hat{E}(\hat{\xi})}
\leq\frac{2\eta^2}{(1-\eta)^3(2E_1(0)- \bar{b}\bar{x}_2 )}.\end{equation}
Notice that when $\eta\leq\min\big\{\frac{E_1(0)}{4\sqrt{\bar{b}}},\frac{1}{2}\big\}$, it holds that \[4E_1^2(0)-\frac{8\bar{b}\eta^2}{(1-\eta)^3}\geq4(E_1^2(0)-2^4\bar{b}\eta^2)\geq0.\]
It then follows from \eqref{new-2.53} that
\begin{equation}\label{new-2.54}
\begin{split}
\bar{x}_2\leq&\frac{1}{2\bar{b}}\left(2E_1(0)
-\left(4E_1^2(0)-\frac{8\bar{b}\eta^2}{(1-\eta)^3}\right)^{1/2}\right)
\\=&\frac{2\eta^2}{(1-\eta)^3\left(E_1(0)+(E_1^2(0)
-\frac{2\bar{b}\eta^2}{(1-\eta)^3})^{1/2}\right)}\\
\leq&\frac{2\eta^2}{(1-\eta)^3E_1(0)}
\leq\frac{2^4\eta^2}{\Lambda_1},
\end{split}\end{equation}where we have used \eqref{E10} in the last inequality.
This inequality together with \eqref{2.51} gives
\begin{equation*}
\hat{E}(\bar{x}_2)\geq 2E_1(0)-\bar{b}
\bar{x}_2\geq2E_1(0)-\frac{2^4\bar{b}\eta^2}{E_1(0)}\geq
2E_1(0)-\frac{2^4\bar{b}\eta^2}{E_1(0)}\cdot\frac{E_1^2(0)}{16\bar{b}}=E_1(0)>E(x_1).\end{equation*}
Here we have used $E_1(0)=E(0)>E(x_1)$ because $E$ is decreasing.

On the other hand, if $\hat{E}_0=\frac{E_1(0)}{2}$, by \eqref{2.49}, one can easily see
that if $\eta<\frac{\Lambda_1}{2^4\bar{b}}$, it holds
$E(x_1)>\frac{E_1(0)}{2}$. Thus,
$\hat{E}(x)<\hat{E}_0=\frac{E_1(0)}{2}<E(x_1)$ for any $x>0$ because
$\hat{E}$ is decreasing. Now by the continuity of the solution with
respect to the initial data, there exist
$\hat{E}_0\in(\frac{E_1(0)}{2},2E_1(0))$ and length $\hat{L}>0$ such
that \eqref{2.50} has a supersonic solution $(\hat{\rho},\hat{E})$
satisfying
\[\hat{\rho}(0)=\hat{\rho}(\hat{L})=1-\delta, \ \hat{E}(0)=\hat{E}(\hat{L})=\hat{E}_0.\]
Moreover, as in \eqref{new-2.54}, there exists a number $x_2\leq
C\eta^2$ such that
\begin{equation}\label{2.52}
\hat{\rho}(x_2)=1-\eta \ \text{ and } \
\hat{E}(x_2)=E(x_1).\end{equation} Thus,
\[0<\hat{E}_0-E(x_1)=\hat{E}_0-\hat{E}(x_2)=-\hat{E}_xx_2=(b-\hat{\rho})x_2<\bar{b}x_2<C_1\eta^2,\]
which in combination with \eqref{2.49} yields
\[\begin{split}
\hat{E}_0-E_1(0)&=\hat{E}_0-E(x_1)+E(x_1)-E_1(0)>E(x_1)-E_1(0)>-C_2\eta^2,\\
\hat{E}_0-E_1(0)&=\hat{E}_0-E(x_1)+E(x_1)-E_1(0)<C_1\eta^2,
\end{split}\]
where we have used the fact that $E$ is decreasing. Thus
\[|\hat{E}_0-E_1(0)|\leq C_3\eta^2 \text{ with }C_3=\min\{C_1,C_2\}.\]
Observing that the length $\hat{L}$ of solution is also continuous
with respect to the initial data, since the length of the solution
$(\rho_1,E_1)$ to \eqref{2.35} with initial data $(1-\delta,
E_1(0))$ is $\frac{1}{2}$, there exists $l_0>0$ independent of
$\tau$, $\delta$ and $\eta$, such that if $C_3\eta^2<l_0$, then
\[\frac{1}{4}\leq\hat{L}\leq\frac{3}{4}.\]

\emph{Step 3.} In this step, we show that when $\tau\gg1$ and
$\bar{b}-\underline{b}\ll1$, system \eqref{2.47} has a unique
solution $(\rho,E)$ on $[0,x_4]$ with
\[\frac{1}{4}-C\eta^2\leq x_4\leq\frac{3}{4}+C\eta^2, \
\rho(0)=\rho(x_4)=1-\delta,
\] for some constant $C$ independent of
$\tau$, $\delta$ and $\eta$. Set
$(\bar{\rho},\bar{E})(x):=(\hat{\rho},\hat{E})(x-x_1+x_2)$, then
$(\bar{\rho},\bar{E})$ satisfies \eqref{2.50} with initial-boundary
data
\[(\bar{\rho},\bar{E})(x_1)=(1-\eta,\hat{E}(x_2))=(\rho,E)(x_1) \ \text { and } \bar{\rho}(x_3)=1-\eta \]
with $x_3:=\hat{L}+x_1-2x_2$, where we have used the symmetry of
$(\hat{\rho},\hat{E})$, and hence
$\hat{\rho}(\hat{L}-x_2)=\hat{\rho}(x_2)=1-\eta$. Set
$\phi:=\bar{\rho}-\rho$, $\psi:=\bar{E}-E$, then by \eqref{2.47} and
\eqref{2.50}, $(\phi,\psi)$ satisfies
\begin{equation} \label{2.53}
\left \{\begin{array}{ll}
        \phi_x=\frac{\bar{\rho}^3\psi}{(\bar{\rho}+1)(\bar{\rho}-1)}
        +\frac{(\bar{\rho}^2\rho^2-\bar{\rho}^2-\bar{\rho}\rho-\rho^2)\phi E}{(\bar{\rho}+1)(\bar{\rho}-1)(\rho+1)(\rho-1)}
        +\frac{\rho^2}{\tau(\rho+1)(\rho-1)},\\
         \psi_x=\phi+b-\underline{b},\\
        (\phi(x_1),\psi(x_1))=0.
        \end{array} \right.
\end{equation}
Define the solution space $X_T:=\{(\phi,\psi)\in
C[x_1,T]|\phi(x_1)=\psi(x_1)=0,|\phi|\leq\eta/2,|\psi|\leq\eta/2\}$.
We only need to show the a priori estimate
\begin{equation} \label{2.54}
\phi^2(x)+ \psi^2(x)\leq\eta^2/4 \ \text{ on }x\in[x_1,x_3].
\end{equation}
Multiplying the first equation of \eqref{2.53} by $\phi$ and  the
second one by $\psi$ and adding them, noting
$|\rho-\bar{\rho}|\leq\eta/2$, by Young's inequality, one can easily
get
\[(\phi^2+ \psi^2)_x\leq\frac{C}{\eta^2}(\phi^2+
\psi^2)+\frac{C}{\tau^2}+C(\bar{b}-\underline{b})^2,
\] where $C$ is a constant independent of
$\tau$, $\delta$ and $\eta$. It then follows from Gronwall's
inequality that
\[\phi^2+ \psi^2\leq C\left[\frac{C}{\tau^2}+(\bar{b}-\underline{b})^2\right]\eta^2e^{Cx/\eta^2}
\leq
C\left[\frac{C}{\tau^2}+(\bar{b}-\underline{b})^2\right]\eta^2e^{C/\eta^2}
\text{ for }x\in[x_1,x_3].
\]
Now taking $\tau\gg1$ and $\bar{b}-\underline{b}\ll1$ such that
$\left[\frac{C}{\tau^2}+(\bar{b}-\underline{b})^2\right]e^{C/\eta^2}\leq\frac{1}{4}$,
we derive \eqref{2.54}. Moreover, we also get
\[|\rho-\bar{\rho}|\leq\eta/2 \text{ and }  |E-\bar{E}|\leq\eta/2,
\]
which gives
$\rho(x_3)\leq\bar{\rho}(x_3)+\frac{\eta}{2}=1-\frac{\eta}{2}$, and
further by \eqref{2.52} and \eqref{2.49},
\begin{equation}\label{2.55}
\begin{split}
E(x_3)\leq\bar{E}(x_3)+\frac{\eta}{2 }
=\hat{E}(\hat{L}-x_2)+\frac{\eta}{2
}=-\hat{E}(x_2)+\frac{\eta}{2}=-E(x_1)+\frac{\eta}{2} \leq
-E_1(0)+C\eta.
\end{split}\end{equation}

Now taking $x_3$ as the initial data, we can extend $(\rho,E)$, the
solution of \eqref{2.47}, to the state satisfying $\rho=1-\delta$.
Denote by $x_4$ the number that $\rho(x_4)=1-\delta$. As in the
proof of step 2, we have
\[x_4-x_3\leq C\eta^2
\] for some constant $C$  independent of
$\tau$, $\delta$ and $\eta$. And then by \eqref{2.55},
\[E(x_4)\leq E(x_3)\leq -E_1(0)+C\eta.\]
Now we obtain a solution of \eqref{2.47} on $[0,x_4]$ satisfying
\begin{equation}\label{2.56}
\rho(0)=\rho(x_4)=1-\delta,\ E(0)=E_1(0),\ E(x_4)\leq -E_1(0)+C\eta.
\end{equation}
Moreover,
\begin{equation}\label{2.60}
\frac{1}{4}-C\eta^2\leq\hat{L}+x_1-2x_2 =x_3\leq x_4\leq
x_3+C\eta^2=\hat{L}+x_1-2x_2+C\eta^2\leq\frac{3}{4}+C\eta^2.
\end{equation}

\emph{Step 4.} In this step, we construct a transonic solution of
\eqref{2.47} on an interval $[0,x_5]$ with
\[\frac{1}{4}-C\eta\leq x_5\leq\frac{3}{4}+C\eta, \ \rho(0)=1-\delta,\ \rho(x_5)=1+\delta.
\]

Set $\rho_l=1-\eta$, then $\rho_r=1/\rho_l>1$. We take the jump
location $\bar{x}_0\in (0,x_4)$ as the last number that
$\rho(\bar{x}_0)=\rho_l$, and restrict our supersonic solution
$(\rho_{sup},E_{sup})(x)$  only on $[0,\bar{x}_0]$. We denote
$E_{sup}(\bar{x}_0)\triangleq E_l$. As in the proof of step 2,
\begin{equation}\label{2.61}
x_4-\bar{x}_0\leq C\eta.\end{equation} Thus, owing to the inequality
of \eqref{2.56}, the supersonic solution satisfies
\begin{equation}\label{2.62}
\rho_l=1-\eta,\ E_l\leq E(x_4)+C\eta\leq-E_1(0)+C\eta.\end{equation}
It is then easy to see that
\[\rho_lE_l-\frac{1}{\tau}\leq(1-\eta)(-E_1(0)+C\eta)=-E_1(0)+(C+E_1(0))\eta.\]
Thus, when $\eta\ll1$ such that
$(C+E_1(0))\eta\leq\frac{E_1(0)}{2}$, it holds
\begin{equation}\label{2.57}
\rho_lE_l-\frac{1}{\tau}\leq-\frac{E_1(0)}{2}<0 \ \text{ and } \
E_l<0.\end{equation}

Next we construct the corresponding subsonic solution. For
$x\geq\bar{x}_0$, let us consider the system \eqref{2.47} with the
initial data
\[\rho(\bar{x}_0)=\rho_r, \ \ E(\bar{x}_0)=E_r=E_l.\] By the standard ODE theory, the initial value problem
admits a unique solution $(\rho,E)(x)$ for $x>\bar{x}_0$. By
\eqref{2.57}, a simple calculation gives
\begin{equation*}
\begin{split}
\rho_rE_r-\frac{1}{\tau}&=\rho_lE_l-\frac{1}{\tau}+(\frac{1}{\rho_l}-\rho_l)E_l\\
&\leq-\frac{E_1(0)}{2}+\Big[\frac{1}{1-\eta}-(1-\eta)\Big](-E_1(0)+C\eta)\\
&\leq-\frac{E_1(0)}{2}+C\eta^2.
\end{split}\end{equation*}
It hence follows that when $C\eta^2<\frac{E_1(0)}{4}$,
\[\rho_rE_r-\frac{1}{\tau}\leq-\frac{E_1(0)}{4}<0.\]
From the first equation of \eqref{2.47}, we know the component
$\rho$ of such solution is decreasing in a neighborhood of
$\bar{x}_0^+$. We denote this subsonic solution by
$(\rho_{sub},E_{sub})(x)$. If $\eta<1-\frac{1}{b}$, then
\begin{equation*}
\begin{split}
E_{sub}(x)&=E_r+ \int_{\bar{x}_0}^{x}(\rho_{sub}-b)dx\\&\leq E_r+
\int_{\bar{x}_0}^{x}\Big(\frac{1}{1-\eta}-b\Big)dx\\&<E_r<0,\end{split}\end{equation*}
where we have used the second inequality of \eqref{2.57} and
$E_r=E_l$. Noting that the function $g(s)=\frac{s^3}{s^2-1}$ is
monotone decreasing on $(1,\sqrt{3})$, we thus get from \eqref{2.62}
that
\[(\rho_{sub})_x=\frac{\rho_{sub}E_{sub}-\frac{1}{\tau}}{1-\frac{1}{\rho_{sub}^2}}
\leq\frac{\rho_r^3E_r}{\rho_r^2-1}=\frac{E_r}{\eta(1-\eta)(2-\eta)}
\leq\frac{-E_1(0)+C\eta}{\eta(1-\eta)(2-\eta)}<-\frac{E_1(0)}{2},\]
if $\eta<\min\left\{\frac{E_1(0)}{2C},\frac{1}{2}\right\}$. This
inequality implies $\rho_{sub}$ will keep decreasing and attains
$1+\delta$ at a finite number $x_5$ and
\begin{equation}\label{2.64}
x_5-\bar{x}_0=\frac{\delta-\frac{\eta}{1-\eta}}{\int_0^1(\rho_{sub})_x(sx_5+(1-s)\bar{x}_0)ds}
\leq C\eta^2 \text{ if
}\eta<\min\left\{\frac{E_1(0)}{2C},\frac{1}{2}\right\}.\end{equation}
Now we have constructed the transonic solution to \eqref{2.47} in
$[0,x_5]$ as follows
\[
(\rho_{trans},E_{trans})(x)=
\begin{cases}
(\rho_{sup},E_{sup})(x), & x\in[0,\bar{x}_0),\\
(\rho_{sub},E_{sub})(x), & x\in(\bar{x}_0,x_5],
\end{cases}
\]
which satisfies the boundary condition
\[\rho_{sup}(0)=1-\delta, \ \rho_{sub}(x_5)=1+\delta,\]
and the entropy condition at $\bar{x}_0$
\[
0<\rho_{sup}(\bar{x}_0^-)=1-\eta<1<\rho_{sub}(\bar{x}_0^+),
\]
and the Rankine-Hugoniot condition \eqref{RH} at $\bar{x}_0$.
Furthermore, it follows from \eqref{2.60}, \eqref{2.61} and
\eqref{2.64} that
$$\frac{1}{4}-C\eta\leq x_5\leq\frac{3}{4}+C\eta.$$

\emph{Step 5.} In this step, we construct a transonic solution of
\eqref{2.47} on an interval $[0,x_7]$ with $\frac{5}{4}-C\eta\leq
x_7\leq\frac{7}{4}+C\eta$, $\rho(0)=1-\delta$ and
$\rho(x_7)=1+\delta$.

We take $L=\frac{3}{2}$ in \eqref{2.35} and denote by $(\rho_2,E_2)$
its solution. As shown in steps 1-3, we know that there exists an
interval $[0,x_6]$ with
\[\frac{5}{4}-C\eta^2\leq x_6\leq\frac{7}{4}+C\eta^2,
\]
such that system \eqref{2.47} has a supersonic solution on
$[0,x_6]$ satisfying
\begin{equation*}
\rho(0)=\rho(x_6)=1-\delta,\ E(0)=E_2(0),\ E(x_6)\leq -E_2(0)+C\eta.
\end{equation*}
As in step 4, we may construct another transonic solution for
\eqref{2.47} in the form of
\[
(\rho_{trans},E_{trans})(x)=
\begin{cases}
(\rho_{sup},E_{sup})(x), & x\in[0,\tilde{x}_0),\\
(\rho_{sub},E_{sub})(x), & x\in(\tilde{x}_0,x_7],
\end{cases}
\]
where $\tilde{x}_0\in (0,x_6)$ and $\frac{5}{4}-C\eta^2\leq
x_7\leq\frac{7}{4}+C\eta^2$ are some determined numbers. This
transonic solution satisfies the boundary condition
\[\rho_{sup}(0)=1-\delta, \ \rho_{sub}(x_7)=1+\delta,\]
the entropy condition at $\tilde{x}_0$
\[
0<\rho_{sup}(\tilde{x}_0^-)=1-\eta<1<\rho_{sub}(\tilde{x}_0^+),
\]
and the Rankine-Hugoniot condition \eqref{RH} at  $\tilde{x}_0$.

\emph{Step 6.} We next construct transonic solutions  of
\eqref{2.47} on $[0,1]$. Without loss of generality, we assume that
$E_1(0)<E_2(0)$. As in step 4, one can see that when
$0<\delta<\eta\ll1$, $\tau\gg1$, for any $E_0\in(E_1(0),E_2(0))$,
there exists a number $x_8>0$ and a transonic solution of
\eqref{2.47} on the interval $[0,x_8]$ satisfying the boundary
condition
\[\rho_{sup}(0)=1-\delta, \ E_{sup}(0)=E_0, \ \rho_{sub}(x_8)=1+\delta,\]
the entropy condition at $\tilde{\bar{x}}_0$
\[
0<\rho_{sup}(\tilde{\bar{x}}_0^-)=1-\eta<1<\rho_{sub}(\tilde{\bar{x}}_0^+),
\]
and the Rankine-Hugoniot condition. Applying the continuation
argument in the length of the interval $L$, we realize that
\eqref{2.47} has some transonic solutions
$(\rho_{trans},E_{trans})(x)$ for $x\in [0,1]$ and satisfies the
boundary condition
\[\rho_{sup}(0)=1-\delta, \ \rho_{sub}(1)=1+\delta,\] the entropy condition
\[
0<\rho_{sup}(x_0^\delta)=1-\eta<1<\rho_{sub}(x_0^\delta),
\]
and the Rankine-Hugoniot condition at some jump location
$x_0^\delta$ in $(0,1)$.

\emph{Step 7.} Let us now prove the existence of transonic solutions
of \eqref{1.5}-\eqref{boundary} on $[0,1]$.

For any $\delta>0$, denote by $(\rho^\delta,E^\delta)$ the transonic
solution of \eqref{2.47} on $[0,1]$ obtained in step 6. Multiplying
the first equation of \eqref{2.47} by
$((1-\delta-\rho^\delta)^2)_x$, integrating the resultant equation
on $(0,x_0^\delta)$, and using the second equation of \eqref{2.47},
noting
\begin{equation*}
\frac{((1-\delta-\rho^\delta)^2)_x}{\rho^\delta}=(-2(1-\delta)\ln\rho^\delta+2\rho^\delta)_x,
\end{equation*}
\begin{equation*}
\begin{split}
\int_0^{x_0^\delta}(b-\rho^\delta)(1-\delta-\rho^\delta)^2dx&\leq\int_0^{x_0^\delta}b(1-\delta-\rho^\delta)^2dx\\&
\leq\frac{1}{4}\int_0^{x_0^\delta}(1-\delta-\rho^\delta)^4dx+ \int_0^{x_0^\delta}b^2dx\\&
\leq\frac{1}{4}\int_0^{x_0^\delta}|((1-\delta-\rho^\delta)^2)_x|^2dx+b^2,
\end{split}\end{equation*}
we have
\begin{equation}\label{2.58}
\begin{split}
&\int_0^{x_0^\delta}\frac{2\delta(\rho^\delta+1)(1-\delta-\rho^\delta)(\rho_x)^2}{(\rho^\delta)^3}
+\frac{(\rho^\delta+1)|((1-\delta-\rho^\delta)^2)_x|^2}{2(\rho^\delta)^3}dx\\
&\leq\frac{1}{4}\int_0^{x_0^\delta}|((1-\delta-\rho^\delta)^2)_x|^2dx
+b^2+E_l(1-\delta-\rho_l)^2\\&\quad
-\frac{2}{\tau}[-(1-\delta)\ln\rho_l+\rho_l+(1-\delta)\ln(1-\delta)-(1-\delta)].
\end{split}\end{equation}
Similarly, multiplying the first equation of \eqref{2.47} by
$((\rho^\delta-1-\delta)^2)_x$, integrating the resultant equation
on $(x_0^\delta,1)$, we have
\begin{equation*}
\begin{split}
&\int_{x_0^\delta}^1\frac{2\delta(\rho^\delta+1)(\rho^\delta-1-\delta)((\rho^\delta)_x)^2}{\rho^3}
+\frac{(\rho^\delta+1)|((\rho^\delta-1-\delta)^2)_x|^2}{2(\rho^\delta)^3}dx\\
&\leq\frac{1}{4}\int_{x_0^\delta}^1|((\rho^\delta-1-\delta)^2)_x|^2dx+b^2
-E_r(\rho^\delta-1-\delta)^2\\&\quad
+\frac{2}{\tau}[\rho_r-(1+\delta)\ln\rho_r-1+(1+\delta)\ln(1+\delta)].
\end{split}\end{equation*}
Substituting this inequality into \eqref{2.58}, we get
\begin{equation*}
\|(1-\delta-\rho_{sup}^\delta)^2\|_{H^1(0,x_0^\delta)}\leq C, \
\|(\rho_{sub}^\delta-1-\delta)^2\|_{H^1(x_0^\delta,1)}\leq C.
\end{equation*}
Since $\eta>0$, as $\delta\rightarrow0^+$, up to a subsequence,
$x_0^\delta\rightarrow x_0\in(0,1)$. Thus, for integer $k$ large
enough, there exists a subsequence, still denoted by
$\{\rho^\delta\}$ such that
\begin{equation*}\begin{split}
&(1-\delta-\rho_{sup}^\delta)^2\rightharpoonup(1-\rho_{sup}^0)^2 \ \
\text{weakly in } H^1(0,x_0-1/k),\\&
(\rho_{sub}^\delta-1-\delta)^2\rightharpoonup(\rho_{sub}^0-1)^2 \ \
\text{weakly in } H^1(x_0+1/k,1).
\end{split}\end{equation*}
Applying the diagonal argument for
$(\rho_{trans}^\delta,E_{trans}^\delta)$, we know that
\eqref{1.5}-\eqref{boundary} has a transonic solution
$(\rho_{trans},E_{trans})(x)$ for $x\in [0,1]$ that satisfies the
sonic boundary condition, the entropy condition and the
Rankine-Hugoniot condition at the jump location $x_0$ in $(0,1)$.

Because $\tau$ only depends on $(E_1(0),E_2(0),\eta)$, and $\eta$
only depends on $(E_1(0),E_2(0))$, there exits a $\eta_0>0$ such
that for any $\eta\in(0,\eta_0)$, there exists a transonic solution
jumps at $\rho_l=1-\eta$. Thus, such transonic solutions are
infinitely many due to arbitrary choice of $0<\eta<\eta_0$. The
proof is complete.
\end{proof}

\subsection{Infinitely many $C^1$  transonic solutions}

In this subsection, we assume that the doping profile $b(x)=b>1$ is
a given constant. We will construct $C^1$ smooth transonic solution
on the base of refined local analysis of the interior subsonic
solutions and interior supersonic solutions on the boundary. The
approach highly relies on the phase-plane analysis.

We first study the structure of interior subsonic solution. For
convenience, we set
\begin{equation}\label{transformation}
F=E-\frac{1}{\tau\rho} \ \text{ and } n=\rho-1.\end{equation} Then
system \eqref{1.5} is transformed to
\begin{equation} \label{new1.5}
\left \{\begin{array}{ll}
        n_x=\dfrac{(1+n)^3F}{(2+n)n},\\
         F_x=n+1-b+\dfrac{(1+n)F}{\tau(2+n)n}.
        \end{array} \right.
\end{equation}
Clearly, $(b-1,0)$ is a saddle point of \eqref{new1.5}. In the
$(n,\F)$ plane, all trajectories satisfy
\begin{equation}\label{new3}
\begin{split}
\frac{d\F}{dn}&=\frac{(n+1-b)(2+n)}{(1+n)^3}\cdot\frac{n}{\F}
+\frac{1}{\tau(1+n)^2}\\
&\triangleq H_1(n,\F).\end{split}
\end{equation}
Here and in the sequel, to avoid confusion, $\F=\F(n)$ denotes the
function of the trajectory. The equation $H_1(n,\F)=0$ determines a
curve
\begin{equation}\label{new4}
\Xi=\Xi(n)=-\frac{\tau(n+1-b)(2+n)n}{1+n}.
\end{equation}
Obviously, if a trajectory interacts with the curve $\Xi=\Xi(n)$,
then the interacting point is a critical point of the trajectory;
and all critical points of a trajectory lie on the curve $\Xi(n)$.
We draw the phase-plane of $(n,\F)$ in Figure \ref{fig1-6}  with
$\tau=0.5$ and $b=1.5$. To state our results more precisely, we need
the following definition.

\begin{definition}
If $(\rho,E)$ is an interior subsonic (interior supersonic) solution
to system \eqref{1.5} on an interval $[0,L]$ satisfying
$\rho(0)=\rho(L)=1$, then the corresponding trajectory
$\ET=\ET(\rho)$ in the phase-plane $(\rho,\ET)$ is called an
interior subsonic (interior supersonic) trajectory to system
\eqref{1.5}. And the transformed trajectory $\F=\F(n)$ in the
$(n,\F)$ plane is called an interior positive (interior negative)
trajectory to system \eqref{new1.5}.

\end{definition}

Clearly, an interior subsonic (interior supersonic) trajectory
corresponds to an interior subsonic (interior supersonic) solution
to system \eqref{1.5} on some interval. Instead of studying system
\eqref{1.5} directly, we turn to analyze the structure of solutions
to the transformed system \eqref{new1.5}. Based on the analysis of
the relation between $\F(n)$ and $\Xi(n)$, we first obtain the
following important lemma.


\begin{figure}
\centering
\includegraphics[height=2.5in]{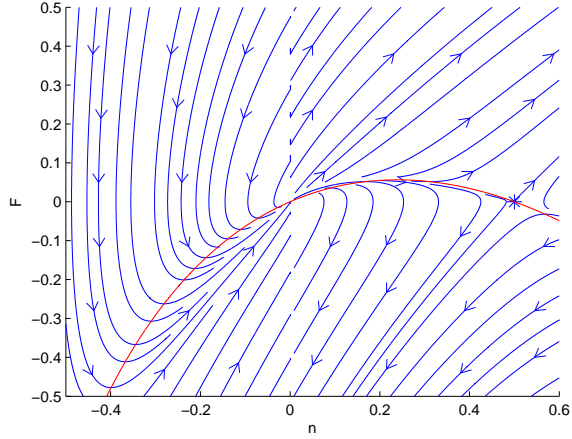}
\caption{Phase plane of $(n,\F)$ with $\tau=0.5$ and $b=1.5$; $*$ is
the saddle point $(0.5,0)$; the red line is the function
$\Xi(n)=-\frac{\tau(n+1-b)(2+n)n}{1+n}$.} \label{fig1-6}
\end{figure}


\begin{lemma}\label{new-lem1}
When $0<\tau<\frac{1}{2\sqrt{b^3+b}}$, all interior positive
trajectories to system \eqref{new1.5} start from the point $(0,0)$.

\end{lemma}

\begin{proof}
It is easy to see that there are two zero points of $\Xi(n)$ on
$[0,+\infty)$: $n_1=0$, $n_2=b-1$ and
\begin{equation}\label{new5}
\Xi'(n)=-\tau\left(2-b+2n-\frac{b}{(1+n)^2}\right)\text{ for }
n\geq0,
\end{equation}
\begin{equation}\label{new52}
\Xi''(n)=-2\tau(1+\frac{b}{(1+n)^3})<0 \text{ for } n\geq0,
\end{equation}
\[
\Xi'(0)=2(b-1)\tau>0 \ \text{ and }\ \
\Xi'(b-1)=-\tau(b-\frac{1}{b})<0.\] Thus, $\Xi(n)$ is concave on
$[0,\infty)$ and has only one maximal point denoted by $n^*$ that
only depends on $b$. We just focus on the region $\F\geq0$. By
\eqref{new3} and \eqref{new4},
\begin{equation}\label{new6}
\frac{d\F}{dn}=-\frac{\Xi}{\tau(1+n)^2\F}+\frac{1}{\tau(1+n)^2}
\end{equation}
which is equivalent to
\begin{equation}\label{2.66}
\frac{d\F}{dn}=\frac{1}{\tau(1+n)^2}\cdot\frac{\F-\beta\Xi}{\F}+\frac{(\beta-1)\Xi}{\tau(1+n)^2\F},
\end{equation}
where $\beta>0$ is a constant to be determined later. This equation
in combination with  \eqref{new5} leads to
\begin{equation}\label{new7}
\begin{split}\left(\F^2-\beta^2\Xi^2\right)'&=\frac{2(\F-\beta\Xi)}{\tau(1+n)^2}+
2\Xi\left[\frac{\beta-1}{\tau(1+n)^2}+\tau\beta^2\left(2-b+2n-\frac{b}{(1+n)^2}\right)\right]\\
&=\left(\F^2-\beta^2\Xi^2\right)\cdot\frac{2}{\tau(1+n)^2(\F+\beta\Xi)}+2\Xi\cdot
I,
\end{split}\end{equation}where
$I:=\frac{\beta-1}{\tau(1+n)^2}+\tau\beta^2\left(2-b+2n-\frac{b}{(1+n)^2}\right)$.
Since $\Xi(0)=0$, if $\F(0)=h>0$, then we have
$\F^2(0)-\beta^2\Xi^2(0)=h^2>0$ for any $\beta>0$. We next determine
$\beta$ such that $I>0$ for $n\in[0,b-1]$. To do this, we set
$\beta=\frac{c_0}{\tau^2}$ with $c_0=\frac{1}{2(b^3+b)}$. When
$\tau^2<\frac{c_0}{2}$,  we have for $n\in[0,b-1]$
\[\begin{split}
I&=\frac{1}{\tau(1+n)^2}\cdot\left[\frac{c_0}{\tau^2}-1
+\frac{c_0^2}{\tau^2}\cdot(2(1+n)^3-b(1+n)^2-b)\right]\\
&\geq\frac{1}{\tau(1+n)^2}\cdot\left[\frac{c_0}{\tau^2}-1
-\frac{c_0^2}{\tau^2}\cdot(b^3+b)\right]\\
&=\frac{1}{\tau(1+n)^2}\cdot\left[\frac{c_0}{\tau^2}\cdot(1-c_0(b^3+b))-1\right]\\
&=\frac{1}{\tau(1+n)^2}\cdot\left(\frac{c_0}{2\tau^2}-1\right)\\
&>0.\end{split}\] Noting $\Xi(n)>0$ on $(0,b-1)$, it then follows
from \eqref{new7} that
\begin{equation}\label{new8}
\F^2(n)>\beta^2\Xi^2(n) \ \text{ for }\ n\in[0,b-1].
\end{equation}
Since $(b-1,0)$ is a saddle point lying on the curve $\Xi=\Xi(n)$,
the trajectories starting from $(0,h)$ with $h>0$ can not go back to
the line $n=0$, but go to infinity. Obviously, a  trajectory can not
start from $(0,-h)$. Therefore, when $\tau<\frac{1}{2\sqrt{b^3+b}}$,
all interior positive trajectories to system \eqref{new1.5} must
start from $(0,0)$.
\end{proof}

\begin{lemma}\label{new-lem3}
When $0<\tau<\frac{1}{3\sqrt{b^3+b}}$, all interior positive
trajectories to system \eqref{new1.5} satisfy
\begin{equation}\label{new3/2}
\F(n)\leq\frac{3}{2}\cdot\Xi(n) \ \text{for} \ n\geq0.
\end{equation}
\end{lemma}

\begin{proof}
Taking $\beta=\frac{3}{2}$ in \eqref{new7}, when
$\tau^2<\frac{1}{9(b^3+b)}$, we have for $n\in[0,b-1]$
\[\begin{split}
I&=\frac{1}{\tau(1+n)^2}\cdot\left[\frac{1}{2}
+\frac{9\tau^2}{4}\cdot(2(1+n)^3-b(1+n)^2-b)\right]\\
&\geq\frac{1}{\tau(1+n)^2}\cdot\left[\frac{1}{2}
-\frac{9\tau^2}{4}\cdot(b^3+b)\right]\\
&>0.\end{split}\] If there is a point $\bar{n}\in(0,b-1)$ on the
trajectory such that $\F(\bar{n})>\frac{3}{2}\Xi(\bar{n})$, then
noting $\Xi(n)>0$ and $I>0$ on $(\bar{n},b-1)$, we get from
\eqref{new7} that $\F(n)>\frac{3}{2}\Xi(n)$ on $(\bar{n},b-1)$.
Because $(b-1,0)$ is a saddle point, this trajectory will go to infinity.
We hence get \eqref{new3/2}.
\end{proof}

\begin{lemma}\label{new-lem2}
When $0<\tau<\frac{1}{2\sqrt{b^3+b}}$, all interior positive
trajectories to system \eqref{new1.5} with $\F\geq0$ are Lipschitz
continuous on a neighborhood of $n=0$.
\end{lemma}

\begin{proof}
We first present a lower bound of $\frac{d\F}{dn}$. Notice that all
critical points of trajectories lie on the curve $\Xi=\Xi(n)$. We
claim that an interior positive trajectory to system \eqref{new1.5}
must have at least one critical point on $(0,b-1)$. Otherwise, the
trajectory has no critical point on $(0,b-1)$, then
\[\F'(n)>0 \ \text{ on }(0,b-1) \ \text{ or } \F
'(n)<0 \ \text{ on
}(0,b-1).\] If the former case holds, by \eqref{new3} and
\eqref{new4},
\[(\F-\Xi)'(n)>0 \ \text{ on }(0,b-1).\]
By Lemma \ref{new-lem1}, when $\tau<\frac{1}{2\sqrt{b^3+b}}$, it
holds $\F(0)=0=\Xi(0)$, then it follows that $\F(n)>\Xi(n)$ on
$(0,b-1)$. Since $(b-1,0)$ is a saddle point, this indicates that the
trajectory can not go back to the line $n=0$ but goes to infinity.
If the latter case holds, since $\F(0)=0$, we get
\[\F(n)<0 \ \text{ for any }n\in(0,b-1).\]
Using \eqref{new3} again, noting $n+1-b<0$ for $n\in(0,b-1)$, we
derive
\[\F'(n)>\frac{1}{\tau(1+n)^2}>0 \ \text{ on }(0,b-1),\]
which is a contradiction. Thus, an interior positive trajectory to
system \eqref{new1.5} has at least one critical point over
$(0,b-1)$.

We next claim that an interior positive trajectory has at most one
critical point. Denote by $n_0$ a critical point of this trajectory.
Taking $\beta=1$ in \eqref{new7}, and using \eqref{new5},
we have
\begin{equation}\label{new2.6}
\begin{split}\left(\F^2-\Xi^2\right)'&=\frac{2(\F-\Xi)}{\tau(1+n)^2}+
2\Xi\tau\left(2-b+2n-\frac{b}{(1+n)^2}\right)\\
&=\left(\F^2-\Xi^2\right)\cdot\frac{2}{\tau(1+n)^2(\F+\Xi)}-2\Xi\Xi'.
\end{split}\end{equation}
Recall $n^*$ is the maximal point of the function $\Xi(n)$ on
$(0,b-1)$. If $n_0\geq n^*$, noting $\F(n_0)=\Xi(n_0)$ and $\Xi(n)>0$,
$\Xi'(n)<0$ on $(n^*,b-1)$, it follows from \eqref{new2.6} that
\[\F(n)>\Xi(n) \ \text{ over } (n^*,b-1).\]
Because $(b-1,0)$ is a saddle point, this trajectory will go to infinity.
Thus, $n_0\in(0,n^*)$.

Now since $\Xi(n)>0$, $\Xi'(n)>0$ on $(n_0,n^*)$ and
$\F(n_0)=\Xi(n_0)$, by \eqref{new2.6} again, we have
\begin{equation}\label{new-2.70}
\F(n)<\Xi(n) \ \text{ over }  (n_0,n^*].\end{equation} Since all
critical points of the trajectory are on the curve $\Xi(n)$,
\eqref{new-2.70} indicates that there is no other critical point on
$(n_0,n^*]$ for this trajectory. On the other hand, suppose that
there is a critical point $n_1\in(0,n_0)$ for this trajectory, then
\[\F(n_1)=\Xi(n_1), \  \Xi(n)>0 \ \text{ and } \Xi'(n)>0 \ \text{ on } (n_1,n^*].\]
Applying \eqref{new2.6} repeatedly, we get
\[\F(n)<\Xi(n) \ \text{ for } n\in(n_1,n^*].\]
This contradicts to the fact that $\F(n_0)=\Xi(n_0)$ because
$n_0\in(n_1,n^*)$. Thus, there is no critical point on $(0,n_0)$ for
this trajectory, and $n_0$ is the unique critical point of this
interior positive trajectory. As a consequence, we conclude that
\begin{equation}\label{new2.7}
\frac{d\F(n)}{dn}>0 \ \text{ on } (0,n_0).
\end{equation}

We next derive an upper bound of $\frac{dF}{dn}$. By \eqref{new1.5},
we get
\begin{equation}\label{2.70}
F_x=n+1-b+\frac{n_x}{\tau(1+n)^2}
=n+1-b-\frac{1}{\tau}\cdot\left(\frac{1}{1+n}\right)_x.\end{equation}
Noting $n(0)=0$, by the continuity of the trajectory, $0\leq n<b-1$
on $[0,z]$ for some $z>0$. Noting $F(0)=0$, hence, for $x\in[0,z]$
\[F(x)<-\frac{1}{\tau}\int_0^x\left(\frac{1}{1+n}\right)_xdx=\frac{1}{\tau}-\frac{1}{\tau(1+n)}.\]
It then follows that for $n\in[0,b-1]$ and $\F\geq0$,
\[\frac{d\F(n)}{dn}<\frac{\tau(n+1-b)(2+n)}{(1+n)^2}+\frac{1}{\tau(1+n)^2}\leq\frac{1}{\tau}.\]
This estimate together with \eqref{new2.7} implies  the trajectory
is Lipschitz continuous  on $(0,n_0)$.
\end{proof}

\begin{lemma}\label{new-lem4}
When $0<\tau<\min\{\frac{1}{3\sqrt{b^3+b}},\frac{1}{4\sqrt{b-1}}\}$,
all interior positive trajectories to system \eqref{new1.5} with
$\F\geq0$ are $C^1$ smooth on a neighborhood of $n=0$ and
\begin{equation}\label{first order}
\frac{d\F}{dn}(0)=
\frac{1}{2}\left(\frac{1}{\tau}-\sqrt{\frac{1}{\tau^2}-8(b-1)}\right).
\end{equation}

\end{lemma}

\begin{proof}
By Lemma \ref{new-lem2}, we only need to show that the second order
derivative of the trajectory does not change sign on a neighborhood
of $n=0$, i.e.
\begin{equation}\label{new-2nd}
\frac{d^2\F}{dn^2} \text{  does not change sign, if } 0<n\ll1.
\end{equation}

Step 1. We first compute $\frac{d^2\F}{dn^2}$. By \eqref{new3} and
\eqref{new4},
\begin{equation}\label{new2.68}
(1+n)^2\F\F'=\frac{1}{\tau}(\F-\Xi).
\end{equation}
Notice that $\F(n)$ is $C^\infty$ over $(0,b-1)$. Differentiating
\eqref{new2.68} in $n$ and using the first equality of \eqref{new6},
a direct calculation yields
\begin{equation}\label{new2.69}
\begin{split}
(1+n)^2\F\F''&=-2(1+n)\F\F'-(1+n)^2(\F')^2+\frac{1}{\tau}(\F'-\Xi')\\
&=-\frac{1}{\tau(1+n)\F^2}\left[2\F^3-(2\Xi-(1+n)\Xi')\F^2-\frac{\Xi
\F}{\tau(1+n)}+\frac{\Xi^2}{\tau(1+n)}\right].
\end{split}
\end{equation}
By \eqref{new4} and \eqref{new5}, it is easy to see that
\begin{equation*}
\begin{split}
2\Xi-(1+n)\Xi'&=-\frac{2\tau(n+1-b)(2+n)n}{1+n}+\tau(1+n)\left(2-b+2n-\frac{b}{(1+n)^2}\right)\\
&=\frac{\tau}{1+n}\left[2(n+1-b)+bn(2+n)\right].
\end{split}
\end{equation*}
It then follows that
\begin{equation}\label{new2.70}
\begin{split}
\F''&=-\frac{2}{\tau(1+n)^3\F^3}\bigg\{\F^3-\frac{\tau\left[2(n+1-b)
+bn(2+n)\right]}{2(1+n)}\cdot \F^2\\
&\quad+\frac{(2+n)n(n+1-b)}{2(1+n)^2}\cdot \F
+\frac{\tau(2+n)^2n^2(n+1-b)^2}{2(1+n)^3}\bigg\}\\
&\triangleq-\frac{2}{\tau(1+n)^3\F^3}\cdot H_2(n,\F).
\end{split}
\end{equation}

Step 2. We next solve the equation $H_2(n,\F)=0$, which is a third
order algebraic equation in the form:
\begin{equation}\label{new2.71}
\F^3+k\F^2+m\F+\ell=0,
\end{equation}
where
\begin{equation}\label{new2.73}
\begin{split}
k&=-\frac{\tau\left[2(n+1-b)+bn(2+n)\right]}{2(1+n)},\ m=\frac{(2+n)n(n+1-b)}{2(1+n)^2},\\
\ell&=\frac{\tau(2+n)^2n^2(n+1-b)^2}{2(1+n)^3}.\end{split}\end{equation}
Denote by $p=-\frac{k^2}{3}+m$,
$q=2(\frac{k}{3})^3-\frac{km}{3}+\ell$, by Cardan's formula,
equation \eqref{new2.71} has three roots:
\[F_1=A^{\frac{1}{3}}+B^{\frac{1}{3}},\ F_2=\varpi A^{\frac{1}{3}}+\varpi^2B^{\frac{1}{3}},
\ F_3=\varpi^2A^{\frac{1}{3}}+\varpi B^{\frac{1}{3}},\] where
$\varpi=\frac{-1+\sqrt{3}i}{2}$,
$A=-\frac{q}{2}+[(\frac{q}{2})^2+(\frac{p}{3})^3]^{\frac{1}{2}}$ and
$B=-\frac{q}{2}-[(\frac{q}{2})^2+(\frac{p}{3})^3]^{\frac{1}{2}}$.
Furthermore, if $(\frac{q}{2})^2+(\frac{p}{3})^3\leq0$, then all
roots are real valued. We claim that when
$\tau<\frac{1}{4\sqrt{b-1}}$ and $0<n\ll1$, then
$(\frac{q}{2})^2+(\frac{p}{3})^3\leq0$. Actually, a simple
calculation gives
\begin{equation}\label{new2.75}
\left(\frac{q}{2}\right)^2+\left(\frac{p}{3}\right)^3=\frac{1}{4\cdot3^4}[(km-9\ell)^2-4(k^2-3m)(m^2-3k\ell)].
\end{equation}
When $0<n\ll1$, by \eqref{new2.73},
\begin{equation*}
k=\tau(b-1)+O(n),\ m=(1-b)n+O(n^2), \
\ell=2\tau(b-1)^2n^2+O(n^3).\end{equation*} It then follows that
\[\begin{split}km-9\ell&=-\tau(b-1)^2n+O(n^2), \ k^2-3m=(b-1)^2\tau^2+O(n), \\ m^2-3k\ell&=(b-1)^2n^2-6\tau^2(b-1)^3n^2+O(n^3).\end{split}\]
Substituting these  three estimates into \eqref{new2.75} yields
\begin{equation*}\begin{split}
\left(\frac{q}{2}\right)^2+\left(\frac{p}{3}\right)^3&=\frac{1}{4\cdot3^4}\cdot[(b-1)^4\tau^2n^2-4(b-1)^2\tau^2((b-1)^2-6\tau^2(b-1)^3)n^2+O(n^3)]\\
&=\frac{1}{4\cdot3^4}\cdot[3(b-1)^4\tau^2(-1+8\tau^2(b-1))n^2+O(n^3)].
\end{split}\end{equation*}
Thus, when $\tau<\frac{1}{4\sqrt{b-1}}$ and $0<n\ll1$, we have
$(\frac{q}{2})^2+(\frac{p}{3})^3<0$.

Now all roots of the equation $H_2(n,\F)=0$ are real valued
functions. And clearly, they are analytic in $n$ on $(0,b-1)$. We
then take an expansion of the roots denoted by $\F_0(n)$ as
\begin{equation*}
\F_0(n)=\theta_0+\theta_1n+O(n^2),
\end{equation*}
and substitute this formula into $H_2(n,\F)=0$ to get
\[\theta_0=0 \ \text{ or } \theta_0=-\tau(b-1)<0,\]
and
\[\theta_1=\frac{1}{2}\left(\frac{1}{\tau}+\sqrt{\frac{1}{\tau^2}-8(b-1)}\right)
\ \text{ or }
\theta_1=\frac{1}{2}\left(\frac{1}{\tau}-\sqrt{\frac{1}{\tau^2}-8(b-1)}\right).\]
Notice that when $\tau\ll1$,
$\frac{1}{2}\left(\frac{1}{\tau}+\sqrt{\frac{1}{\tau^2}-8(b-1)}\right)=O(\frac{1}{\tau})$
and
$\frac{1}{2}\left(\frac{1}{\tau}-\sqrt{\frac{1}{\tau^2}-8(b-1)}\right)=\frac{4(b-1)\tau}{1+\sqrt{1-8(b-1)\tau^2}}=O(\tau)$.
Because we are interested in the interior positive trajectories with
$\F\geq0$, by Lemma \ref{new-lem3}, it holds that
\begin{equation}\label{new2.76}
\theta_0=0\text{ and } \theta_1=\frac{1}{2}\left(\frac{1}{\tau}-\sqrt{\frac{1}{\tau^2}-8(b-1)}\right).
\end{equation}
Thus, the solution curve of the equation $H_2(n,\F)=0$ satisfies
\begin{equation}\label{new2.77}
\F_0(n)=\theta_1n+O(n^2)=\frac{4(b-1)\tau}{1+\sqrt{1-8(b-1)\tau^2}}\cdot
n+O(n^2).
\end{equation}

Step 3. We proceed to show that when $0<n\ll1$, the function
$\frac{d\F}{dn}(n)$ is monotone.

Assume that $\hat{n}_0>0$ is a critical point of the function
$\frac{d\F}{dn}(n)$, then $\frac{d^2\F}{dn^2}(\hat{n}_0)=0$. We
claim that when $\hat{n}_0$ is small enough, it holds that
$\frac{d^3\F}{dn^3}(\hat{n}_0)>0$. Differentiating \eqref{new2.69}
in $n$, we have
\begin{equation*}
\begin{split}
&2(1+n)\F\F''+(1+n)^2\F'\F''+(1+n)^2\F\F'''\\
&=-2\F\F'-4(1+n)(\F')^2-2(1+n)\F\F''
-2(1+n)^2\F'\F''+\frac{1}{\tau}(\F''-\Xi'').\end{split}
\end{equation*}
Noting $\F''(\hat{n}_0)=0$, it then follows from \eqref{new52} that
\begin{equation}\label{new2.78}
\begin{split}
(1+\hat{n}_0)^2\F\F'''(\hat{n}_0)
&=-2\F\F'(\hat{n}_0)-4(1+\hat{n}_0)(\F'(\hat{n}_0))^2
-\frac{\Xi''(\hat{n}_0)}{\tau}\\
&=-2\F\F'(\hat{n}_0)-4(1+\hat{n}_0)(\F'(\hat{n}_0))^2
+2+\frac{2b}{(1+\hat{n}_0)^3}.
\end{split}\end{equation}
Using \eqref{new2.69} again, since $\F''(\hat{n}_0)=0$, it holds
\[(1+\hat{n}_0)^2(\F'(\hat{n}_0))^2=-2(1+\hat{n}_0)\F\F'(\hat{n}_0)
+\frac{1}{\tau}(\F'(\hat{n}_0)-\Xi'(\hat{n}_0)).\] Substituting this
inequality into \eqref{new2.78} and using \eqref{new6} leads to
\begin{equation}\label{new2.79}
\begin{split}
&(1+\hat{n}_0)^3\F\F'''(\hat{n}_0)\\
=&-2(1+\hat{n}_0)\F\F'(\hat{n}_0)-4(1+\hat{n}_0)^2(\F'(\hat{n}_0))^2
+2(1+\hat{n}_0)+\frac{2b}{(1+\hat{n}_0)^2}\\
=&6(1+\hat{n}_0)\F\F'(\hat{n}_0)-\frac{4\F'(\hat{n}_0)}{\tau}
+\frac{4\Xi'(\hat{n}_0)}{\tau}+2(1+\hat{n}_0)+\frac{2b}{(1+\hat{n}_0)^2}\\
=&\frac{6(\F(\n)-\Xi(\n))}{\tau(1+\n)}
-\frac{4(\F(\n)-\Xi(\n))}{\tau^2(1+\n)^2\F(\n)}
+\frac{4\Xi'(\n)}{\tau}+2(1+\n)+\frac{2b}{(1+\n)^2}\\
=&\frac{2}{(1+\n)^2\F}\cdot\Big[\frac{3\F(\F-\Xi)(1+\n)}{\tau}
-\frac{2(\F-\Xi)}{\tau^2}+\frac{2(1+\n)^2\F\Xi'}{\tau}+(1+\n)^3\F+b\F\Big]\\
:=&\frac{2}{(1+\n)^2\F}\cdot J(\n).
\end{split}\end{equation}
By \eqref{new4}, \eqref{new5} and \eqref{new2.77}, when $\n\ll1$,
\[\begin{split}&\F(\n)=\theta_1\n+O(\n^2), \
\F(\n)-\Xi(\n)=(\theta_1-2\tau(b-1))\n+O(\n^2),
\\&\hat{\Xi}'(\n)=2\tau(b-1)+O(\n).\end{split}\]
Thus,
\begin{equation}\label{new2.80}\begin{split}J(\n)=&-\frac{2(\theta_1-2\tau(b-1))}{\tau^2}\n
+4(b-1)\theta_1\n+\theta_1\n+b\theta_1\n+O(\n^2)\\
=&\left[\frac{4(b-1)}{\tau}-\frac{2\theta_1}{\tau^2}+(5b-3)\theta_1\right]\n+O(\n^2).
\end{split}\end{equation}
By \eqref{new2.76},
\[\frac{4(b-1)}{\tau}-\frac{2\theta_1}{\tau^2}=-\theta_1\cdot\frac{8(b-1)}{1+\sqrt{1-8(b-1)\tau^2}}.\]
It hence follows that if $\tau\ll1$ such that
$\tau^2<\frac{1}{16(b-1)}<\frac{2}{25(b-1)}$, then
\[\begin{split}\frac{4(b-1)}{\tau}-\frac{2\theta_1}{\tau^2}+(5b-3)\theta_1
&=\theta_1\left[5b-3-\frac{8(b-1)}{1+\sqrt{1-8(b-1)\tau^2}}\right]\\&>\theta_1(5b-3-5(b-1))\\&=2\theta_1>0.
\end{split}\]
Substituting this inequality into \eqref{new2.80} and then
\eqref{new2.79}, we conclude that
\[\F'''(\n)>0 \ \text{ if } \n\ll1 \ \text{ and } \tau<\frac{1}{4\sqrt{b-1}}.\]
Thus, the critical point $\n$ must be the local minimal point of $\frac{d\F}{dn}(n)$. And
hence there exists $n_2>0$ such that the function
$\frac{d\F}{dn}(n)$ has at most one critical point over $(0,n_2)$.
This implies $\frac{d^2\F}{dn^2}$ could change sign at most once on
$(0,n_2)$. As a consequence, there exists $n_3\in(0,n_2)$ such that the function $\frac{d\F}{dn}(n)$ is
monotone  on $(0,n_3)$.

Step 4. Now by Lemma \ref{new-lem2} and the monotonicity of
$\frac{d\F}{dn}$, one can easily see that
\[\lim_{n\rightarrow0^+}\F'(n) \text{ exists }\triangleq \F'(0).\]
Then $\F'(n)$ is continuous on $[0,n_2]$. It is left to show
\eqref{first order}. Applying L'Hospital principle to equation
\eqref{new3} at $n=0$, it holds that
\[\F'(0)=\frac{2(1-b)}{\F'(0)}+\frac{1}{\tau}.\]
Thus,
\[\F'(0)=\frac{1}{2}\left(\frac{1}{\tau}+\sqrt{\frac{1}{\tau^2}-8(b-1)}\right)=O\left(\frac{1}{\tau}\right)
\ \text{ or }
\F'(0)=\frac{1}{2}\left(\frac{1}{\tau}-\sqrt{\frac{1}{\tau^2}-8(b-1)}\right)=O(\tau).\]
By Lemma \ref{new-lem3},
$\F'(0)=\frac{1}{2}\left(\frac{1}{\tau}-\sqrt{\frac{1}{\tau^2}-8(b-1)}\right)$.
\end{proof}

\begin{theorem}\label{new-thm1}
Assume that $b(x)=b>1$ is a constant. There exists a constant
$\tau_0=\tau_0(b)$ only depending on $b$, such that for any
$0<\tau<\tau_0$ the interior subsonic solution to system \eqref{1.5}
satisfies
\begin{equation}\label{new-2.82}
\rho\in C^1[0,\epsilon], \ \rho(0)=1, \ E(0)=\frac{1}{\tau} \ \ and
\ \
\rho_x(0)=\frac{1}{4}\left(\frac{1}{\tau}-\sqrt{\frac{1}{\tau^2}-8(b-1)}\right)
\end{equation}
for some $\epsilon>0$.
\end{theorem}

\begin{proof}
Recalling the transformation \eqref{transformation}, $\rho=n+1$ and
$E=F+\frac{1}{\tau(n+1)}$. By Lemma \ref{new-lem1}, one can find
that all interior subsonic trajectories to system \eqref{1.5} must
start from $(1,\frac{1}{\tau})$. In other words, all interior
subsonic solutions to system \eqref{1.5} must satisfy $\rho(0)=1$
and $E(0)=\frac{1}{\tau}$. By L'Hospital principle and \eqref{first
order},
\[\lim_{n\rightarrow0^+}\frac{n}{\F(n)}
=\lim_{n\rightarrow0^+}\frac{1}{\F'(n)}=\frac{1}{\theta_1},\] which
together with the first equation of \eqref{new1.5} gives
\[n_x(0)=\lim_{x\rightarrow0^+}\frac{F(x)}{2n(x)}=\frac{\theta_1}{2}.\]
Thus, $n\in C^1[0,\epsilon]$ for some $\epsilon>0$. Recalling
$n=\rho-1$, we have $\rho\in C^1[0,\epsilon]$ for some $\epsilon>0$,
and
\[\rho_x(0)=\frac{\theta_1}{2}
=\frac{1}{4}\left(\frac{1}{\tau}-\sqrt{\frac{1}{\tau^2}-8(b-1)}\right),\]
where we have used
$\theta_1=\frac{1}{2}\left(\frac{1}{\tau}-\sqrt{\frac{1}{\tau^2}-8(b-1)}\right)$.
\end{proof}

We next study the structure of the interior supersonic solutions. To
do so, we still study the transformed equations \eqref{new1.5} and
\eqref{new3} but with $n\in(-1,0]$.

\begin{lemma}\label{new-lem1-2}
When $0<\tau<\frac{1}{3\sqrt{b}}$, all interior negative
trajectories to system \eqref{new1.5} end at the point $(0,0)$.

\end{lemma}

\begin{proof}
By \eqref{new4}-\eqref{new52},
\begin{equation}\label{new2.85}
\Xi(n)\leq0,\ \Xi'(n)>2\tau(b-1)>0, \ \Xi''(n)<0 \text{ for }
n\in(-1,0],
\end{equation}
\begin{equation}\label{new2.86}
\lim_{n\rightarrow-1}\Xi(n)=-\infty.
\end{equation}
We next focus on the region $\F\leq0$. Notice that \eqref{new7}
still holds. If $\F(n)=-h<0$, then $\F^2(0)-\beta^2\Xi^2(0)=h^2>0$
for any $\beta>0$. To ensure $\F^2(n)>\beta^2\Xi^2(n)$ for any
$n\in(-1,0]$, we also need to determine $\beta$ such that $I>0$ for
$n\in(-1,0]$. Setting $\beta=\frac{c_1}{\tau^2}$ with
$c_1=\frac{1}{3b}$, when $\tau<\frac{1}{3\sqrt{b}}$,  we have for
$n\in(-1,0]$
\[\begin{split}
I&=\frac{1}{\tau(1+n)^2}\cdot\left[\frac{c_1}{\tau^2}-1
+\frac{c_1^2}{\tau^2}\cdot(2(1+n)^3-b(1+n)^2-b)\right]\\
&\geq\frac{1}{\tau(1+n)^2}\cdot\left[\frac{c_1}{\tau^2}-1
-\frac{2bc_1^2}{\tau^2}\right]\\
&=\frac{1}{\tau(1+n)^2}\cdot\left(\frac{1}{9b\tau^2}-1\right)\\
&>0.\end{split}\] It then follows from \eqref{new7} that
$\F^2(n)>\frac{\Xi^2(n)}{3b\tau^2}$ for $n\in(-1,0)$. Noting
$\F(0)<0$ and $\Xi(n)<0$ on $(-1,0)$, we get
\[\F(n)<\Xi(n) \ \text{ for }n\in(-1,0).\]
It hence follows from \eqref{new2.86} that
\[\lim_{n\rightarrow-1}\F(n)=-\infty,\]
and the trajectory ending at $(0,-h)$ with $h>0$ does not start from
a point of the line $n=0$. Thus, when $\tau<\frac{1}{3\sqrt{b}}$,
all interior negative trajectories should end at the point $(0,0)$.
\end{proof}

\begin{lemma}\label{new-lem3-2}
When $\tau<\frac{1}{3\sqrt{b}}$, all interior negative trajectories
to system \eqref{new1.5} satisfy
\begin{equation}\label{new3/2-2}
\F(n)\geq\frac{3}{2}\cdot\Xi(n) \ \text{for} \ n\in(-1,0].
\end{equation}
\end{lemma}

\begin{proof}
Taking $\beta=\frac{3}{2}$ in \eqref{new7}, when
$\tau<\frac{1}{3\sqrt{b}}$, we have for $n\in(-1,0]$
\[\begin{split}
I&=\frac{1}{\tau(1+n)^2}\cdot\left[\frac{1}{2}
+\frac{9\tau^2}{4}\cdot(2(1+n)^3-b(1+n)^2-b)\right]\\
&\geq\frac{1}{\tau(1+n)^2}\cdot\left[\frac{1}{2}
-\frac{9\tau^2}{4}\cdot2b\right]\\
&>0.\end{split}\] If there is a point $\bar{n}\in(-1,0)$ on the
trajectory such that $\F(\bar{n})<\frac{3}{2}\cdot\Xi(\bar{n})<0$,
then noting $\Xi(n)<0$ and $I>0$ on $(-1,\bar{n})$, by \eqref{new7},
we have
$$\F^2(n)>\frac{9}{4}\cdot\Xi^2(n) \ \text { on }
(-1,\bar{n}).$$ Because $\Xi(\bar{n})<0$ and $\F(\bar{n})<0$ on
$(-1,\bar{n})$, we have $\F(n)<\frac{3}{2}\cdot\Xi(n)$ for
$n\in(-1,\bar{n})$. Thus, by \eqref{new2.86},
$\lim_{n\rightarrow-1}\F(n)=-\infty$, and this trajectory starts
from infinity and can not be an interior negative trajectory to
system \eqref{new1.5}. We hence get \eqref{new3/2-2}.
\end{proof}

\begin{lemma}\label{new-lem2-2}
When $\tau<\frac{1}{3\sqrt{b}}$, all interior negative trajectories
to system \eqref{new1.5} with $\F\leq0$ are Lipschitz continuous on
a neighborhood of $n=0$.
\end{lemma}

\begin{proof}
We first show that an interior negative trajectory must have at
least one critical point on $(-1,0)$. Otherwise, the trajectory has
no critical point over $(-1,0)$, then
\[\F'(n)>0 \ \text{ on }(-1,0) \ \text{ or } \F'(n)<0 \ \text{ on }(-1,0).\]
If $\F'(n)>0$ on $(-1,0)$, then
\[(\F-\Xi)'(n)>0 \ \text{ on }(-1,0).\]
By Lemma \ref{new-lem1-2}, when $\tau<\frac{1}{3\sqrt{b}}$,
$\F(0)=\Xi(0)=0$, we then have $\F(n)<\Xi(n)<0$ on $(-1,0)$. Thus,
by \eqref{new2.86}
\[\lim_{n\rightarrow-1}\F(n)<\lim_{n\rightarrow-1}\Xi(n)=-\infty.\]
This implies the trajectory can not start from a point of the line
$n=0$. If $\F'(n)<0$ on $(-1,0)$, noting $\F(0)=0$, it holds that
$\F(n)>0$ for $n\in(-1,0).$ By \eqref{new3}, since
$\frac{(n+1-b)(2-n)n}{(1+n)^3F}>0$, we have
\[\F'(n)>\frac{1}{\tau(1+n)^2}>0 \ \text{ on }(-1,0),\]
which is a contradiction. Thus, an interior negative trajectory to
system \eqref{new1.5} has at least one critical point on $(-1,0)$.

We next claim that this interior negative trajectory has at most one
critical point. Denote by $\widetilde{n_0}\in(-1,0)$ a critical
point of this trajectory. Then by \eqref{new2.85},
\begin{equation}\label{new2.88}
\F(\widetilde{n_0})=\Xi(\widetilde{n_0})<0, \ \Xi(n)<0, \Xi'(n)>0 \
\text{ on } (-1,\widetilde{n_0}).
\end{equation}
Noting \eqref{new2.6} still holds, it follows from \eqref{new2.88}
and \eqref{new2.6} that when $F\leq0$,
\[0\geq \F(n)>\Xi(n) \ \text{ for } n\in(n_*,\widetilde{n_0}),\]
where $n_*$ is the point that $\F(n_*)=0$. In other words, there is
no critical point on $(-1,\widetilde{n_0})$. On the other hand, if
there is a critical point $\widetilde{n_1}\in (\widetilde{n_0},0)$,
then
\[\F(\widetilde{n_1})=\Xi(\widetilde{n_1})<0, \  \Xi(n)<0
\ \text{ and } \Xi'(n)>0 \ \text{ on }
(\widetilde{n_0},\widetilde{n_1}).\] Applying \eqref{new2.6} again,
we have
\[\F(n)>\Xi(n) \ \text{ for } n\in(n_*,\widetilde{n_1}),\]
which contradicts to the fact that
$\F(\widetilde{n_0})=\Xi(\widetilde{n_0})$. Thus, there is no
critical point on $(\widetilde{n_0},0)$ for this trajectory, and
$\widetilde{n_0}$ is the unique critical point of this trajectory.
As a consequence, we obtain
\begin{equation}\label{new2.89}
\frac{d\F(n)}{dn}>0 \ \text{ on } (\widetilde{n_0},0).
\end{equation}

We next derive an upper bound of $\frac{d\F}{dn}$. Integrating
\eqref{2.70} on $(x,1)$, noting $F(1)=0$, we have
\[F(x)>\frac{1}{\tau}\int_x^1\left(\frac{1}{1+n}\right)_xdx
=\frac{1}{\tau}-\frac{1}{\tau(1+n)}=\frac{n}{\tau(1+n)}.\] Noting
$F<0$ on $(x,1)$, it follows that $\frac{n}{F}>\tau(1+n)$. By
\eqref{new3}, we obtain
\[\frac{d\F(n)}{dn}
<\frac{\tau(n+1-b)(2+n)}{(1+n)^2}
+\frac{1}{\tau(1+n)^2}\leq\frac{1}{\tau(1+n)^2}<\frac{1}{\tau(1+\widetilde{n_0})^2}
\ \text{ for } n\in(\widetilde{n_0},0).\] This estimate together
with \eqref{new2.89} implies the trajectory is Lipschitz continuous
on $(\widetilde{n_0},0)$.
\end{proof}

\begin{lemma}\label{new-lem4-2}
When $\tau<\min\{\frac{1}{3\sqrt{b}},\frac{1}{4\sqrt{b-1}}\}$, all
interior negative trajectories to system \eqref{new1.5} with
$\F\leq0$ are $C^1$ smooth on a neighborhood of $n=0$ and
\begin{equation*}
\frac{d\F}{dn}(0)=
\frac{1}{2}\left(\frac{1}{\tau}-\sqrt{\frac{1}{\tau^2}-8(b-1)}\right).
\end{equation*}

\end{lemma}

\begin{proof}
The proof is quite similar to that of Lemma \ref{new-lem4}. The main
difference is that, now the unique critical point of the function
$\frac{d\F}{dn}$ is the maximal point of $\frac{d\F}{dn}$. Other
changes are obvious.
\end{proof}
On the base of Lemma \ref{new-lem4-2}, analog to Theorem
\ref{new-thm1}, one can obtain the refined structure of the interior
supersonic solution established in Theorem \ref{thm2}.

\begin{theorem}\label{new-thm2}
Assume that $b(x)=b>1$ is a constant. There exists a constant
$\tau_0=\tau_0(b)$ such that for any $0<\tau<\tau_0$ the interior
supersonic solution $(\rho,E)$ on an interval $[0,L]$ satisfies
\begin{equation}\label{new-2.92}
\rho\in C^1[L-\epsilon,L], \ \rho(0)=\rho(L)=1, \
E(L)=\frac{1}{\tau} \ \ and \ \
\rho_x(L)=\frac{1}{4}\left(\frac{1}{\tau}-\sqrt{\frac{1}{\tau^2}-8(b-1)}\right).
\end{equation}
for some $\epsilon>0$.
\end{theorem}

On the base of Theorems \ref{new-thm1} and \ref{new-thm2}, we are
able to construct interior $C^1$ smooth transonic solutions to
system \eqref{1.5}.

\begin{theorem}\label{new-thm3}
Assume that $b(x)=b>1$ is a constant. There exists a constant
$\tau_0=\tau_0(b)$ such that for any $0<\tau<\tau_0$, there exist
infinitely many interior $C^1$ smooth transonic solutions to system
\eqref{1.5}-\eqref{boundary} in the form
\begin{equation*}
\rho(x)=\left \{\begin{array}{ll}
        \rho_{sup}(x), \ x\in(0,x_0),\\
        \rho_{sub}(x), \ x\in(x_0,1),
        \end{array} \right.
\end{equation*}
where $x_0\in(0,1)$ is the location of transition,
$0<\rho_{sup}(x)\le 1$ and $\rho_{sub}(x)\ge 1$ satisfy
\begin{equation}\label{new2.93}
\rho_{sup}(x_0)=\rho_{sub}(x_0)=1,
\end{equation}
\begin{equation}\label{new2.94}
\begin{split}
(\rho_{sup})_x(x_0)=(\rho_{sub})_x(x_0)&=\frac{1}{4}\left(\frac{1}{\tau}-\sqrt{\frac{1}{\tau^2}-8(b-1)}\right),\\
E_{sup}(x_0)&=E_{sub}(x_0)=\frac{1}{\tau}.
\end{split}
\end{equation}

\end{theorem}

\begin{proof}
For any $x_0\in (0,1)$, by Theorem \ref{thm2}, system \eqref{1.5}
admits an interior supersonic solution $\rho_{sup}$ on $[0,x_0]$
satisfying
\[\rho_{sup}(0)=\rho_{sup}(x_0)=1.\]
By Theorem \ref{new-thm2}, there exists a constant
$\tau_0=\tau_0(b)$ such that for any $0<\tau<\tau_0$
\begin{equation}\label{new-2.95}
\rho_{sup}\in C^1[x_0-\epsilon_0,x_0], \ E_{sup}(x_0)=\frac{1}{\tau}
\ \ and \ \
(\rho_{sup})_x(x_0)=\frac{1}{4}\left(\frac{1}{\tau}-\sqrt{\frac{1}{\tau^2}-8(b-1)}\right).
\end{equation}
for some $\epsilon_0>0$.

On the other hand, by Theorem \ref{thm1}, system \eqref{1.5} has a
unique interior subsonic solution $\rho_{sub}$ on $[x_0,1]$
satisfying
\[\rho_{sub}(x_0)=\rho_{sub}(1)=1.\]
By Theorem \ref{new-thm1}, there exists a constant
$\tau_1=\tau_1(b)$ such that for any $0<\tau<\tau_1$
\begin{equation}\label{new-2.96}
\rho_{sub}\in C^1[x_0,x_0+\epsilon_1], \ E_{sub}(x_0)=\frac{1}{\tau}
\ \ and \ \
(\rho_{sub})_x(x_0)=\frac{1}{4}\left(\frac{1}{\tau}-\sqrt{\frac{1}{\tau^2}-8(b-1)}\right).
\end{equation}
for some $\epsilon_1>0$. We can now construct an interior $C^1$
smooth transonic solution by
\begin{equation*}
\rho(x)=\left \{\begin{array}{ll}
        \rho_{sup}(x), \ x\in[0,x_0],\\
        \rho_{sub}(x), \ x\in[x_0,1].
        \end{array} \right.
\end{equation*}
Furthermore, \eqref{new2.93} and \eqref{new2.94} follows from
\eqref{new-2.95} and \eqref{new-2.96}. Because $x_0\in(0,1)$ is
arbitrary, the $C^1$ smooth transonic solutions are infinitely many.
\end{proof}

As a byproduct, one can easily see that when $0<\tau\ll1$, there is
no transonic solution with shock. In other words, when $\tau$ is
small, system \eqref{1.5} admits transonic solution of $C^1$ smooth
type only.

\begin{theorem}\label{new-thm4}
Assume that $b(x)=b>1$ is a constant. There exists a constant
$\tau_0=\tau_0(b)$ such that for any $0<\tau<\tau_0$, system
\eqref{1.5}-\eqref{boundary} has no transonic shock solution.

\end{theorem}

\begin{proof}
We argue by contradiction. Assume that there is a transonic solution
with shock. Denote by $x_0$ the jump location. By the
Rankine-Hugoniot condition \eqref{RH} and \eqref{2.24},
\begin{equation}\label{new2.97}
E_l=E_r\ \text{ and } \ \rho_l\rho_r=1.
\end{equation}
Because the solution is discontinuous, it holds $0<\rho_l<1<\rho_r$.
Clearly, there are two cases for the value of $E_l$:
\[E_l\leq\frac{1}{\tau} \ \text{ or } \ E_l>\frac{1}{\tau}.\]
If the former case holds, observing that at $x_0$,
$\rho_{sup}(x_0)E_{sup}(x_0)-\frac{1}{\tau}=\rho_lE_l-\frac{1}{\tau}<0$,
it follows from the first equation of \eqref{1.5} that
\[\rho_{sup}(x_0)=\frac{\rho_lE_l-\frac{1}{\tau}}{1-\frac{1}{\rho_l^2}}>0.\]
Thus, we can extend this supersonic solution to an interval $[0,L]$
such that
\[\rho_{sup}(L)=1, \ \text{ and }\ E_{sup}(L)<E_{sup}(x_0)=E_l<\frac{1}{\tau}.\]
Here we have used the fact that $E_{sup}$ is monotone decreasing.
Recalling the transformation \eqref{transformation}, this implies
\[F_{sup}(L)=E_{sup}(L)-\frac{1}{\tau} <0.\]
In view of the proof of Lemma \ref{new-lem1-2}, we find that the
corresponding trajectory satisfies
\[\lim_{n\rightarrow-1^+}\F_{sup}(n)=-\infty.\]
Thus, this supersonic solution can not satisfy the left boundary
condition $\rho_{sup}(0)=1$, which is a contradiction.

If the latter case happen, by \eqref{new2.97}, we get
\[E_r>\frac{1}{\tau} \ \text{ and } \rho_r>1.\]
Thus, we can extend backward this subsonic part to an interior
subsonic solution of system \eqref{1.5}, still denoted by
$(\rho_{sub},E_{sub})$ such that for some $x_{-1}\in\R$
\[\rho_{sub}(x_{-1})=1, \ E_{sub}(x_{-1})>E_r>\frac{1}{\tau},\]
where we have used the fact that $E_{sub}$ is monotone decreasing.
Recalling the transformation again, we have
\[F_{sub}(x_{-1})=E_{sub}(x_{-1})-\frac{1}{\tau}>E_r-\frac{1}{\tau}>0.\]
In view of the proof of Lemma \ref{new-lem1}, one can see that the
corresponding trajectory will go to infinity, which contradicts to
the right boundary condition $\rho_{sub}(1)=1$. Therefore, there is
no transonic solution with shock.
\end{proof}

\begin{proof}[Proof of Theorem \ref{main-thm-1}] Combining Theorems \ref{thm1}, \ref{thm2}, \ref{thm3},
\ref{new-thm3} and \ref{new-thm4},  we immediately obtain Theorem
\ref{main-thm-1}.
\end{proof}

\section{The case of supersonic doping profile}\label{supersonic doping profile}
In this section, we consider the more general case of non-subsonic
doping file, namely, we assume that $0\leq\underline{b}\le
b(x)\leq\overline{b}\leq1$, which allows $b(x)$ partially supersonic
and partially sonic in the domain $[0,1]$. To observe the structure
of the stationary solutions, let us first test the specially case
with a constant supersonic doping profile $b(x)\equiv b<1$, where
the phase-plane analysis is helpful. In this case, the critical
point $A=\left(b,\frac{1}{\tau b}\right)$ sits at the left hand side
of the sonic line $\rho=1$. By \eqref{eigenvalue}, the eigenvalues
of the Jacobian matrix $J(A)$ satisfy
$$\lambda_1+\lambda_2<0,\ \ \lambda_1\lambda_2>0.$$
Thus, $\lambda_1,\lambda_2<0$, which indicates that $A$ is a stable
focal point. In view of \eqref{direction} and \eqref{1.5}, we draw
the phase-plane of $(\rho,E)$ in Figure 7 with $\tau=15$ and
$b=0.5$. From Figure 7, we see that one outside curve starts from
the sonic line, passes through the supersonic regime, then ends at
the sonic line, so there is a possible interior supersonic solution.
The other curve starts from the sonic line, but rotates in the
supersonic regime, and never ends at the sonic line, thus such a
curve is not a solution. Obviously, there is no interior subsonic
solution.


\begin{figure}
\begin{center}
\includegraphics[height=2.5in]{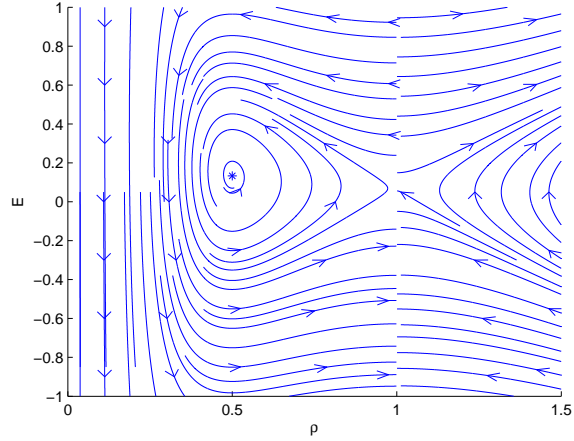}
\caption{Phase plane of $(\rho,E)$ with $\tau=15$ and $b=0.5$; $*$
is the focal point $A=(0.5,2/15)$.}
\end{center}
\label{fig5-1}
\end{figure}


Now, for the general  case of  supersonic doping profile,  we are
going to prove that, there is no interior subsonic solution, nor
transonic solution, even no interior supersonic solution if the
doping profile $b(x)\ll 1$  or $\tau\ll1$ , namely, when the
semiconductor device is almost  pure, or the relaxation time is
really small (equivalently, the semiconductor effect is large). The
supersonic solution and transonic solution  exist only when the
doping profile is close to the sonic line  and $\tau$ is large
enough. This is totally different from the previous studies
\cite{Luo-Rauch-Xie-Xin,Luo-Xin} for the case without semiconductor
effect.


Now we are going to prove each case stated in Theorem \ref{main-thm-3}.

\subsection{Non-existence of interior subsonic/supersonic/transonic solutions}

In this subsection, we are going to prove the non-existence of interior
subsonic/supersonic/transonic solutions when the doping profile is
small or the relaxation time is small.

\begin{theorem}\label{thm3.1}  No interior subsonic solution to \eqref{elliptic} exists  for the case of non-subsonic doping profile $0\leq\underline{b}\le b(x)\leq\overline{b}\leq1$.
\end{theorem}
\begin{proof}
Suppose there is an interior subsonic solution
$\rho_{sub}$ of \eqref{elliptic} defined in Definition
\ref{definition-1}, let us take the test function  by
$\varphi=(\rho-1)^2\in H_0^1(0,1)$ in \eqref{weak-solution}, then we
have
\begin{equation}
\frac{1}{2}\int_0^1\frac{\rho+1}{\rho^3}\cdot\left|\left[(\rho-1)^2\right]_x\right|^2dx
+\int_0^1\frac{\left[(\rho-1)^2\right]_x}{\tau\rho}dx
+ \int_0^1(\rho-b)(\rho-1)^2dx=0. \label{new}
\end{equation}
Noting that
$$\int_0^1\frac{\left[(\rho-1)^2\right]_x}{\tau\rho}dx=\frac{2}{\tau}\int_0^1(\rho-\ln\rho)_xdx=0,$$
and that $\rho-b> 0$ on $(0,1)$, namely,
\[
 \int_0^1(\rho-b)(\rho-1)^2dx>0,
\]
then, from \eqref{new}, we get a contradiction:
\[
\frac{1}{2}\int_0^1\frac{\rho+1}{\rho^3}\cdot\left|\left[(\rho-1)^2\right]_x\right|^2dx<0.
\]
Therefore, there is no interior subsonic solution.
\end{proof}

\begin{theorem}\label{3.2} No interior supersonic solution to \eqref{elliptic} exists, when the doping profile $b(x)$ is small such that
$\bar{b}(1+\sqrt{2\bar{b}})<1$,
 or the relaxation time $\tau$ is small such that $\tau<\frac{1}{3}$.
 \end{theorem}

\begin{proof}

Assume that $\rho(x)$ is an interior supersonic solution of
\eqref{elliptic} satisfying Definition \ref{definition-1}. The
velocity $u(x)=\dfrac{1}{\rho(x)}$ satisfies
\begin{equation} \label{4.1}
\left \{\begin{array}{ll}
        \left(u-\dfrac{1}{u}\right)u_x=E-\dfrac{u}{\tau},\\
         E_x=\dfrac{1}{u}-b(x).
        \end{array} \right.
\end{equation}
Because $u\in C[0,1]$, there exists a maximal point denote by
$\hat{y}$ such that $u(x)\leq u(\hat{y})$ for any $x\in[0,1]$. At
$\hat{y}$, the first equation of \eqref{4.1} gives
\begin{equation} \label{4.2}
E(\hat{y})=\frac{u(\hat{y})}{\tau}.
\end{equation}
Multiplying the first equation of \eqref{4.1} by $((u-1)^2)_x$,
integrating the resultant equation over $(\hat{y},1)$, using the
second equation of \eqref{4.1}, and noting
\[
u((u-1)^2)_x=\frac{1}{3}((u-1)^2(2u+1))_x,
\]
we obtain
\begin{equation}\label{3.1}
\begin{split}
&\int_{\hat{y}}^1\frac{u(x)+1}{2u(x)}|[(u(x)-1)^2]_x|^2dx
\\&= \int_{\hat{y}}^1\left(b(x)-\dfrac{1}{u(x)}\right)(u(x)-1)^2dx
-(u(\hat{y})-1)^2\left(E(\hat{y})-\frac{2u(\hat{y})+1}{3\tau}\right)\\
&= \int_{\hat{y}}^1\left(b(x)-\dfrac{1}{u(x)}\right)(u(x)-1)^2dx
-\frac{(u(\hat{y})-1)^3}{3\tau},\end{split}\end{equation} where we
have used \eqref{4.2} in the second equality.

In the case $b(x)\ll 1$, since $u(\hat{y})>1$, we get from
\eqref{3.1} that
\begin{equation}\label{3.5}
\begin{split}
&\int_{\hat{y}}^1\frac{u(x)+1}{2u(x)}|[(u(x)-1)^2]_x|^2dx \\
&\leq \int_{\hat{y}}^1\left(b(x)-\dfrac{1}{u(x)}\right)(u(x)-1)^2dx\\
&\leq \int_{\hat{y}}^1b(x)(u(x)-1)^2dx\\
&\leq\frac{1}{4}\int_{\hat{y}}^1(u(x)-1)^4dx
+ \int_{\hat{y}}^1b^2(x)dx\\
&\leq\frac{1}{4}\int_{\hat{y}}^1|[(u(x)-1)^2]_x|^2dx
+ \overline{b}^2 .\end{split}\end{equation} Here we
have used
\begin{equation}\label{3.6}
\int_{y}^1(u(x)-1)^4dx\leq\int_{y}^1|[(u(x)-1)^2]_x|^2dx\ \
\text{for } y\in(0,1).\end{equation} Then \eqref{3.5} gives
\[\int_{y}^1|[(u(x)-1)^2]_x|^2dx\leq4\bar{b}^2,\] and further
\begin{equation*}
u(x)\leq1+(2\overline{b} )^{1/2} \text{ on } [\hat{y},1].
\end{equation*}
It then follows that
\begin{equation*}
\left(b(x)-\dfrac{1}{u(x)}\right)(u(x)-1)^2\leq\left(\overline{b}-\dfrac{1}{1+(2\overline{b} )^{1/2}}\right)(u(x)-1)^2\
\ \text{for any }x\in[\hat{y},1].
\end{equation*}
Thus, when $\overline{b}$ is small enough such that
$\overline{b}-\dfrac{1}{1+(2\overline{b} )^{1/2}}<0$, we get from
the first inequality of \eqref{3.5} that
\begin{equation}\label{eqn-3.7}
\begin{split}
0\leq\int_{\hat{y}}^1\frac{u+1}{2u}|[(u-1)^2]_x|^2dx
\leq \left(\overline{b}-\dfrac{1}{1+(2\overline{b} )^{1/2}}\right)\int_{\hat{y}}^1(u-1)^2dx<0,
\end{split}\end{equation}
which is a contradiction.

In the case $\tau\ll1$, since $b\leq1$ and $1\leq u\leq u(\hat{y})$,
it follows from \eqref{3.1} that
\begin{equation*}
\int_{\hat{y}}^1\frac{u+1}{2u}|[(u-1)^2]_x|^2dx
\leq \int_{\hat{y}}^1\frac{(u-1)^3}{u}dx-\frac{(u(\hat{y})-1)^3}{3\tau}
\leq\left( -\frac{1}{3\tau}\right)(u(\hat{y})-1)^3.\end{equation*}
Thus, when $\tau$ is small such that $\tau<\frac{1}{3}$, we
get a contradiction. Therefore,  no interior supersonic solutions exist. The
proof is complete.
\end{proof}

\begin{theorem}\label{3.3}  No transonic solution  to system \eqref{1.5}-\eqref{boundary} exists, when the doping profile $b(x)$ is small such that
$\bar{b}(1+\sqrt{2\bar{b}})<1$,
 or the relaxation time $\tau$ is small such that $\tau<\frac{1}{3}$.
 \end{theorem}

\begin{proof}

Suppose that $(\rho,E)$ is a transonic solution separated by a point
$y_0\in(0,1)$ in the form
\begin{equation*}
\rho(x)=\left \{\begin{array}{ll}
        \rho_{sup}(x), \ x\in(0,y_0),\\
        \rho_{sub}(x), \ x\in(y_0,1),
        \end{array} \right.
\end{equation*}
and
\[
\rho_l\rho_r=1, \ E_l=E_r\ \text{with}\ \rho_l<1\ \text{and}\
\rho_r>1.
\]
We first claim
\begin{equation}\label{3.7}
E_l=E_r<\frac{1}{\tau}.
\end{equation}
In fact, if $E_r\geq\frac{1}{\tau}$, noting the
second equation of \eqref{1.5} gives
\[
(E_{sub})_x(x)= (\rho_{sub}-b)> (1-b)\geq0,
\]
i.e. $E_{sub}$ is monotone increasing, we have
\[
E_{sub}(x)\geq E_r\geq\frac{1}{\tau},\ \text{ and } \
\rho_{sub}(x)E_{sub}(x)-\frac{1}{\tau}>E_r-\frac{1}{\tau}\geq0,
\text{ on }(y_0,1),
\]
which in combination with the first equation of \eqref{1.5} further
gives $(\rho_{sub})_x(x)>0$ on $(y_0,1)$. Thus, $1<\rho_r<\rho$ over
$(y_0,1)$, which contradicts to $\rho_{sub}(1)=1$. Hence \eqref{3.7}
holds.

In the case $b(x)\ll 1$, multiplying the first equation of
\eqref{4.1} by $((u-1)^2)_x$ and integrating the resultant equation
over $(0,y_0)$, as in \eqref{3.1}, we get
\begin{equation*}
\begin{split}
&\int_{0}^{y_0}\frac{u(x)+1}{2u(x)}|[(u(x)-1)^2]_x|^2dx\\
&= \int_{0}^{y_0}\left(b(x)-\frac{1}{u(x)}\right)(u(x)-1)^2dx
+(u_l-1)^2\left(E_l-\frac{2u_l+1}{3\tau}\right)\\
&< \int_{0}^{y_0}\left(b(x)-\frac{1}{u(x)}\right)(u(x)-1)^2dx
-\frac{2(u_l-1)^3}{3\tau}\\
&< \int_{0}^{y_0}b(x)(u(x)-1)^2dx,\end{split}\end{equation*}
where we have used \eqref{3.7} in the first inequality. Thus, as in
\eqref{3.5}-\eqref{eqn-3.7}, when $\overline{b}$ is small enough
such that
$\overline{b}-\dfrac{1}{1+(2\overline{b} )^{1/2}}<0$, we get
the contradiction
\begin{equation*}
\int_{0}^{y_0}\frac{u+1}{2u}|[(u-1)^2]_x|^2dx <0.\end{equation*}

In the case $\tau\ll1$, since $\rho_l<1$, by \eqref{3.7} we get
$\rho_lE_l-1/\tau<0$. Thus, $\underset{x\rightarrow
y_0^-}{\lim}(\rho_{sup})_x(x)
=(1-1/\rho_l^2)^{-1}(\rho_lE_l-1/\tau)>0$. It is then easy to see
that $\rho_{sup}(x)$ attains a local minimal point on $(0,y_0)$.
Denote by $\breve{y}$ the last local minimal point of
$\rho_{sup}(x)$ on $(0,y_0)$, then $(\rho_{sup})'_x(\breve{y})=0$.
Set $u(x):=\frac{1}{\rho_{sup}(x)}$, then $u_x(\breve{y})=0$ and
$u_l=\frac{1}{\rho_l}>1$. Hence by the first equation of
\eqref{4.1}, we also get \eqref{4.2} at $\breve{y}$. Multiplying the
first equation of \eqref{4.1} by $((u-1)^2)_x$ and integrating the
resultant equation over $(\breve{y},y_0)$, as shown in \eqref{3.1},
using \eqref{4.2}, we get
\begin{equation*}
\begin{split}
&\int_{\breve{y}}^{y_0}\frac{u(x)+1}{2u(x)}|[(u(x)-1)^2]_x|^2dx
\\&= \int_{\breve{y}}^{y_0}\left(b(x)-\frac{1}{u(x)}\right)[u(x)-1]^2dx
+(u_l-1)^2\left(E_l-\frac{2u_l+1}{3\tau}\right)\\
&\ \ \ -(u(\breve{y})-1)^2\left(E(\hat{y})-\frac{2u(\breve{y})+1}{3\tau}\right)\\
&= \int_{\breve{y}}^{y_0}\left(b(x)-\frac{1}{u(x)}\right)(u(x)-1)^2dx
+(u_l-1)^2\left(E_l-\frac{1}{\tau}\right)-\frac{2(u_l-1)^3}{3\tau}
-\frac{(u(\breve{y})-1)^3}{3\tau}\\
&\leq \int_{\breve{y}}^{y_0}[u(x)-1]^3dx
-\frac{1}{3\tau}((u_l-1)^3+(u(\breve{y})-1)^3),\end{split}\end{equation*}
where we have used $b\leq1$ and \eqref{3.7} in the inequality.
Noting
\[
\max_{x\in[\breve{y},y_0]}[u(x)-1]^3=\max\{(u_l-1)^3,(u(\breve{y})-1)^3\}=:K,\]
we further have
\begin{equation*}
\int_{\breve{y}}^{y_0}\frac{u+1}{2u}|[(u-1)^2]_x|^2dx
\leq-\frac{K}{3\tau}<0 \  \text{ if }\
\tau<\frac{1}{3}.\end{equation*} We thus get a contradiction.
\end{proof}

\subsection{Existence of interior supersonic/transonic solutions}

In this subsection, we prove the existence of supersonic/transonic
solutions when the doping profile is close to the sonic line and the
relaxation time is large. The approach adopted is still the
compactness technique.

\begin{theorem}\label{thm3.4} There exists at least one interior supersonic solution  to  system \eqref{1.5}-\eqref{boundary} satisfying $\rho\in C^{\frac{1}{2}}[0,1]$ and the optimal estimate \eqref{2.2-3}, when $b(x)$ is
close to the sonic boundary $1$ and the relaxation time is large $\tau\gg1$.
\end{theorem}

\begin{proof}

The proof is similar to that of Theorem \ref{thm3}.

\emph{Step 1.} We first consider the Euler-Poisson equations without
the semiconductor effect:
\begin{equation} \label{3.8}
\left \{\begin{array}{ll}
        \left(1-\dfrac{1}{\rho^2}\right)\rho_x=\rho E,\\
         E_x=\rho-1,\\
        \rho(0)=\rho(L)=1-\delta, \ \mbox{ (supersonic boundary),}
        \end{array} \right.
\end{equation}
where $L\geq\frac{1}{4}$ is the parameter of length and $\delta>0$
is a small constant. Taking $\underline{b}=1$ in Lemma \ref{lem4},
one can see that \eqref{3.8} has a supersonic solution
$(\rho_L,E_L)(x)$ satisfying
\begin{equation}\label{3.13-2}
\beta(L)\leq\underline{\rho}\leq\gamma(L), \ E_L(0)\geq
\sqrt{f(\gamma(L)) }>0,
\end{equation}
where $\underline{\rho}:=\underset{x\in[0,L]}{\min}\rho_L(x)$.

\emph{Step 2.} Let $\eta$ be a small number to be determined such that $\delta<\eta\ll
1$. Denote by $(\rho_1,E_1)(x)$ the solution of \eqref{3.8} with
$L=\frac{1}{2}$. Now let us consider the ODE system with the
semiconductor effect $-\frac{1}{\tau}$ and a small perturbation of
the doping profile around the sonic line, i.e. $b(x)=1-\epsilon
e(x)$:
\begin{equation} \label{3.15}
\left \{\begin{array}{ll}
        \left(1-\dfrac{1}{\rho^2}\right)\rho_x=\rho E-\dfrac{1}{\tau},\\
         E_x=\rho-1+\epsilon e(x),\\
        (\rho(0),E(0))=(1-\delta,E_1(0)).
        \end{array} \right.
\end{equation}
Here $\tau\gg1$, $0<\epsilon\ll1$, $0\leq e(x)\in
L^\infty(\mathbb{R}^+)$, and we have extended periodically the
doping profile $b$ to $\mathbb{R}^+$. We claim that there exists a
number $y_1\leq C\eta$ such that $\rho(y_1)=1-\eta$, where $C>0$ is a constant independent of $\tau$, $\epsilon$, $\delta$ and $\eta$.

It is easy to see that if $\tau\geq\frac{4}{E_1(0)}$ and
$\delta\leq\frac{1}{4}$, then the initial data of \eqref{3.15}
satisfies
\[
\rho(0)E(0)-\frac{1}{\tau}=(1-\delta)E_1(0)-\frac{1}{\tau}\geq\frac{E_1(0)}{2}>0.
\]
From the first equation of \eqref{3.15}, we know that $\rho$ is
decreasing in a neighborhood of $0$. If $\rho$ keeps decreasing on
$[0,x]$, then
\begin{equation} \label{eqn-3.12}
E(x)=E_1(0)+ \int_0^{x}(\rho-1+\epsilon e(s))ds\geq
E_1(0)-x,\end{equation} which indicates that if
\begin{equation}\label{eqn-3.13}
x\leq\frac{ E_1(0)}{4},\end{equation} then
\[E(x)\geq
E_1(0)-\frac{E_1(0)}{4}=\frac{3E_1(0)}{4}.\] We next prove that if
$\rho$ keeps decreasing, denoting by $y_1$ the first number that
$\rho$ attains $1-\eta$, then $y_1\leq C\eta^2$ for some constant $C>0$. In fact, observing
that
\begin{equation*}
\rho_x=\frac{\rho
E-\frac{1}{\tau}}{1-\frac{1}{\rho^2}}=\frac{\rho^2(\rho
E-\frac{1}{\tau})}{\rho^2-1}\leq\frac{\rho^3
E}{\rho^2-1}\leq-\frac{3(1-\eta)^3E_1(0)}{4\eta(2-\eta)}\leq-\frac{E_1(0)}{16\eta},
\text{ if }\eta\leq\frac{1}{2}.
\end{equation*}
Thus,
\[y_1=\frac{\delta-\eta}{\int_0^1\rho_x(sy_1)ds}\leq\frac{16\eta^2}{E_1(0)}.\]
Hence, if $\eta\leq\frac{ E_1(0)}{8}$, then
\eqref{eqn-3.13} holds, and $\rho$ keeps decreasing and attains
$1-\eta$ at $y_1$ with $y_1\leq\frac{16\eta^2}{E_1(0)}$. By
\eqref{eqn-3.12},
\begin{equation}\label{3.25}
E_1(0)-C\eta^2\leq E(y_1)\leq E_1(0)+C\eta^2.
\end{equation}

\emph{Step 3.} Now let us reconsider the ODE system without the
semiconductor effect
\begin{equation} \label{eqn-3.15}
\left \{\begin{array}{ll}
        \left(1-\dfrac{1}{\hat{\rho}^2}\right)\hat{\rho}_x=\hat{\rho} \hat{E},\\
         \hat{E}_x=\hat{\rho}-1,\\
        (\hat{\rho}(0),\hat{E}(0))=(1-\delta,\hat{E}_0).
        \end{array} \right.
\end{equation}
Taking $b=1$ in step 2 in the proof of Theorem \ref{thm3}, we know
that there exist $\hat{E}_0\in(\frac{E_1(0)}{2},2E_1(0))$ and
$y_2\leq C\eta^2$ such that \eqref{eqn-3.15} has a supersonic
solution $(\hat{\rho},\hat{E})$ satisfying
\begin{equation}\label{eqn-3.17}
\hat{\rho}(y_2)=1-\eta, \ \hat{E}(y_2)=E(y_1), \
E_1(0)-C\eta^2\leq\hat{E}(y_2)\leq E_1(0)+C\eta^2.\end{equation}
Here $E$ and $y_1$ are given by step 2. Moreover, the length
$\hat{L}$ of the solution of \eqref{eqn-3.15} with initial boundary
data $(\hat{\rho}(0),\hat{E}(0))=(1-\delta,\hat{E}_0)$,
$\hat{\rho}(\hat{L})=1-\delta$ satisfies
\[\frac{1}{4}\leq\hat{L}\leq\frac{3}{4}.\]

\emph{Step 4.} Set
$(\bar{\rho},\bar{E})(x):=(\hat{\rho},\hat{E})(x-y_1+y_2)$, then
$(\bar{\rho},\bar{E})$ satisfies \eqref{3.8} with initial-boundary
data
\[(\bar{\rho},\bar{E})(y_1)=(1-\eta,\hat{E}(y_2))=(\rho,E)(y_1) \ \text { and } \bar{\rho}(y_3)=1-\eta \]
with $y_3:=\hat{L}+y_1-2y_2$. As in step 3 of the proof of Theorem
\ref{thm3}, when $\tau\gg1$ and $0<\epsilon\ll1$ such that
$C(\frac{1}{\tau^2}+\epsilon^2)e^{C/\eta^2}\leq1/4$, system
\eqref{3.15} has a unique solution $(\rho,E)$ on $[0,y_3]$
satisfying
\begin{equation}\label{3.29}
\rho(y_3)\leq1-\frac{\eta}{2},\ \ E(y_3)\leq E_1(0)+C\eta.
\end{equation}

Now taking $y_3$ as the initial data, as in step 3 of the proof of
Theorem \ref{thm3}, we can extend $(\rho,E)$, the solution of
\eqref{3.15}, to the state $\rho=1-\delta$. Denote by $y_4$ the
number that $\rho(y_4)=1-\delta$, then
\begin{equation}\label{3.30}
\rho(0)=\rho(y_4)=1-\delta,\ E(0)=E_1(0),\ E(y_4)\leq E_1(0)+C\eta.
\end{equation}
Moreover,
\[\frac{1}{4}-C\eta^2\leq y_4\leq\frac{3}{4}+C\eta^2.
\]
Now we take $L=\frac{3}{2}$ in \eqref{3.8} and denote by
$(\rho_2,E_2)$ its solution. Applying a similar argument above, we
know that there exists an interval $[0,y_5]$ with
\[\frac{5}{4}-C\eta^2\leq y_5\leq\frac{7}{4}+C\eta^2,
\]
such that system \eqref{3.15} has a solution on  $[0,y_5]$
satisfying
\begin{equation}\label{3.31}
\rho(0)=\rho(y_5)=1-\delta,\ E(0)=E_2(0),\ E(y_5)\leq -E_2(0)+C\eta.
\end{equation}
Without loss of generality, we assume that $E_1(0)<E_2(0)$, then
when $\eta\ll1$, for any initial data $E_L(0)\in(E_1(0),E_2(0))$,
\eqref{3.15} has a solution. Noting the length parameter $L$ is
continuous with respect to the initial data, system \eqref{3.15} has
a solution on $[0,1]$ satisfying $\rho(0)=\rho(1)=1-\delta$ and
$E(0)\in\left(E_1(0),E_2(0)\right)$.

\emph{Step 5.} For any $\delta>0$, denote by
$(\rho^\delta,E^\delta)$ the solution of \eqref{3.15} with boundary
data $\rho^\delta(0)=\rho^\delta(1)=1-\delta$. The velocity
$u^\delta=1/\rho^\delta$ satisfies
\begin{equation}\label{3.19}
\left(\left(u^\delta-\dfrac{1}{u^\delta}\right)(u^\delta)_x\right)_x+\frac{(u^\delta)_x}{\tau}
- \left(\frac{1}{u^\delta}-b\right)=0, \
u^\delta(0)=u^\delta(1)=\frac{1}{1-\delta}.
\end{equation}
Multiplying \eqref{3.19} by
$\left(u^\delta-\frac{1}{1-\delta}\right)^2$, as in
\eqref{eqn-3.13}, we have
\[
\begin{split}
&\frac{2\delta}{1-\delta}\int_0^1\frac{(u^\delta+1)}{u^\delta}\Big(u^\delta-\frac{1}{1-\delta}\Big)|u^\delta_x|^2dx
+\int_0^1\frac{(u^\delta+1)}{2u^\delta}\left|\Big(\Big(u^\delta-\frac{1}{1-\delta}\Big)^2\Big)_x\right|^2
\\&= \int_0^1\Big(b-\frac{1}{u^\delta}\Big)\Big(u^\delta-\frac{1}{1-\delta}\Big)^2\\&
\leq\frac{1}{2}\int_0^1\Big(u^\delta-\frac{1}{1-\delta}\Big)^4
+\frac{1}{2}\int_0^1\Big(b-\frac{1}{u^\delta}\Big)^2,\\
&\leq\frac{1}{4}\int_0^1\left|\Big(\Big(u^\delta-\frac{1}{1-\delta}\Big)^2\Big)_x\right|^2
+\frac{1}{2}\int_0^1b^2,
\end{split}\]
which gives
\[\Big\|\Big(u^\delta-\frac{1}{1-\delta}\Big)^2\Big\|_{H^1}\leq C,
\]
and hence
\[\|u^\delta\|_{L^\infty}\leq C.
\]
It then follows that
\[\rho^\delta=\frac{1}{u^\delta}\geq\frac{1}{\|u^\delta\|_{L^\infty}}\geq\frac{1}{C}, \text{ and } \big\|\big(1-\delta-\rho^\delta\big)^2\big\|_{H^1}\leq C.
\]
Therefore, there exists a function $\rho^0$ such that, as
$\delta\rightarrow1^+$, up to a subsequence,
\begin{equation}\label{2.22-1}
\begin{split}
&(1-\rho^\delta)^2\rightharpoonup(1-\rho^0)^2 \ \ \text{weakly in }
H^1(0,1),\\&
(1-\rho^\delta)^2\rightarrow(1-\rho^0)^2 \ \ \text{strongly
in } C^{\frac{1}{2}}[0,1].
\end{split}\end{equation} Applying the same procedure as the proof
of Theorem \ref{thm1}, one can show that $\rho^0$ is the supersonic solution of \eqref{elliptic}.
\end{proof}

\begin{theorem}\label{thm3.5} There exist infinitely many transonic
solutions to \eqref{1.5}-\eqref{boundary}, when $b(x)$
is close to the sonic boundary $1$ and $\tau\gg1$.
\end{theorem}

\begin{proof}

The proof is similar to that of Theorem \ref{thm3}.

\emph{Step 1.} Consider the ODE system \eqref{3.15}, in view of step
4 in the proof of Theorem \ref{thm3.4}, given small constants
$\eta\ll1$ ($\delta<\eta$), $\epsilon\ll1$, $\tau\gg1$, \eqref{3.15}
has a supersonic solution $(\rho,E)$ on $[0,y_4]$ satisfying
\begin{equation*}
\frac{1}{4}-C\eta^2\leq y_4\leq\frac{3}{4}+C\eta^2,\
\rho(0)=\rho(y_4)=1-\delta,\ E(0)=E_1(0), \ E(y_4)\leq-E_1(0)+C\eta,
\end{equation*}
where $E_1$ is the solution of \eqref{3.8} with $L=\frac{1}{2}$.
Setting $\rho_l=1-\eta$ and taking the jump location $\bar{y}_0\in
(0,y_4)$ by the last number when $\rho(\bar{y}_0)=\rho_l$, we focus
this supersonic solution $(\rho_{sup},E_{sup})(x)$  only on
$[0,\bar{y}_0]$. As in step 4 of the proof of Theorem \ref{thm3},
when $\eta\ll1$ such that $(C+E_1(0))\eta\leq\frac{E_1(0)}{2}$ and
$C\eta^2<\frac{E_1(0)}{4}$, then
\[\rho_rE_r-\frac{1}{\tau}\leq-\frac{E_1(0)}{4}<0.\]
From the first equation of \eqref{3.15}, we know such an initial
value problem has a decreasing subsonic solution in a neighborhood
of $\bar{y}_0^+$. We denote this subsonic solution by
$(\rho_{sub},E_{sub})(x)$. If $\rho_{sub}$ keeps decreasing, then
\begin{equation*}
\begin{split}
E_{sub}(x)&=E_r+ \int_{\bar{y}_0}^{x}(\rho_{sub}-1+\epsilon
e(x))dx\\&\leq-E_1(0)+C\eta+ \int_{\bar{y}_0}^{x}(\rho_r-1+\epsilon
e(x))dx\\&\leq-E_1(0)+C\eta+(x-\bar{y}_0)(\frac{\eta}{1-\eta}+\epsilon\|e\|_{L^\infty}),
\end{split}\end{equation*}
which implies that if $C\eta\leq\min(\frac{E_1(0)}{2},\frac{1}{2})$,
$\epsilon\leq1$ and
\begin{equation}\label{3.26}
x-\bar{y}_0\leq\frac{4
}{(1+\|e\|_{L^\infty})E_1(0)},\end{equation} then
\[E_{sub}(x)<-\frac{E_1(0)}{4}<0.\]
We now claim that if $\rho_{sub}$ keeps decreasing, denoting by
$y_6$ the number that $\rho_{sub}$ attains $1+\delta$, then
$y_6-\bar{y}_0\leq C\eta$.

In fact, observing that
\[(\rho_{sub})_x=\frac{\rho_{sub}E_{sub}-\frac{1}{\tau}}{1-\frac{1}{\rho_{sub}^2}}
\leq-\frac{(1-\eta)^2E_1(0)}{4\eta(2-\eta)}<-\frac{(1-\eta)^2E_1(0)}{4\eta},\]
we get
\[y_6-\bar{y}_0=\frac{\delta-\frac{\eta}{1-\eta}}{\int_0^1(\rho_{sub})_x(sy_6+(1-s)\bar{y}_0)ds}
\leq\frac{\eta}{1-\eta}\cdot\frac{4\eta}{E_1(0)(1-\eta)^2}\leq32\eta
\text{ if }\eta<\min(E_1(0),\frac{1}{2}).\] Obviously, if
$\eta<\frac{1}{16(1+\|e\|_{L^\infty})E_1(0)}$, then
\eqref{3.26} holds and $\rho_{sub}$ keeps decreasing and attains
$1+\delta$ at $y_6$. Now we have constructed the transonic solution
to \eqref{1.5} in $[0,y_6]$ with $\frac{1}{4}-C\eta\leq
y_6\leq\frac{3}{4}+C\eta$ as follows
\[
(\rho_{trans},E_{trans})(x)=
\begin{cases}
(\rho_{sup},E_{sup})(x), & x\in[0,\bar{y}_0),\\
(\rho_{sub},E_{sub})(x), & x\in(\bar{y}_0,y_6],
\end{cases}
\]
which satisfies the boundary condition
\[\rho_{sup}(0)=1-\delta, \ \rho_{sub}(y_6)=1+\delta,\]
and the entropy condition at $\bar{y}_0$
\[
0<\rho_{sup}(\bar{y}_0^-)=1-\eta<1<\rho_{sub}(\bar{y}_0^+),
\]
and the Rankine-Hugoniot condition \eqref{RH} at $\bar{y}_0$.

\emph{Step 2.} Denote by $(\rho_2,E_2)$ the solution of  \eqref{3.8}
with $L=\frac{3}{2}$, by step 4 of the proof of Theorem
\ref{thm3.4}, \eqref{3.15} has a supersonic solution $(\rho,E)$ on
$[0,y_7]$ with
\begin{equation*}
\frac{5}{4}-C\eta^2\leq y_7\leq\frac{7}{4}+C\eta^2,\
\rho(0)=\rho(y_7)=1-\delta,\ E(0)=E_2(0), \
E(x_{10})\leq-E_2(0)+C\eta.
\end{equation*}
As in step 1, we may construct another transonic solution for
\eqref{1.5} in the form of
\[
(\rho_{trans},E_{trans})(x)=
\begin{cases}
(\rho_{sup},E_{sup})(x), & x\in[0,\tilde{y}_0),\\
(\rho_{sub},E_{sub})(x), & x\in(\tilde{y}_0,y_7],
\end{cases}
\]
where $\tilde{y}_0\in (0,x_{10})$ and $\frac{5}{4}-C\eta^2\leq
y_7\leq\frac{7}{4}+C\eta^2$ are some determined numbers. This
transonic solution satisfies the boundary condition
\[\rho_{sup}(0)=1-\delta, \ \rho_{sub}(y_7)=1+\delta,\]
the entropy condition at $\tilde{y}_0$
\[
0<\rho_{sup}(\tilde{y}_0^-)=1-\eta<1<\rho_{sub}(\tilde{y}_0^+),
\]
and the Rankine-Hugoniot condition \eqref{RH} at $\tilde{y}_0$.

\emph{Step 3.} Without loss of generality, we assume that
$E_1(0)<E_2(0)$. As in step 6 in the proof of Theorem \ref{thm3},
for any $E_0\in(E_1(0),E_2(0))$, \eqref{3.15} has a transonic
solution on an interval $[0,y_8]$. Applying the continuation
argument in the length of the interval, one can see that for any
$\delta>0$ \eqref{1.5}-\eqref{boundary} has a transonic solution
denote by $(\rho_{trans}^\delta,E_{trans}^\delta)$ on $[0,1]$, and
it satisfies the boundary conditions
\[\rho_{sup}^\delta(0)=1-\delta, \ \rho_{sub}^\delta(1)=1+\delta,\] the entropy condition
\[
0<\rho_{sup}^\delta(y_0^\delta)=1-\eta<1<\rho_{sub}^\delta(y_0^\delta),
\]
and the Rankine-Hugoniot condition \eqref{RH} at a jump location
$y_0^\delta$ in $(0,1)$. Letting $\delta\rightarrow0^+$, applying
the diagonal argument for $(\rho_{trans}^\delta,E_{trans}^\delta)$,
we know that \eqref{1.5}-\eqref{boundary} has a transonic solution
$(\rho_{trans},E_{trans})(x)$ for $x\in [0,1]$ and satisfies the
sonic boundary condition, the entropy condition and the
Rankine-Hugoniot condition at a jump location $y_0$ in $(0,1)$.

Because $\tau$ and $\epsilon$ only depend on $(E_1(0),E_2(0),\eta)$,
and $\eta$ only depends on $(E_1(0),E_2(0))$, there exists a
$\eta_0>0$ such that for any $\eta\in(0,\eta_0)$, there exists a
transonic solution jumps at $\rho_l=1-\eta$. Thus, we obtain
infinitely many transonic solutions due to arbitrary choice of
$0<\eta<\eta_0$.
\end{proof}

\begin{proof}[Proof of Theorem \ref{main-thm-3}] Combining Theorems \ref{thm3.1}-\ref{thm3.5}, we immediately obtain Theorem
\ref{main-thm-3}.
\end{proof}

\section{Concluding remarks}

In this section, we remark that Theorems  \ref{main-thm-1}-\ref{main-thm-3} both hold for the isentropic hydrodynamic model. For isentropic flow, the pressure function satisfies $P(\rho)=T\rho^\gamma$ for some constants $T>0$ and $\gamma>1$. Then system \eqref{stationary} reduces to
\begin{equation} \label{ientro}
\left \{\begin{array}{ll}
        J = \text{constant},\\
        \left(\dfrac{J^2}{\rho}+T\rho^\gamma\right)_x=\rho E-\dfrac{J}{\tau}, \qquad x\in (0,1).\\
         E_x=\rho-b(x).
        \end{array} \right.
\end{equation}
The sonic flow means
\begin{equation*}
\mbox{fluid velocity: } u=\frac{J}{\rho} =  c=\sqrt{P'(\rho)}=\sqrt{T\gamma\rho^{\gamma-1}}: \mbox{ sound speed}. \\
\end{equation*}
Without loss of generality, we assume that
\[
J=T\gamma=1.
\]
Then \eqref{ientro} is transformed to
\begin{equation} \label{4.3}
\left \{\begin{array}{ll}
        \left(\rho^{\gamma-1}-\dfrac{1}{\rho^2}\right)\rho_x=\rho E-\dfrac{1}{\tau},\\
         E_x=\rho-b(x),
        \end{array} \right.
\end{equation}
and our sonic
boundary conditions are proposed as
\begin{equation}\label{4.4}
 \rho(0)=\rho(1)=1.
\end{equation}
Now as in the isothermal fluid,  we can also identify that, for system \eqref{4.3}, $\rho>1$ is for
the subsonic flow, $\rho=1$ stands for the sonic flow, and
$0<\rho<1$ represents for the supersonic flow.

Following the proofs of Theorems \ref{main-thm-1}-\ref{main-thm-3}, one can easily obtain the following classification of solutions to system \eqref{4.3}-\eqref{4.4} for isentropic flow:

\begin{theorem} \

\begin{enumerate}
\item \underline{For subsonic doping profile}: $b(x)\in L^\infty(0,1)$ and $b(x)>1$ in $[0,1]$,
then system \eqref{4.3}-\eqref{4.4} admit:
\begin{enumerate}
\item a unique pair of interior subsonic solutions
$(\rho_{sub},E_{sub})(x)\in C^{\frac{1}{2}}[0,1]\times H^1(0,1)$
with $\rho_{sub}(x)\ge 1$;

\item at least a pair of interior supersonic solutions
$(\rho_{sup},E_{sup})(x)\in C^{\frac{1}{2}}[0,1]\times H^1(0,1)$
with $\rho_{sup}(x)\le 1$;

\item if further that $\tau$ is large and that $\bar{b}-\underline{b}\ll1$, then
system \eqref{4.3}-\eqref{4.4} has infinitely many transonic shock
solutions $(\rho_{trans},E_{trans})(x)\in L^\infty(0,1)\times
C^0(0,1)$;

\item if further that $b(x)=b>1$ is a constant, then when $\tau$ is
small enough, \eqref{4.3}-\eqref{4.4} has infinitely many $C^1$
 transonic solution; moreover, in this case there is no
transonic shock solution.
\end{enumerate}
\item \underline{For supersonic doping profile}: $b(x)\in L^\infty(0,1)$ and $0<b(x)\leq1$ in $[0,1]$,
then \eqref{4.3}-\eqref{4.4} admit:
\begin{enumerate}
\item a pair of interior supersonic solutions $(\rho_{sup},E_{sup})(x)
\in C^{\frac{1}{2}}[0,1]\times H^1(0,1)$  and infinitely many
transonic shock solution $(\rho_{trans},E_{trans})(x)\in
L^\infty(0,1)\times C^0(0,1)$ if $b(x)$ is close to 1 and $\tau$ is
large enough;
\item no interior subsonic solutions $(\rho_{sub},E_{sub})(x)$;
\item no interior supersonic solutions $(\rho_{sup},E_{sup})(x)$,
nor transonic shock solutions $(\rho_{trans},$ $E_{trans})(x)$ if
$b(x)$ is small
or $\tau$ is small.
\end{enumerate}

\end{enumerate}
\end{theorem}

\section*{Acknowledgements}
J. Li's research was partially supported by NSF of China under the
grant 11571066. M. Mei's research was partially supported by NSERC
grant RGPIN 354724-2016 and FRQNT grant 192571.  K. Zhang's work was
partially supported by NSF of China under the grant 11371082 and the
Fundamental Research Funds for the Central Universities under the
grant 111065201.

\bibliographystyle{amsplain}

\end{document}